\documentclass[moor]{informs1}              

\usepackage{amsmath,amsfonts,dsfont,amstext,amssymb,stmaryrd,graphicx,algorithm,algorithmic,setspace,url,graphicx,fancyhdr,cite,mathrsfs,ifthen,theorem}
\usepackage[table]{xcolor}
\usepackage{subfigure}
\usepackage{ulem} 
\usepackage{multirow}
\usepackage{rotating}
\usepackage{adjustbox}

\DeclareMathOperator{\R}{\mathbb{R}}
\DeclareMathOperator{\N}{\mathbb{N}}

\DeclareMathOperator{\E}{\mathbb{E}}

\DeclareMathOperator{\C}{\mathcal{C}}

\DeclareMathOperator{\NN}{\mathcal{N}}
\DeclareMathOperator{\PP}{\mathcal{P}}
\DeclareMathOperator{\QQ}{\mathcal{Q}}
\DeclareMathOperator{\RR}{\mathcal{R}}
\DeclareMathOperator{\PPH}{\widehat{\mathcal{P}}}
\DeclareMathOperator{\QQH}{\widehat{\mathcal{Q}}}
\DeclareMathOperator{\IPS}{\mathcal{I}^{+}(S)}
\DeclareMathOperator{\IMS}{\mathcal{I}^{-}(S)}
\DeclareMathOperator{\JPS}{\mathcal{J}^{+}(S)}
\DeclareMathOperator{\JMS}{\mathcal{J}^{-}(S)}
\DeclareMathOperator{\KPS}{\mathcal{K}^{+}(S)}
\DeclareMathOperator{\KMS}{\mathcal{K}^{-}(S)}

\DeclareMathOperator{\IPK}{\mathcal{I}^{+}(K)}
\DeclareMathOperator{\IMK}{\mathcal{I}^{-}(K)}
\DeclareMathOperator{\JPK}{\mathcal{J}^{+}(K)}
\DeclareMathOperator{\JMK}{\mathcal{J}^{-}(K)}
\DeclareMathOperator{\JHPK}{\widehat{\mathcal{J}}^{+}(K)}
\DeclareMathOperator{\JHMK}{\widehat{\mathcal{J}}^{-}(K)}
\DeclareMathOperator{\KHPK}{\widehat{\mathcal{K}}^{+}(K)}
\DeclareMathOperator{\KHMK}{\widehat{\mathcal{K}}^{-}(K)}
\DeclareMathOperator{\LPK}{\mathcal{L}^{+}(K)}
\DeclareMathOperator{\LMK}{\mathcal{L}^{-}(K)}

\DeclareMathOperator{\tran}{^{\top}}
\DeclareMathOperator{\ri}{ri}
\DeclareMathOperator{\dom}{\mathop{\rm dom}}

\DeclareMathOperator{\DC-SOS}{\mathcal{DC-SOS}}

\newenvironment{proc}[2]
{\begin{algorithm}[#1]\floatname{algorithm}{#2:}}
	{\end{algorithm}\addtocounter{algorithm}{-1}}

\newcommand{\orcid}[1]{\href{https://orcid.org/#1}{\includegraphics[scale=1]{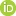}}}

\usepackage[breakable,fitting]{tcolorbox}
\newtcolorbox{mybox}{breakable,coltitle=black,colbacktitle=gray!20,colframe=black,colback=white,coltext=black}
{}

\usepackage{natbib}
 \NatBibNumeric
 \def\bibfont{\small}%
 \bibpunct[, ]{[}{]}{,}{n}{}{,}%

\usepackage[colorlinks=true,breaklinks=true,bookmarks=true,urlcolor=blue,
     citecolor=blue,linkcolor=blue,bookmarksopen=false,draft=false]{hyperref}

\def\EMAIL#1{\href{mailto:#1}{#1}}

\TheoremsNumberedThrough     

\EquationsNumberedThrough    

\begin{document}


\RUNAUTHOR{Y.S. Niu et al.}

\RUNTITLE{MVSK portfolio optimization based on accelerated DCA and SOS}

\TITLE{High-order Moment Portfolio Optimization via An Accelerated Difference-of-Convex Programming Approach and Sums-of-Squares}

\ARTICLEAUTHORS{%
\AUTHOR{Yi-Shuai Niu}
\AFF{Department of Applied Mathematics, The Hong Kong Polytechnic University, Hong Kong, \EMAIL{yi-shuai.niu@polyu.edu.hk}}
\AUTHOR{Ya-Juan Wang}
\AFF{School of Management, Fudan University, China,
\EMAIL{wangyajuan@fudan.edu.cn}}
\AUTHOR{Hoai An Le Thi}
\AFF{
	LGIPM, University of Lorraine, France; and
	Institut Universitaire de France (IUF), \EMAIL{hoai-an.le-thi@univ-lorraine.fr}}
\AUTHOR{Dinh Tao Pham}
\AFF{Laboratory of Mathematics, National Institute of Applied Sciences of Rouen, France, \EMAIL{pham@insa-rouen.fr}}
} 

\ABSTRACT{%
The Mean-Variance-Skewness-Kurtosis (MVSK) portfolio optimization model is a quartic nonconvex polynomial minimization problem over a polytope, which can be formulated as a Difference-of-Convex (DC) program. In this manuscript, we investigate four DC programming approaches for solving the MVSK model. First, two DC formulations based on the projective DC decomposition and the Difference-of-Convex-Sums-of-Squares (DC-SOS) decomposition are established, where the second one is novel. Then, DCA is applied to solve these DC formulations. The convergence analysis of DCA for the MVSK model is established. Second, we propose an accelerated DCA (Boosted-DCA) for solving a general convex constrained DC program involving both smooth and nonsmooth functions. The acceleration is realized by an inexact line search of the Armijo-type along the DC descent direction generated by two consecutive iterates of DCA. The convergence analysis of the Boosted-DCA is established. Numerical simulations of the proposed four DC algorithms on both synthetic and real portfolio datasets are reported. Comparisons with KNITRO, FILTERSD, IPOPT and MATLAB fmincon optimization solvers demonstrate good performance of our methods. Particularly, two DC algorithms with DC-SOS decomposition require less number of iterations, which demonstrates that DC-SOS decomposition can provide better convex over-approximations for polynomials. Moreover, the accelerated versions indeed reduce the number of iterations and achieve the best numerical results.}


\MSCCLASS{Primary: 91G10, 90C06; secondary: 90C29, 90C30, 90C90}
\ORMSCLASS{Primary: ; secondary: }
\HISTORY{}

\maketitle

\section{Introduction}
The concepts of portfolio optimization and diversification are fundamental to understand financial market and financial decision making. The major breakthrough came in \cite{markowitz1952} with the introduction of the mean-variance portfolio selection model (MV model) developed by Harry Markowitz (Nobel Laureate in Economics in 1990). This model provided an answer to the fundamental question: How should an investor allocate funds among the possible investment choices? Markowitz firstly quantified return and risk of a security, using the statistical measures of its expected return and variance. Then, he suggested that investors should consider return and risk together, and determine the allocation of funds based on their return-risk trade-off. Before Markowitz's seminal article, the finance literature had treated the interplay between return and risk in an ad hoc fashion. Based on MV model, the investors are going to find among the infinite number of portfolios that achieve a particular return objective with the smallest variance. The portfolio theory had a major impact on the academic research and financial industry, often referred to as ``the first revolution on Wall Street". More discussions about MV model can be found in the review article \cite{steinbach2001}.

For a long time, there is a confusion that the application of the MV model requires Gaussian return distribution. This is not true! For this issue, Markovitz has declared in \cite{markowitz2014} that ``the persistence of the Great Confusion - that MV analysis is applicable in practice only when return distributions are Gaussian or utility functions quadratic - is as if geography textbooks of 1550 still described the Earth as flat."
In fact, the normality of asset returns is not necessary in MV model, and has been widely rejected in empirical tests. Many return distributions in real market exhibit fat tails and asymmetry which will significantly affect portfolio choices and asset pricing \cite{arditti1975,jondeau2003}. E.g., \cite{harvey2000} showed that in the presence of positive skewness, investors may be willing to accept a negative expected return. There are some rich literature that attempted to model higher-order moments in the pricing of derivative securities, starting from the classic models of \cite{merton1976} (jump-diffusions) and \cite{heston1993} (stochastic volatility), see \cite{bhandari2009} for more related works. Therefore, many scholars suggested introducing high-order moments such as skewness (3rd-order moment) and kurtosis (4th-order moment) into portfolio optimization model. 

The first work attempted to extend the MV model to higher-order moments was proposed in \cite{jean1971}. Some noteworthy works such as \cite{arditti1975} and \cite{levy1979} were mainly focused on the mean-variance-skewness model (MVS model). Later, more extensions of high-order moment portfolio models adapted kurtosis were investigated by several authors (e.g., \cite{de2003incorporating}, \cite{de2004}, \cite{maringer2009}, and \cite{harvey2010} etc). From a mathematical point of view, a higher-order moment portfolio model can be viewed as an approximation of the general expected utility function, in which people consider the Taylor series expansion of the utility function and drop the higher-order terms from the expansion. Therefore, the classical MV model is in fact a rough approximation of the general utility function, and the higher-order moment model will be more accurate. The reader is referred to the excellent survey on the 60 years' development in portfolio optimization \cite{kolm2014} for more information about different portfolio selection models. 

Despite the advantages of the higher-order moment portfolio models, in practice, these models are however seldom used. There are some reasons, typically, practitioners rely upon a utility function based on mean-variance approximation, which is trusted to perform well enough \cite{levy1979}. Moreover, due to the limitations of computing power in the 20th century, constructing and solving a higher-order moment portfolio model is very difficult, e.g., a model with quartic polynomial approximation and several hundreds of assets is already intractable. Fortunately, with the rapid development of CPU and GPU hardware in the early of 21th century, as well as the adequate computer memory, the computing power available today can handle some higher-order moment portfolios (at least portfolios with moderate size). On the other hand, recently, we developed in \cite{niu2018} the Difference-of-Convex-Sums-of-Squares (namely, DC-SOS) decomposition technique for general multivariate polynomials, which provides us a robust tool to formulate any polynomial optimization to a DC program. Motivated by the above reasons, it is the right time to attack the higher-order moment portfolio optimization problems. 

In this paper, we will focus on a high-order moment portfolio model which takes \emph{mean}, \emph{variance}, \emph{skewness}, and \emph{kurtosis} into consideration, namely the MVSK model. It consists of maximizing the mean and skewness of the portfolio while minimizing the variance and kurtosis, which is in fact a multi-objective polynomial optimization problem. This problem can be further formulated as a weighting single objective quartic polynomial optimization problem with positive weights, called linear weighting method (see e.g., \cite{niu2011}). It is known that any optimal solution of the linear weighting formulation is also a Pareto optimal solution of the multi-objective formulation, but the reverse is not true. The linear weighting quartic polynomial formulation of the MVSK model is in general nonconvex and NP-hard. Thus, we cannot expect a polynomial time global optimization algorithm to solve this problem. The existing methods in the literatures aim at efficiently finding local optimal solutions for problems with moderate size such as: stochastic algorithms (Differential Evolution and Stochastic Differential Equation) in \cite{maringer2009}; DC programming approach based on a commonly used DC decomposition and DCA in \cite{niu2011}; classical nonlinear optimization approaches, e.g., sequential quadratic programming method, trust-region method, Lasserre's hierarchy, and Branch-and-Bound in \cite{niu2011}; machine learning approaches based on regularization and cross-validation in \cite{ban2016}; and multi-linear formulation based on non-negative symmetric tensors in \cite{chen2017}.

The major contributions in our paper include: (i) Develop a DC-SOS decomposition for the MVSK model; (ii) Design a Boosted-DCA for convex constrained DC program involving both smooth and non-smooth functions; (iii) Establish the convergence analysis of DCA for the MVSK model and Boosted-DCA for the convex constrained DC program. 

More specifically, four DCA-based algorithms (with and without acceleration) for solving the MVSK portfolio optimization model are proposed. Firstly, two DC decompositions are considered: the first one is based on a commonly used DC decomposition (namely, the projective DC decomposition) discussed in \cite{niu2011}; the second one is based on a DC-SOS decomposition which is new and more suitable for polynomial functions. Secondly, we apply DCA with the two DC decompositions. DCA with the first DC decomposition requires solving a convex quadratic minimization problem over the standard simplex in each iteration, which is equivalent to find a projection of a vector on the standard simplex, and can be solved very effectively using some explicit algorithms, e.g., BPPPA proposed in \cite{judice1992}, and direct projection methods proposed in \cite{minoux1984,held1974}; DCA with the second DC decomposition requires solving a convex quartic polynomial minimization problem over the standard simplex in each iteration, which can be effectively solved by using first- and second-order convex optimization approaches, e.g., the gradient method with and without acceleration (e.g., Nesterov's acceleration), the interior point method and the Newton method etc. Some efficient solvers are available such as IPOPT, KNITRO, FILTERSD, CVX and MATLAB fmincon. Note that the DC-SOS decomposition leads to a 4th degree convex overestimator, which should be a better convex approximation of the 4th degree nonconvex polynomial objective function than the convex quadratic overestimator used in the common DC decomposition. The superiority of the DC-SOS decomposition is also verified in the numerical simulations. The convergence analysis of DCA for the MVSK model is established which completes the work in \cite{niu2011}. Particularly, the global convergence of the sequence $\{x^k\}$ generated by DCA is established based on the well-known {\L}ojasiewicz subgradient inequality. Thirdly, we established an acceleration technique for general DC program over closed convex set. Different to the one proposed in \cite{artacho2018} which only focused on unconstrained smooth DC program, our Boosted-DCA is designed for general convex constrained DC program with either smooth or nonsmooth functions. The convergence analysis for the Boosted-DCA is established. Lastly, we implemented our four algorithms as a software package on MATLAB, namely \texttt{MVSKOPT}, and tested on both synthetic and real datasets of the MVSK model. Numerical comparisons with KNITRO, FILTERSD, IPOPT and MATLAB fmincon are reported, which demonstrate good performances of DC approaches and the nice numerical effect of the Boosted-DCA.

The paper is organized as follows: Section \ref{sec:mvskmodel} presents the MVSK model. After a brief introduction about DC program and DCA in Section \ref{sec:dcp&dca}, we focus on the presentation of two DC formulations (commonly used DC decompositions and DC-SOS decomposition), and establishing the convergence analysis of DCA for the MVSK model in Section \ref{sec:twoDCformulation&DCA}. The DC acceleration techniques are discussed in Section \ref{sec:boosteddca}, where the definition and properties of DC descent direction are introduced in Subsection \ref{subsec:dcdd}; the Armijo-type line search is discussed in Subsection \ref{subsec:Armijo}; the Boosted-DCA and its convergence are established in Subsection \ref{subsec:Boosted-DCA}. Applying Boosted-DCA to both the projective DC decomposition and the DC-SOS decomposition for the MVSK model is given in Section \ref{sec:UDCA&UBDCA}. Numerical simulations comparing four DCA-based algorithms and several classical nonlinear optimization solvers on both synthetic and real datasets are reported in Section \ref{sec:Simu}. Concluding remarks and some further topics are discussed in the last section.

\section{High-order Moment Portfolio Optimization Model}
\label{sec:mvskmodel} 
Consider a portfolio with $n$ assets. In this section, we investigate a high-order moment portfolio optimization model consists of the first 4th-order moments (Mean-Variance-Skewness-Kurtosis, namely MVSK). The portfolio model involving moments beyond 4th-order can be defined in a similar way. 
\subsection{Portfolio Inputs}
The inputs of the MVSK model consist of the first four order moments and co-moments of the portfolio returns which are estimated by \emph{sample moments and co-moments} defined as follows: Let $\E$ denote the expectation operator; let $n$ be the number of assets and $T$ be the number of periods; let $R_{i,t}$ be the return rate of the asset $i\in\{1,\ldots,n\}$ in the period $t\in \{1,\ldots,T\}$. The return rate of the asset $i$ is denoted by $R_i$, and $R=(R_{i})\in \R^n$ stands for the return rate vector. We have

\begin{enumerate}
	\item \textbf{Mean} (1st-order moment): denoted by $\mu = (\mu_i) \in \R^n$ whose $i$-th element $\mu_i$ is defined by
	\begin{equation}
	\label{eq:mu}
	\mu_i := \E(R_i) \approx \frac{1}{T}\sum_{t=1}^{T}R_{i,t}.
	\end{equation}
	\item \textbf{Variance} and \textbf{Covariance} (second central moment and co-moment): denoted by $\Sigma = (\sigma_{i,j})\in \R^{n^2}$ where $\sigma_{i,j}$ is defined by
	\begin{equation}\label{eq:Sigma}
	\sigma_{ij}:=\E[(R_{i}-\mu_{i})(R_{j}-\mu_{j})] \approx \frac{1}{T-1}\sum_{t=1}^{T}(R_{i,t}-\mu_i)(R_{j,t}-\mu_j).
	\end{equation}
	\item \textbf{Skewness} and \textbf{Co-skewness} (third central moment and co-moment): denoted by $S = (S_{i,j,k}) \in \R^{n^3}$ where $S_{i,j,k}$ is defined by
	\begin{align}
	\label{eq:S}
	S_{i,j,k} := &\E[(R_{i}-\mu_{i})(R_{j}-\mu_{j})(R_{k}-\mu_{k})] \approx \frac{1}{T}\sum_{t=1}^{T}(R_{it}-\mu_{i})(R_{jt}-\mu_{j})(R_{kt}-\mu_{k}).
	\end{align}
	\item \textbf{Kurtosis} and \textbf{Co-kurtosis} (fourth central moment and co-moment): denoted by $K = (K_{i,j,k,l})\in \R^{n^4}$ where $K_{i,j,k,l}$ is defined by
	\begin{align}\label{eq:K}
	K_{i,j,k,l}:= & \E[(R_{i}-\mu_{i})(R_{j}-\mu_{j})(R_{k}-\mu_{k})(R_{l}-\mu_{l})] \approx \frac{1}{T}\sum_{t=1}^{T}(R_{it}-\mu_{i})(R_{jt}-\mu_{j})(R_{kt}-\mu_{k})(R_{lt}-\mu_{l}).
	\end{align}
\end{enumerate}

These inputs can be written as tensors and easily computed from data using the formulas \eqref{eq:mu}, \eqref{eq:Sigma}, \eqref{eq:S} and \eqref{eq:K}. Note that these tensors have perfect symmetry, e.g., $\Sigma$ is a real symmetric positive semi-definite matrix, and the values of $S_{i,j,k}$ (resp. $K_{i,j,k,l}$) with all permutations of the index $(i,j,k)$ (resp. $(i,j,k,l)$) are equals. Therefore, we only need to compute $\binom{n+1}{2}$, $\binom{n+2}{3}$ and $\binom{n+3}{4}$ independent elements respectively. 

When dealing with these high-order moments and co-moments, it is convenient to ``slice" these tensors and create a big matrix from the slices. In our previous work \cite{niu2011}, we have discussed using Kronecker product $\otimes$ to rewrite co-skewness (resp. co-kurtosis) tensor to $n\times n^2$ (resp. $n\times n^3$) matrix by the formulations:
\[{\hat{S}} =
\E[(R-\mu)(R-\mu)\tran\otimes
(R-\mu)\tran];~\hat{K} = \E[(R-\mu)(R-\mu)\tran \otimes (R-\mu)\tran \otimes (R-\mu)\tran].\]
Then converting $\hat{S}$ and $\hat{K}$ into sparse matrices by keeping only the independent elements based on symmetry. This computing technique is very useful when dealing with large-scale cases. 

\subsection{Mean-Variance-Skewness-Kurtosis Portfolio Model}
Let us denote the decision variable of the portfolio (called \emph{portfolio weights}) as $x\in \R^n$. We assume that \emph{no short sales or leverage are allowed}, i.e., $x\geq 0$ and sums up to one, thus $x$ is restricted in the standard $(n-1)$-simplex $\Omega:= \{ x\in \R^n_+: e\tran x = 1\}$ where $e$ denotes the vector of ones. The first four order portfolio moments are functions of the portfolio decision variable $x$ defined as follows:
\begin{enumerate}
	\item Mean (1st-order portfolio moment): $m_1(x)= \mu\tran x.$
	\item Variance (2nd-order portfolio moment): $m_2(x)= x\tran \Sigma x.$
	\item Skewness (3rd-order portfolio moment): $m_3(x)= x^{\top} \hat{S} (x \otimes x)=\sum_{i,j,k=1}^{n}S_{i,j,k}~x_{i}x_{j}x_{k}.$
	\item Kurtosis (4th-order portfolio moment): $m_4(x)= x^{\top} \hat{K} (x \otimes x \otimes x)= \sum_{i,j,k,l=1}^{n}K_{i,j,k,l}~x_{i}x_{j}x_{k}x_{l}.$
\end{enumerate}

A rational investor's preference is the highest odd moments, as this would decrease extreme values on the side of losses and increase them on the side of gains. As far as even moments, the wider the tails of the returns distribution, the higher the even moments will be. Therefore, the investor prefers low even moments which implies decreased dispersion of the payoffs and less uncertainty of returns \cite{Scott1980}. Based on these observations, the MVSK portfolio optimization model consists of maximizing the expected return and skewness while minimizing the variance and kurtosis \cite{fabozzi2006,parpas2006}.

Let us denote $F:\R^n \to \R^4$ defined by:
$$F(x) := (-m_1(x), m_2(x), -m_3(x),m_4(x))\tran.$$
The MVSK model is described as a multi-objective optimization problem as:
$$\min \{F(x): x\in \Omega \},$$
which can be further investigated as a weighted single-objective optimization:
\begin{equation}\label{P_mvsk}
\min \{f(x):= c\tran F(x) : x\in \Omega\}, \tag{MVSK}
\end{equation}
where the parameter $c$ denotes the investor's preference verifying $c\geq 0$. For example, the risk-seeking investor will have more weights on $c_1$ and $c_3$, while the risk-aversing investor will have more weights on $c_2$ and $c_4$. It is well-known that any optimal solution of the weighted single-objective MVSK model is also an optimal solution of the multi-objective MVSK model (not conversely). We are interested in developing DC programming approaches for solving the weighted single-objective MVSK model. In next section, we will briefly outline some preliminaries about DC program and DCA algorithm. 

\section{Preliminaries on DC program and DCA}\label{sec:dcp&dca}
\subsection{DC program}
Let us denote $\Gamma_0(\R^n)$, the set of closed proper convex functions from $\R^n$ to $(-\infty, \infty]$ under the convention that $(\infty) - (\infty) = \infty$. The \emph{standard DC program} is defined by
\begin{equation}\label{prob:pdc}
\alpha = \inf \{f(x):=g(x)-h(x): x\in \R^n \},
\end{equation}
where $g$ and $h$ are both $\Gamma_0(\R^n)$ functions, and $\alpha$ is assumed to be finite, so that $\emptyset \neq \dom g \subset \dom h$. Let $\mathcal{C}\subset \R^n$ be a nonempty closed convex set, the convex constrained DC program is given by $$\inf \{g(x)-h(x) : x\in \mathcal{C} \},$$
which can be standardized as a standard DC program as 
$$\inf \{ (g+\chi_{\mathcal{C}})(x)-h(x): x\in \R^n\}$$
by introducing the indicator function of $\mathcal{C}$:
$$\chi_{\mathcal{C}}(x) = \begin{cases}
0, & \text{if } x\in \mathcal{C},\\
\infty, & \text{otherwise.}
\end{cases}$$
Clearly, $\chi_C\in \Gamma_0(\R^n)$ and both $g+\chi_{\mathcal{C}}$ and $h$ belong to $\Gamma_0(\R^n)$.

For any convex function $h\in \Gamma_0(\R^n)$ and a point $x^*\in \R^n$, we denote $\partial h(x^*)$ as the subdifferential of $h$ at $x^*$, defined by, see e.g. \cite{rockafellar1970,beck2017first}, 
$$\partial h(x^*): = \{ y\in \R^n: h(x) \geq h(x^*) + \langle x-x^*, y\rangle,\forall x\in \R^n \}.$$ 
If $x^*\notin \dom h$, then $\partial h(x^*)=\emptyset$. Note that $\partial h(x^*)$ is a closed convex set. If $x^*\in \text{int}(\dom h)$, then $\partial h(x^*)$ is nonempty and bounded. If $x^*\in\ri(\dom h)$, then $\partial h(x^*)$ is nonempty but it could be still unbounded, for example, when $\dim (\dom f)<n$. The subdifferential generalizes the derivative in the sense that $\partial h(x^*)$ reduces to the singleton $\{\nabla h(x^*)\}$ if $h$ is differentiable at $x^*$. 

For nonconvex nonsmooth settings, there are several important subdifferentials such as
\begin{definition}[Fr\'echet subdifferential]
	\label{def:F-subdiff}
	The F(r\'echet)-subdifferential of a proper closed function $f:\R^n\to (-\infty,\infty]$ at $x\in \dom f$ is defined by
	\begin{equation*}\label{eq:F-subdiff}
		\partial^F f(x):= \left \{  y\in\R^n ~|~ \liminf\limits_{z\neq x\atop z\to x } \frac{f(z)-f(x)-\langle y,z-x\rangle}{\|z-x\|} \geq 0\right\}.
	\end{equation*}
	If $x\notin \dom f$, then $\partial^F f(x)= \emptyset$. 
\end{definition}
\begin{definition}[Limiting subdifferential]
	\label{def:l-subdiff}
	The l(imiting)-subdifferential of a proper closed function $f:\R^n\to (-\infty,\infty]$ at $x\in \R^n$ is defined by
	\begin{equation*}\label{eq:l-subdiff}
		\partial^L f(x) :=\{y\in\R^n ~|~ \exists (x^k\to x, f(x^k)\to f(x), y^k \in \partial^F f(x^k)) \text{ such that } y^k\to y\}.
	\end{equation*}
\end{definition}
It is known that $\partial^F f(x)\subset \partial^L f(x)$, and $\partial^F f(x)$ is a closed convex set, while $\partial^L f(x)$ is closed. Both $\dom \partial^F f$ and $\dom \partial^L f$ are dense in $\dom f$ \cite{bolte2007lojasiewicz}. If the function $f$ is of class $C^1$, then the Fr\'echet-subdifferential and the limiting subdifferential coincide with the gradient, i.e., $\partial^F f(x) = \partial^L f(x)=\{\nabla f(x)\}.$ If $f$ is convex and $x\in \text{ri}(\dom f)$, then 
$\partial f(x) = \partial^F f(x) = \partial^L f(x).$

A point $x^*\in \R^n$ is called a \emph{limiting-stationary point} (or (generalized) critical point) of $f$ if $0\in \partial^L f(x^*)$. In DC optimization, $x^*$ is called a \textit{DC critical point} of the standard DC program if $\partial g(x^*)\cap \partial h(x^*)\neq \emptyset$ (or equivalently $0\in \partial g(x^*)-\partial h(x^*)$). If $g$ and $h$ are non-differentiable at $x^*$, then $\partial^L f(x^*)\subset \partial g(x^*)-\partial h(x^*)$. Especially, if $g$ (resp. $h$) is differentiable at $x^*$, then $\partial^L f(x^*)=\nabla g(x^*)-\partial h(x^*)$  (resp. $\partial^L f(x^*)=\partial g(x^*)- \nabla h(x^*)$). Particularly, if both $g$ and $h$ are differentiable, then the DC critical point $x^*$ reduces to the classical stationary point verifying the Fermat's condition for the unconstrained optimization problem \eqref{prob:pdc} as $\nabla f(x^*) = \nabla g(x^*)-\nabla h(x^*) = 0.$ 

Note that a DC critical point may not be a local minimizer except for some particular cases (e.g., when $f$ is locally convex at critical points). A stronger definition than the DC critical point, namely \emph{strongly DC critical point}, which is a DC critical point $x^*$ verifying $\emptyset \neq \partial h(x^*) \subset \partial g(x^*)$, i.e., $f'(x^*;x-x^*) \geq 0$ for all $x\in \dom g$. The notation $f'(x;d)$ stands for the directional derivative at $x$ in the direction $d$ defined by:
$$f'(x;d) = \lim\limits_{t\downarrow 0} \frac{f(x+td)-f(x)}{t}.$$
The strong DC criticality often coincides with the commonly used \emph{directional stationarity} (cf. d-stationarity), the d-stationarity depends on $f$ but the strong DC criticality depends on $g$ and $h$. A strongly DC critical point may not be a local minimizer. For instance, let $f(x)= 1 - \|x\|^2$ with a DC decomposition $g(x)= 1$ and $h(x)= \|x\|^2$, then the point $x^*=0$ verifies $\nabla g(x^*)-\nabla h(x^*)=0$, which is indeed a strongly DC critical point of $f$ but not a local minimizer of $f$. More discussions about the d-stationarity and the strong DC criticality are referred to \cite{Lethi2018Convergence}.

\subsection{DCA}
An efficient DC Algorithm for solving the standard DC program, namely DCA, was first introduced by Pham Dinh Tao in 1985 as an extension of the subgradient method, and has been extensively developed by Le Thi Hoai An and Pham Dinh Tao since 1994. The readers are referred to  \cite{pham1997,pham1998,le2003,pham2005,pham2016,le2018} and the references therein. 

DCA consists of solving the standard DC program by a sequence of convex subproblems
\begin{equation}
x^{k+1}\in \argmin \{ g(x) - \langle x, y^k \rangle : x\in \R^n \},
\end{equation}
where $y^k\in \partial h(x^k)$. This convex subproblem is in fact derived from convex overestimation of the DC function $f$ at the current iteration point $x^k$, denoted $f^k$, which is constructed by linearizing $h$ at $x^k$ as for all $x\in \R^n$,
$$
f(x) = g(x)-h(x)\leq g(x) - (h(x^k)+\langle x-x^k, y^k \rangle)= f^k(x),
$$ 
where $y^k\in \partial h(x^k)$. 

DCA applied to convex constrained DC program yields a similar scheme as:
$$x^{k+1}\in \argmin \{ g(x) - \langle x, y^k \rangle : x\in \mathcal{C}\}$$
with $y^k\in \partial h(x^k)$. 

For stopping criteria of DCA, we often use one of the following conditions:
\begin{itemize}
	\item[$\bullet$] $\Delta f = |f(x^{k+1})-f(x^{k})|/(1+|f(x^{k+1})|)\leq \varepsilon_1$.
	\item[$\bullet$] $\Delta x = \|x^{k+1}-x^{k}\|/(1+\|x^{k+1}\|)\leq \varepsilon_2$.
\end{itemize} 
Note that $\|x^{k+1}-x^{k}\|/(1+\|x^{k+1}\|)$ (resp. $|f(x^{k+1})-f(x^{k})|/(1+|f(x^{k+1})|)$) provides a compromised error between the absolute error and the relative error. If $\|x^k\|$ (resp. $f(x^k)$) is close to $0$, then it is an approximation of the absolute error, otherwise, it is an approximation of the relative error. 

\begin{theorem}[Convergence theorem of DCA, see e.g., \cite{pham1997}]
DCA applied to the standard DC program starting with an initial point $x^0\in \dom \partial h$ generates a sequence $\{x^k\}$ such that
\begin{itemize}
	\item[$\bullet$] The sequence $\{f(x^k)\}$ is decreasing and bounded from below, thus convergent.
	\item[$\bullet$] Every limit point of the sequence $\{x^k\}$ is a DC critical point of the standard DC program.
\end{itemize}
\end{theorem}

For DC program with continuously differentiable $h$, we can prove in the next theorem that any DC critical point is strongly DC critical, whose proof is given in Appendix \ref{appendix:A_thm:convtod-stationaryforsmoothdcp}. 
\begin{theorem}\label{thm:convtod-stationaryforsmoothdcp}
	For convex constrained DC program over a nonempty closed convex set $\mathcal{C}\subset \R^n$, if $h$ is continuously differentiable, then DCA starting from an initial point $x^0\in \R^n$ will generate a sequence $\{x^k\}$ such that every limit point is a strongly DC critical point.
\end{theorem}

In general, we only have the subsequential convergence of the sequence $\{x^k\}$ generated by DCA. Its global convergence need more assumptions. We recall the well-known \emph{{\L}ojasiewicz subgradient inequality} established by Bolte-Daniilidis-Lewis \cite[Theorem 3.1]{bolte2007lojasiewicz}, which is important to the global convergence analysis for DCA and for our proposed algorithms in this paper.
\begin{theorem}[{\L}ojasiewicz subgradient inequality, see Theorem 3.1, \cite{bolte2007lojasiewicz}]\label{thm:Loja-ineq}
	Let $f:\R^n\to (-\infty,\infty]$ be a subanalytic function with closed domain, and assume that $f$ is continuous over its domain. Let $x^*\in \R^n$ be a limiting-stationary point of $f$. Then there exists a {\L}ojasiewicz exponent $\theta\in [0,1)$, a finite constant $L>0$, and a neighbourhood $\mathcal{V}$ of $x^*$ such that 
	\begin{equation}
		\label{eq:Loja-ineq}
		|f(x)-f(x^*)|^{\theta}\leq L \|y\|, \forall x\in \mathcal{V}, y\in \partial^L f(x),
	\end{equation}
where the convention $0^0=1$ is adopted.
\end{theorem} 

Note that the notions and properties of the subanalytic function/set are classical and discussed in \cite{lojasiewicz1965ensembles,lojasiewicz1993geometrie,bierstone1988semianalytic,bolte2007lojasiewicz}. 
The class of subanalytic sets (resp. functions) contains all analytic sets (resp. functions). As a matter of fact, any polynomial function is analytic and thus subanalytic as well. The polyhedral set of $\R^n$ is a subanalytic set, and its indicator function is a subanalytic function. Subanalytic sets and subanalytic functions enjoy interesting properties. For instance, the class of subanalytic sets is closed under locally finite unions/intersections, relative complements, and the usual projection. The distance function to a subanalytic set is subanalytic; the sum/difference of continuous and subanalytic functions is also subanalytic.

\section{DC formulations and DCA for \eqref{P_mvsk}}\label{sec:twoDCformulation&DCA}
The \eqref{P_mvsk} model, as a nonconvex quartic polynomial optimization problem, can be no doubt formulated as a DC programming problem, since any polynomial function as a $\mathcal{C}^{\infty}$ function is indeed a DC function. However, constructing a DC decomposition for a polynomial of degree higher than $2$ is often a difficult problem. In \cite{niu2011}, we proposed a DC decomposition for \eqref{P_mvsk} model based on a commonly used DC decomposition in form of $f(x)=g(x) - h(x)$ with $g(x)=\frac{\eta}{2}\|x\|^2$ and $h(x)= \left(\frac{\eta}{2}\|x\|^2 - f(x)\right)$ over the standard simplex $\Omega$. The parameter $\eta>0$ should be large enough to ensure the convexity of $g$ and $h$ over $\Omega$. In other words, $f$ is supposed to be $\eta$-smooth (with a finite positive $\eta$) over $\Omega$, which is always true for any polynomial function $f$ since the spectral radius of the Hessian matrix $\nabla^2 f(x)$ over $\Omega$ is finite. However, this kind of DC decomposition requires estimating a large-enough parameter $\eta$ to ensure the local convexity of $h$ over $\Omega$. The quality of the DC decomposition depends on $\eta$, and a smaller (but large-enough) $\eta$ leads to a better DC decomposition.  Motivated by a new DC decomposition technique based on Sums-of-Squares for polynomials proposed in our recent work \cite{niu2018}, namely \emph{Difference-of-Convex-Sums-of-Squares} (cf. DC-SOS) decomposition, we can construct a new type of DC decomposition, without estimating any parameter $\eta$, in form of difference of two convex sums-of-squares. In this section, we will briefly present the commonly used DC decomposition and its corresponding DCA, then we focus on establishing the new DC-SOS decomposition and the corresponding DCA. 

\subsection{Commonly used DC decomposition and DCA for \eqref{P_mvsk} model}
The commonly used DC decompositions for minimizing a real-valued $C^2$ function $f$ over a compact convex set $\mathcal{C}\subset \R^n$ has been proposed in several literatures \cite{Lethi2000efficient,pham1997,pham1998}. In this subsection, we will briefly summarize these decompositions as follows: Let $f=f_1-f_2$ be a decomposition of $f$, where $f_1$ and $f_2$ are not supposed to be convex or differentiable, such that there exists a finite parameter $\eta>0$ to ensure that $$g(x)=\frac{\eta}{2} \|x\|_2^2 + \chi_{\mathcal{C}}(x) +f_1(x) \text{ and } h(x)=\frac{\eta}{2}\|x\|^2+f_2(x)$$ are both proper closed convex functions over $\R^n$. Particularly, there are two well-known special forms frequently used in practice: 
	\begin{enumerate}
		\item \emph{Projective DC decomposition} ($f_1=0$ and $f_2=f$): 
		$$g(x)=\frac{\eta}{2} \|x\|_2^2 + \chi_{\mathcal{C}}(x),~ h(x)=\frac{\eta}{2}\|x\|^2 - f(x).$$
		\item \emph{Proximal DC decomposition} ($f_1=f$ and $f_2=0$):
		$$g(x)=\frac{\eta}{2} \|x\|_2^2 + \chi_{\mathcal{C}}(x) +f(x),~ h(x)=\frac{\eta}{2}\|x\|^2.$$ 
	\end{enumerate} 
\paragraph{DCA with the projective DC decomposition for \eqref{P_mvsk} model}
The projective DC decomposition (cf. universal DC decomposition) applied to \eqref{P_mvsk} model is proposed in \cite{niu2011} as:
\begin{equation}\label{prob:DC_univ}
\min \{f(x) = G(x)-H(x): x \in \Omega\},
\end{equation}
where 
\begin{equation}
\label{eq:dcd-type1}
G(x)=\frac{\eta}{2} \|x\|_2^2, ~ H(x)=
\frac{\eta}{2} \|x\|^2_2 - f(x).
\end{equation} The parameter $\eta$ is estimated in \cite[Proposition 2]{niu2011} as:
\begin{equation}
\label{eq:estimateeta}
\eta = 2c_2 \|
\Sigma\|_{\infty} + 6c_3\max_{1\leq i \leq
	n}(\sum_{j,k=1}^{n}|S_{i,j,k}|)+ 12 c_4\max_{1\leq i \leq
	n}(\sum_{j,k,l=1}^{n}|K_{i,j,k,l}|)>0.
\end{equation}
The gradient of $H$ is computed by:
$$\nabla H(x) = \eta x + c_1 \mu -
2c_2 \Sigma x + c_3 \frac{\nabla^2 m_3(x) x}{2} - c_4 \frac{\nabla^2 m_4(x)
	x}{3},$$
where the Hessian matrices $\nabla^2 m_3(x)$ and $\nabla^2 m_4(x)$ are given explicitly by $$\nabla^{2} m_3(x) = \left( 6\sum_{k=1}^{n}~S_{i,j,k}x_{k} \right)_{(i,j)\in \NN^2}, \nabla^2
m_4(x) = \left(12\sum_{k,l=1}^{n}K_{i,j,k,l}~x_{k}x_{l} \right)_{(i,j)\in \NN^2}.$$ 

DCA applied to the projective DC decomposition requires solving a sequence of convex constrained strongly convex quadratic subproblems:
\begin{equation}\label{prob:qcp}
x^{k+1} = \argmin_{x\in \Omega} \frac{\eta}{2} \|x\|_2^2  - \langle x, \nabla H(x^k)\rangle,
\end{equation}
whose optimal solution $x^{k+1}$ exists and is unique due to the strong convexity of the objective function in \eqref{prob:qcp} and the compactness of the $\Omega$. Problem \eqref{prob:qcp} is equivalent to 
$$x^{k+1} = \argmin_{x\in \Omega} \| x - \frac{\nabla H(x^k)}{\eta} \|^2,$$
which is the projection of the vector $\nabla H(x^k)/\eta$ on the standard simplex $\Omega$, and can be computed explicitly by using either the strongly polynomial algorithm \emph{Block Pivotal Principal Pivoting Algorithm} (BPPPA) introduced in \cite{niu2011,judice1992}, or the explicit \emph{direct projection methods} given in \cite{minoux1984,held1974}. 

We call the DCA with the projective DC decomposition as UDCA summarized in Algorithm \ref{alg:UDCA}, whose convergence theorem is not given in \cite{niu2010}. Here we complete the convergence analysis described in Theorem \ref{thm:convUDCA} whose proof is developed in Appendix \ref{appendix:D_thm:convUDCA}.

\begin{algorithm}[ht!]
	\caption{UDCA for \eqref{P_mvsk}}
	\label{alg:UDCA}
	\begin{algorithmic}[1]
		\REQUIRE Initial point $x^0\in \R^n_+$; Tolerance for optimal value $\epsilon_1 >0$; Tolerance for optimal solution $\epsilon_2>0$;
		\ENSURE Computed solution $x^*$;
		
		\STATE $k\leftarrow 0$; $\Delta f\leftarrow \infty; \Delta x\leftarrow \infty$;
		\WHILE{$\Delta f > \epsilon_{1}$ or $\Delta x > \epsilon_{2}$}
		
		\STATE \textbf{Compute $x^{k+1}$ by solving problem \eqref{prob:qcp}};
		\STATE $f^*\leftarrow f(x^{k+1})$; $x^*\leftarrow x^{k+1}$;
		\STATE $\Delta f \leftarrow |f^* - f(x^{k})|/(1+|f^*|)$; $\Delta x \leftarrow \|x^*-x^{k}\|/(1+\|x^*\|)$;
		\STATE $k \leftarrow k+1$;
		\ENDWHILE
	\end{algorithmic}
\end{algorithm}

\begin{theorem}[Convergence theorem of UDCA for \eqref{P_mvsk}] \label{thm:convUDCA}
	UDCA for \eqref{P_mvsk} starting from any initial point $x^0\in \R^n$ (may not be in $\Omega$) will generate a sequence $\{x^k\}$ such that
	\begin{enumerate}
		\item[(i)] (sufficiently descent property) for every $k=1,2,\ldots$, we have \begin{equation}
			f(x^k)-f(x^{k+1}) \geq \frac{\eta}{2}\|x^{k}-x^{k+1}\|^2.
		\end{equation} 
	\item[(ii)] (convergence of $\{f(x^k)\}$) the sequence $\{f(x^k)\}$ is non-increasing and bounded from below, thus convergent.
	\item[(iii)] (convergence of $\{\|x^k-x^{k+1}\|\}$) the sequence $\{\|x^k-x^{k+1}\|\}$ converges to $0$ as $k\to \infty$.
	\item[(iv)] (summable properties) both $\sum_{k\geq 0} \|x^{k+1}-x^{k}\|^2 < \infty$ and $\sum_{k\geq 0} \|x^{k+1}-x^{k}\| < \infty$.
	\item[(v)] (convergence of $\{x^k\}$) the sequence $\{x^k\}$ converges to a strongly DC critical point of \eqref{P_mvsk}, which is also a KKT point of \eqref{P_mvsk}.
	\end{enumerate}
\end{theorem}

\paragraph{DCA with the proximal DC decomposition for \eqref{P_mvsk} model} DCA applied to the proximal DC decomposition for problem \eqref{P_mvsk} requires solving convex subproblems
\begin{equation}
	\label{prob:proximalDCA}
	x^{k+1}\in \argmin\{\frac{\eta}{2} \|x\|_2^2 + f(x) + \eta \langle x^k ,x\rangle: x\in \Omega \},
\end{equation}
which may be expensive numerically. Moreover, it seems that finding a stationary point of the convex problem requires almost the same amount of computation as finding a stationary point of the original problem \eqref{P_mvsk}. Therefore, we are not interested in the proximal DC decomposition in practice, except that problem \eqref{prob:proximalDCA} has a closed-form solution or can be solved efficiently.

\subsection{DC-SOS decomposition and DCA for \eqref{P_mvsk} model}
The initial motivation for proposing the DC-SOS decomposition for polynomials is to establish a DC decomposition for any polynomial without the compactness of the convex set, and without the dependence of the parameter $\eta$. The basic idea of DC-SOS decomposition is to represent any polynomial as difference of two convex and sums-of-squares polynomials, it has been proved in \cite{niu2018} that any polynomial can be rewritten in form of DC-SOS. In this subsection, after a short presentation on preliminaries of DC-SOS decomposition, we will focus on constructing a suitable DC-SOS decomposition for \eqref{P_mvsk} model and developing the corresponding DCA. 
\subsubsection{Preliminaries on DC-SOS decomposition}
\begin{definition}[DC-SOS decomposition, see \cite{niu2018}]\label{def:DC-SOS}
	A polynomial $p$ is called \emph{difference-of-convex-sum-of-squares} (DC-SOS) if there exist convex-sum-of-squares (CSOS) polynomials $s_1$ and $s_2$ such that $p = s_1 - s_2.$ 
	\begin{enumerate}
		\item[$\bullet$] The components $s_1$ and $s_2$ are called \emph{DC-SOS components} of $p$;
		\item[$\bullet$] The set of all DC-SOS polynomials in $\R[x]$ is denoted by $\DC-SOS_n$;
	\end{enumerate}
\end{definition}
The next theorem shows the equivalence among the vector spaces $\R[x]$ and $\DC-SOS_n$ as well as the minimal degree for DC-SOS components.
\begin{theorem}[See \cite{niu2018}]\label{thm:equivDC-SOS} Any polynomial can be presented in DC-SOS and  
	\begin{itemize} 
		\item[$\bullet$] For any DC-SOS components $s_1$, $s_2$ of $p \in \R[x]$, we have $\max\{\deg(s_1),\deg(s_2)\} \geq 2\lceil\frac{\deg(p)}{2} \rceil$.
		\item[$\bullet$] There exist DC-SOS components $s_1$, $s_2$ of $p\in \R[x]$ such that $\max\{\deg(s_1),\deg(s_2)\}= 2\lceil\frac{\deg(p)}{2} \rceil$.
	\end{itemize}
\end{theorem}
The complexity for constructing DC-SOS decompositions is:
\begin{theorem}[See \cite{niu2018}]\label{thm:complexityforD-SOSandDC-SOS}
	Any polynomial $p\in \R_d[x]$ can be rewritten as DC-SOS for any desired precision in polynomial time by solving an SDP (Semi-Definite Program).
\end{theorem}

Note that using SDP is only for theoretical proof of polynomial-time constructibility of DC-SOS decompositions for any desired precision. However, solving an SDP is not suggested in practice, which often leads to an approximate DC-SOS decomposition, and whose construction requires solving a large-scale SDP for high-order polynomials, despite that SDP can be solved to any desired precision via polynomial-time interior point method. In fact, in many practical applications, we do not need to solve any SDP to find an \textbf{exact} DC-SOS decomposition. There are several practical DC-SOS decomposition techniques as established in \cite{niu2018}.

\subsubsection{DC-SOS decomposition for \eqref{P_mvsk} model}
For the \eqref{P_mvsk} model, we suggest using the parity DC-SOS decomposition algorithm \cite{niu2018}. The basic idea of the parity DC-SOS decomposition is based on the fact that any polynomial can be factorized as multiplications and additions of three elementary cases whose DC-SOS decompositions are easily computed as follows:

\noindent $\rhd$ \underline{For $x_ix_j$}: a DC-SOS decomposition is 
\begin{equation}\label{eq:DC-SOS_elementarycase1}
x_ix_j=\frac{1}{4}(x_i+x_j)^2-\frac{1}{4}(x_i-x_j)^2 \text{ or } x_ix_j=\frac{1}{2}(x_i+x_j)^2-\frac{1}{2}(x_i^2 + x_j^2).
\end{equation} 
A single variable $x_i$ is a special case of $x_ix_j$ with $x_j=1$.

\noindent $\rhd$  \underline{For $x_i^{2k}, k\in \N$}: a DC-SOS decomposition is
\begin{equation}\label{eq:DC-SOS_elementarycase2}
x_i^{2k} = x_i^{2k} - 0.
\end{equation}

\noindent $\rhd$  \underline{For $p\times q$ with $(p,q)\in \DC-SOS^2$}: Let $p_1-p_2$ and $q_1-q_2$ be DC-SOS decompositions of $p$ and $q$, then a DC-SOS decomposition of $p\times q$ is given by
\begin{equation}\label{eq:DC-SOS_elementarycase3}
p\times q 
= \frac{1}{2}[(p_1+q_1)^2 + (p_2+q_2)^2] - \frac{1}{2}[(p_1+q_2)^2 + (p_2+q_1)^2].
\end{equation}

For the \eqref{P_mvsk} model, we only need to apply the parity DC-SOS decomposition to polynomials $m_3$ and $m_4$, since $m_1$ and $m_2$ are already convex (linear for $m_1$ and quadratic convex for $m_2$).
\paragraph{DC-SOS decomposition for $m_3$}
By the symmetry of the co-skewness tensor $S$, we can rewrite $m_3$ as
\begin{align*}
m_3(x) = &\sum_{i,j,k=1}^{n}S_{i,j,k}~x_ix_jx_k = \sum_{i=1}^{n}S_{i,i,i}~x_{i}^3 + \binom{3}{1} \sum_{i=1}^{n}\sum_{k\neq i}S_{i,i,k}~x_{i}^2x_{k} + 3! \sum_{1\leq i<j<k\leq n}S_{i,j,k}~x_{i}x_{j}x_{k}.
\end{align*}
Let $\NN=\{1,\ldots,n\}$, $\PP=\{(i,k): i\in \NN, k\neq i\}$, and $\QQ=\{(i,j,k): 1\leq i<j<k\leq n\}$, with sizes $|\NN|=n, |\PP|=n(n-1),$ and $|\QQ|=\binom{n}{3}$, then the expression of $m_3$ is simplified as:
$$m_3(x) = \sum_{i\in \NN}S_{i,i,i}~x_{i}^3 + 3 \sum_{(i,k)\in \PP}S_{i,i,k}~x_{i}^2x_{k} + 6 \sum_{(i,j,k)\in \QQ}S_{i,j,k}~x_{i}x_{j}x_{k}.$$
There are three types of monomials $x_i^2x_k$, $x_i x_j x_k$ and $x_i^3$ in $m_3$ whose DC decompositions based on DC-SOS can be easily established. Then a DC-SOS decomposition for $m_3$ is given by $$m_3(x) =  g_{m_3}(x) - h_{m_3}(x),$$
where $g_{m_3}$ and $h_{m_3}$ are both convex functions on $\R^n$ defined in \eqref{eq:gm3} and \eqref{eq:hm3}. More details for computing $g_{m_3}$, $h_{m_3}$ and $\nabla h_{m_3}$ are described in Appendix \ref{appendix:B_m3}.

\paragraph{DC-SOS decomposition for $m_4$}
A DC-SOS decomposition for $m_4$ is constructed in a similar way as $m_3$. We firstly rewrite $m_4$ as  
\begin{align*}
m_4(x) = & \sum_{i\in \NN}K_{i,i,i,i}~x_{i}^4 + 4 \sum_{(i,k)\in \PP}K_{i,i,i,k}~x_{i}^3x_{k} + 6 \sum_{(i,k)\in\PPH}K_{i,i,k,k}~x^2_{i}x^2_{k} \\
& +  12 \sum_{(i,j,k)\in \QQH}K_{i,i,j,k}~x^2_{i}x_{j}x_{k} + 24 \sum_{(i,j,k,l)\in \RR}K_{i,j,k,l}~x_{i}x_{j}x_{k}x_{l}.
\end{align*}
where 
$\NN=\{1,\ldots,n\}$, $\PP=\{(i,k): i\in \NN, k\neq i\}$, $\PPH=\{(i,k)\in \PP: k>i\}$, $\QQH=\{(i,j,k): i\in \NN, (j<k)\neq i\}$, and $\RR=\{(i,j,k,l): 1\leq i<j<k<l\leq n \}$. The sizes of these sets are $|\NN|=n, |\PP|=n(n-1), |\PPH|=\frac{n(n-1)}{2}, |\QQH|=n\binom{n-1}{2},$ and $|\RR| = \binom{n}{4}$. 

Five types of monomials $x_i^4$, $x_i^3x_k$, $x_i^2 x_k^2$, $x_i^2 x_j x_k$ and $x_i x_j x_k x_l$ are considered whose DC decompositions based on DC-SOS can be easily established. A DC decomposition for $m_4$ is given by $$m_4(x) =  g_{m_4}(x) - h_{m_4}(x),$$
where $g_{m_4}$ and $h_{m_4}$ are convex functions on $\R^n$ defined in \eqref{eq:gm4} and \eqref{eq:hm4}. The computations of $g_{m_4}$, $h_{m_4}$ and $\nabla h_{m_4}$ are summarized in Appendix \ref{appendix:C_m4}. 

\paragraph{DC-SOS decomposition for $f$}\label{subsec:dcp}
Based on above discussions, a DC decomposition for the polynomial objective function $f$ of \eqref{P_mvsk} model based on DC-SOS decomposition is given by
\begin{align*}
f(x) =& -c_1 m_1(x) + c_2 m_2(x) - c_3 m_3(x) + c_4 m_4(x) = \bar{G}(x) - \bar{H}(x),
\end{align*}
where
\begin{equation}
\label{eq:G}
\bar{G}(x) = -c_1 m_1(x) + c_2 m_2(x) + c_3 h_{m_3}(x) + c_4 g_{m_4}(x),
\end{equation}
\begin{equation}
\label{eq:H}
\bar{H}(x) = c_3 g_{m_3}(x) + c_4 h_{m_4}(x),
\end{equation}
are both convex quartic polynomials on $\Omega$. Then \eqref{P_mvsk} model is formulated as a DC program:
\begin{equation}
\label{prob:DC}
\min\{ \bar{G}(x)-\bar{H}(x): x\in \Omega \}.\tag{DCP}
\end{equation}

Note that DC-SOS decomposition can be applied to the portfolio optimization involving any order of moments (even greater than $4$), and its DC decomposition can be derived in a similar way. The numerical performance for constructing DC-SOS decompositions for polynomials of large-scale with both sparse structure and dense structure are reported in \cite{niu2018}.

\subsubsection{DCA with DC-SOS decomposition for \eqref{P_mvsk} model}\label{sec:DCA}
DCA applied to DC-SOS decomposition is very similar to Algorithm \ref{alg:UDCA}. The only difference is in the line 3 where $x^{k+1}$ is computed by solving a \emph{convex quartic optimization problem}:
\begin{equation}
\label{prob:Pk}
x^{k+1}\in \argmin \{ \bar{G}(x) - \langle x, \nabla \bar{H}(x^k) \rangle: x\in \Omega \}.
\end{equation}
where $\bar{G}$ is a $4$th degree CSOS polynomial given in \eqref{eq:G}, and the gradient $\nabla \bar{H}(x^k) = c_3\nabla g_{m_3}(x^k) + c_4 \nabla h_{m_4}(x^k)$ are easily derived in Appendix \ref{appendix:B_m3} and \ref{appendix:C_m4}. 

Note that efficiently solving the convex quartic optimization problem over a standard simplex of type \eqref{prob:Pk} is a crucial question to the performance of DCA with DC-SOS decomposition. Fortunately, problem \eqref{prob:Pk}, as a smooth and convex optimization problem over a standard simplex, can be solved efficiently by many classical first- and second-order approaches such as gradient-type methods, interior point methods, and Newton-type methods etc. Accelerations (e.g., the Heavy ball-type and the Nesterov's acceleration) can be applied to these methods for better numerical performance \cite{2018Lectures}. As an example, FISTA (A fast iterative shrinkage-thresholding algorithm), a variant of the gradient descent method with Nesterov's acceleration, has $O(1/k^2)$ convergence rate, which is faster than the non-accelerated version ISTA with $O(1/k)$ convergence rate only \cite{FISTA2009A}. Moreover, there exist some efficient optimization packages such as \verb|KNITRO|, \verb|IPOPT|, \verb|FILTERSD| and MATLAB \verb|fmincon| for solving subproblem \eqref{prob:Pk}. In our numerical tests, we have tried all above solvers and chosen \verb|KNITRO| as the fastest solver among the others. Note again that
finding a suitable DC-SOS decomposition (e.g., with sparsity) for dense polynomials such that the convex subproblem \eqref{prob:Pk} can be solved more effectively is an important topic in our future work. 



%
%

The convergence analysis of DCA applied to the DC-SOS decomposition is exactly the same as in Theorem \ref{thm:convUDCA}. The only concern is the convergence of $\{x^k\}$ which requires the strongly convexity of $\bar{G}$ or $\bar{H}$.  This will not be a problem since a strong convex term $\frac{\rho}{2}\|x\|^2$ (for any $\rho >0$) can be introduced to both $\bar{G}(x)$ and $\bar{H}(x)$ to get a DC-SOS decomposition 
	\begin{equation}
	\label{eq:DCSOS-stronglyconvex}
	f(x) = \left(\bar{G}(x)+\frac{\rho}{2}\|x\|^2\right) - \left(\bar{H}(x) + \frac{\rho}{2}\|x\|^2\right)
	\end{equation}
with $\rho$-strong convexity in DC components. 

Note that there is always a trade off between the convexity of the DC components and the convergence of the sequence $\{x^k\}$ generated by DCA. Roughly speaking, a stronger convexity in the DC components gives more guarantee in the convergence of the sequence $\{x^k\}$, but leads to worse convex overestimations $f^k$ at $x^k$, and requires more iterations of DCA. For example, although computing a projection of a point on the standard simplex \eqref{prob:qcp} could be less expensive than solving the convex quartic polynomial optimization problem \eqref{prob:Pk}, however, DC-SOS decomposition may provide better convex overestimation of $f$, which can be determined by the number of iterations required by DCA from the same initial point. Later, large number of numerical tests in Section \ref{sec:Simu} will demonstrate that the proposed DC-SOS decomposition indeed requires less number of iterations for DCA, thus provides better DC decomposition for polynomials.

\section{DC acceleration technique}\label{sec:boosteddca}
In the case where the minimizer of the DC function $f=g-h$ lies in a flat region, then the convergence of DCA will be slow down. This is a common issue to the first-order methods. Particularly for DCA, when the iteration point $x^k$ comes into a flat region of $f$, the convex overestimation $f^k:x\mapsto g(x)- h(x^k) - \langle \nabla h(x^k),x-x^k \rangle$ of $f$ at $x^k$ may not be flat (since $g$ may not be flat) around $x^k$, in this case, $f^k$ will be a bad convex overestimation of $f$ around $x^k$, thus DCA will be slow down. 

To overcome this difficulty and improve the overall performance of DCA, we propose an accelerated technique for DC programming, namely  Boosted-DCA, which consists of introducing an inexact line search (e.g., of the Armijo-type) to get an improved iteration point. Fukushima-Mine introduced such a line search in a proximal point algorithm for minimizing the sum of continuously differentiable functions with a closed proper convex function \cite{fukushima1981,mine1981}. Then, Arag{\'{o}}n Artacho-Vuong et al. applied Armijo line search to accelerate DCA for unconstrained DC program with smooth $g$ and $h$ in \cite{artacho2018}, and with smooth $g$ and nonsmooth $h$ in \cite{artacho2020} (after the first version of our manuscript on Arxiv). In this section, we will focus on more general cases: DC program (both smooth and nonsmooth cases) within a closed convex set.

Consider the DC program over a convex set defined as:
\begin{equation}\label{prob:CCDC}
\min \{f(x):=g(x) - h(x) : x\in \mathcal{C}\},\tag{P}
\end{equation}
where $g$ and $h$ belong to $\Gamma_0(\R^n)$, and $\C$ is a nonempty closed convex set which is defined by a set of inequalities and equalities as $\mathcal{C}:= \{x\in \R^n: u(x)\leq 0, v(x)=0\}$ with 
$u:\R^n \to \R^p$ being convex and $v:\R^n\to \R^q$ being affine. 

\subsection{DC descent direction}\label{subsec:dcdd}
\begin{definition}[DC descent direction]
	Let $x^k$ be a feasible point of problem \eqref{prob:CCDC}, and $y^{k}$ be the next iteration point obtained by DCA from $x^k$. The vector $d^{k}=y^{k}-x^{k}$ is called a \emph{DC descent direction of $f$ at $y^{k}$ over $\mathcal{C}$} if $\exists \eta >0, \forall t\in (0,\eta), y^{k}+t d^k \in \mathcal{C}$ and $f(y^{k}+td^k) < f(y^{k})$.
\end{definition}

The name of DC decent direction comes from the fact that $d^k$ is constructed using two consecutive iteration points $x^k$ and $y^{k}$ of DCA. Note that if $x^k$ is not a critical point, then $d^k$ is always a feasible direction of $f$ at $x^{k}$ over $\mathcal{C}$, and it is a descent direction of $f$ at $x^k$ over $\C$ if $f(x^k)>f(y^{k})$ (in general, we only have $f(x^k)\geq f(y^{k})$). However, $d^k$ may neither be a feasible direction nor a descent direction of $f$ at $y^{k}$ over $\C$. 

Given a DC descent direction $d^k$, then we can proceed a line search at $y^{k}$ along the direction $d^k$ to accelerate the convergence of DCA. Next, we will discuss some properties of the DC descent direction for both differentiable and non-differentiable cases.

\subsubsection{Differentiable case : $g$, $h$ and $u$ are continuously differentiable}\label{subsubsec:diff}
We first consider the differentiable case where the functions $g$, $h$ and $u$ in \eqref{prob:CCDC} are all continuously differentiable, and some regularity conditions (e.g., the linearity constraint qualification, the Slater's condition, or the Mangasarian-Fromovitz constraint qualification) hold. Then we can prove the following two theorems whose proofs are given in Appendix \ref{appendix:E_thm:decdirect} and \ref{appendix:F_thm:decdirectbis}.
\begin{theorem}\label{thm:decdirect}
	Let $g, h$ and $u$ be continuously differentiable, $v$ be affine. Let $x^k$ and $y^{k}$ be two consecutive iteration points obtained by DCA for problem \eqref{prob:CCDC}, and $d^k:=y^{k}-x^{k}$. Then 
	$$\langle \nabla f(y^{k}), d^k \rangle \leq 0. $$
\end{theorem}

\begin{theorem}\label{thm:decdirectbis}
	Under the same assumptions as in Theorem \ref{thm:decdirect} and further suppose that $h$ is $\rho$-strongly convex ($\rho>0$). Then 
	$$\langle \nabla f(y^{k}),d^k \rangle \leq - \rho \|d^k\|^2.$$
\end{theorem}

Note that the regularity condition is in fact not necessary for proving Theorems \ref{thm:decdirect} and \ref{thm:decdirectbis}. A more general proof using normal cone without regularity condition is also provided later in the proof of Theorem \ref{thm:decdirect-nondiff} for non-differentiable cases, which is also available for the differentiable one.  

Proposition \ref{prop:SCforDCdescentdirection} provides two sufficient conditions for a DC descent direction, whose proof is described in Appendix \ref{appendix:prop:SCforDCdescentdirection}.
\begin{proposition}[Sufficient conditions for a DC descent direction]\label{prop:SCforDCdescentdirection} Under the same assumptions as in Theorem \ref{thm:decdirect}. Then $d^k$ is a DC descent direction of $f$ at $y^{k}$ over $\C$ if one of the following conditions holds:
	\begin{itemize}
		\item[(i)] $\langle \nabla f(y^{k}),d^k \rangle<0$ and $d^k$ is a feasible direction of $\C$ at $y^{k}$; 
		\item[(ii)] $h$ is $\rho$-strongly convex ($\rho>0$), $d^k\neq 0$ and $d^k$ is a feasible direction of $\C$ at $y^{k}$. 
	\end{itemize}
\end{proposition}

Based on Proposition \ref{prop:SCforDCdescentdirection}, one important question for checking DC descent direction is to know that $d^k$ is a feasible direction of $\C$ at $y^{k}$. Let $A(x)$ denote the \emph{active set} at a point $x\in \C$, i.e., $u_i(x)=0,\forall i\in A(x)$, then the next Proposition \ref{prop:feasdirect} provides a necessary condition for a feasible direction, whose proof is described in Appendix \ref{appendix:G_prop:feasdirect}.

\begin{proposition}[Necessary condition for a feasible direction]\label{prop:feasdirect}
	Under the same assumptions as in Theorem \ref{thm:decdirect}. Then $d^k$ is a feasible direction of $\C$ at $y^{k}$ implies that 
	\begin{equation}
		\label{eq:NC}
		A(y^{k}) \subset A(x^{k}).
	\end{equation}
\end{proposition}

Condition \eqref{eq:NC} is in general not a sufficient for a feasible direction. For instance, consider the constraint $\{x\in \R^2: u_1(x) =\|x\|^2-1\leq 0, u_2(x) = \|x-1\|^2-1\leq 0 \}$, if we take $x^k=(0.5,\sqrt{3}/2)$ and $y^{k}=(0,0)$, then clearly $A(y^{k})=\{2\}\subset \{1,2\} = A(x^k)$, but the vector $d^k=y^{k}-x^{k}$ is not a feasible direction at $y^{k}$. Particularly, this condition can be sufficient if $u$ is affine. The next theorem describes a necessary and sufficient condition whose proof is given in Appendix \ref{appendix:H_thm:necsuffcondfeasdirect}. 
\begin{theorem}[Necessary and sufficient condition for a feasible direction]\label{thm:necsuffcondfeasdirect}
	Under the same assumptions as in Theorem \ref{thm:decdirect} and suppose that $u$ is affine. Then $A(y^{k}) \subset A(x^{k})$ is a necessary and sufficient condition for $d^k$ being a feasible direction of $\C$ at $y^{k}$.
\end{theorem}

Note that in \eqref{P_mvsk}, $\Omega=\{x: v(x) = e\tran x=1, u(x) =x\geq 0 \}$, the function $u$ is linear, so that based on Theorems \ref{thm:necsuffcondfeasdirect}, $A(y^{k})\subset A(x^k)$ is a necessary and sufficient condition for the feasibility of the direction $d^k$ at $y^{k}$ over $\C$.  

\subsubsection{Non-differentiable case: $h$ and $u$ are non-differentiable}\label{subsubsec:nondiff}
An interesting question is to generalize Theorems \ref{thm:decdirect}, \ref{thm:decdirectbis} and Proposition \ref{prop:feasdirect} for non-differentiable case without regularity conditions. Fortunately, we have similar results as described in the next Theorem if the function $g$ is differentiable, $h$ and $u$ are non-differentiable, and the regularity conditions are not required. 

\begin{theorem}
	\label{thm:decdirect-nondiff}
	Let $g$ be differentiable, $v$ be affine, $x^k$ and $y^{k}$ be two consecutive iteration points obtained by DCA for problem \eqref{prob:CCDC}, and $d^k=y^{k}-x^{k}$.
	\begin{itemize}
		\item[$\bullet$] If $h$ is non-differentiable convex, then $f'(y^{k}; d^k) \leq 0.$
		\item[$\bullet$] If $h$ is non-differentiable $\rho$-strongly convex ($\rho >0$), then $\exists \rho>0$ such that 
		$f'(y^{k};d^k) \leq - \rho \|d^k\|^2.$
		\item[$\bullet$] If $u$ is non-differentiable convex and $d^k$ is a feasible direction of $\C$ at $y^{k}$, then $A(y^{k}) \subset A(x^{k}).$		
	\end{itemize}
\end{theorem}

A proof of Theorem \ref{thm:decdirect-nondiff} is quite different from the differentiable one, which is described in Appendix \ref{appendix:I_thm:decdirect-nondiff}. Note that if $u$ is differentiable and affine, then $A(y^{k}) \subset A(x^{k})$ is again a necessary and sufficient condition for a feasible direction $d^k$ at $y^{k}$.

Note that the problem \eqref{P_mvsk} has polynomial functions only, so Theorem \ref{thm:decdirect-nondiff} is not addressed to this application. But, we still present Theorem \ref{thm:decdirect-nondiff} to the interested readers for potential uses in non-differentiable applications. 

\subsection{Armijo-type line search}\label{subsec:Armijo}
Suppose that $d^k=y^{k}-x^{k}$ is a DC descent direction for problem \eqref{prob:CCDC}, we are going to find a suitable stepsize $\alpha > 0$ moving from $y^{k}$ to $x^{k+1}$ along the direction $d^k$ as
\begin{equation}
\label{eq:linesearch}
x^{k+1} = y^{k} + \alpha d^k,
\end{equation}
verifying $f(x^{k+1})< f(y^{k})$ and $x^{k+1}\in \mathcal{C}$. The exact line search finds the best $\alpha$ by solving the one-dimensional minimization problem:
$$\min \{f(y^{k} + \alpha d^k) : \alpha > 0,  y^{k} + \alpha d^k\in \mathcal{C} \}$$
using classical line search methods such as Fibonacci and golden section search, and line search methods based on curve fitting etc. However, the exact line search is often cumbersome. As a matter of fact, for handling large-scale cases, it is often desirable to sacrifice accuracy in the line search in order to conserve overall computation time. Therefore, we are more interested in inexact line search, e.g., Armijo-type line search, in which we won't find the best $\alpha$, but try to find an available $\alpha$ satisfying $f(x^{k+1})< f(y^{k})$ and $x^{k+1}\in \Omega$. 

The Armijo's rule (see e.g., \cite{bertsekas1999}) suggests to find a suitable $\alpha>0$ verifying
\begin{equation}
	\label{eq:armijorule-org}
	f(y^{k}) - f(x^{k+1}) \geq  - \sigma\alpha f'(y^{k};d^k).
\end{equation}
We start from an initial trial stepsize $\alpha>0$ (neither too large nor too small, e.g., $\alpha=1$). Then, taking $\sigma\in (0,1)$ and $\beta\in (0,1)$, where $\beta$ denotes the \emph{reduction factor} (or \emph{decay factor}) to reduce the stepsize $\alpha$ to $\beta \alpha,$ and $\sigma$ is chosen to be closed to zero, e.g., $\beta\in [0.1,0.5]$ and $\sigma\in [10^{-5},0.1]$. 

Note that \eqref{eq:armijorule-org} is applicable if and only if $f'(y^{k};d^k)$ can be easily computed. For example, when $f$ is differentiable at $y^k$, then $f'(y^{k};d^k) = \langle \nabla f(y^k), d^k \rangle$. Otherwise, for $\rho$-strongly convex function $h$ ($\rho>0$), if $\alpha$ is reduced smaller than $\rho$, we get from the Armijo's rule \eqref{eq:armijorule-org} and Theorem \ref{thm:decdirectbis} (differentiable cases) or Theorem \ref{thm:decdirect-nondiff} (non-differentiable cases) that  
$$f(y^{k}) - f(x^{k+1}) \geq  - \sigma\alpha f'(y^{k};d^k) \geq \sigma \alpha \rho \|d^k\|^2\geq \sigma \alpha^2 \|d^k\|^2 > 0.$$
Therefore, we can stop reducing $\alpha$ when $x^{k+1}\in \C$ and verifying  
\begin{equation}
\label{eq:armijo}
f(x^{k+1}) \leq  f(y^{k}) - \sigma\alpha^2 \|d^k\|^2.
\end{equation}

Note that the choice of the parameters $\beta$ and $\sigma$ depends on the confidence we have on the initial stepsize $\alpha$, which should be neither too large nor too small. If $\alpha$ is too large, then we may need a fast reduction in $\alpha$, so that $\beta$ and $\sigma$ should be chosen small; If $\alpha$ is too small, e.g., $\alpha \leq \varepsilon/\|d^k\|$ for tolerance $\varepsilon>0$, then we get from \eqref{eq:linesearch} that $$\|x^{k+1}-y^{k}\| = \|\alpha d^k\| \leq \varepsilon.$$
In this case, there is no need to continue the line search and we will set $x^{k+1}=y^{k}$. The proposed Armijo-type line search is summarized as follows:

\begin{proc}{ht!}{Procedure}
	\caption{Armijo line search}
	\label{alg:Armijo}
	\begin{algorithmic}[1]
		\REQUIRE descent direction $d^k=y^{k}-x^k$; point $y^{k}$; reduction factor $\beta\in (0,1)$ (e.g., $\beta=0.3$); initial stepsize $\alpha>0$ (e.g., $\alpha=1$); parameter $\sigma\in (0,1)$ (e.g., $\sigma=10^{-3}$); tolerance for line search $\varepsilon>0$.
		\ENSURE potentially improved candidate $x^{k+1}$.
		
		\WHILE{$\alpha > \varepsilon/\|d^k\|$}		
		\STATE $x^{k+1} \leftarrow y^{k} + \alpha d^k$;
		\STATE $\Delta \leftarrow f(y^{k}) - f(x^{k+1}) - \sigma \alpha^2 \|d^k\|^2$;
		\IF{$\Delta \geq 0$ and $x^{k+1}\in \C$}
		\RETURN $x^{k+1}$;
		\ENDIF
		\STATE $\alpha \leftarrow \beta \alpha$;
		\ENDWHILE
		\STATE $x^{k+1}\leftarrow y^{k}$; 
		\RETURN $x^{k+1}$.
	\end{algorithmic}
\end{proc}

\subsection{Boosted-DCA}\label{subsec:Boosted-DCA}
Combining the DCA with Armijo line search along DC descent direction, we propose the Boosted-DCA for problem \eqref{prob:CCDC} described in Algorithm \ref{alg:BDCAforP}. 
\begin{algorithm}[ht!]
	\caption{Boosted-DCA for problem \eqref{prob:CCDC}}
	\label{alg:BDCAforP}
	\begin{algorithmic}[1]
		\REQUIRE initial point $x^0\in \C$; tolerance for optimal value $\varepsilon_1 >0$; tolerance for optimal solution $\varepsilon_2>0$;
		\ENSURE computed solution $x^*$;
		
		\STATE $k\leftarrow 0$; $\Delta f\leftarrow \infty; \Delta x\leftarrow \infty$;
		\WHILE{$\Delta f > \varepsilon_{1}$ or $\Delta x > \varepsilon_{2}$}
		\STATE $z^k \in \partial h(x^k)$;
		\STATE $y^{k}\in \argmin\{g(x)-\langle x, z^k \rangle :x\in \C\}$;
		\STATE $d^k\leftarrow y^{k}-x^{k}$;
		\IF{$A(y^{k})\subset A(x^{k})$ and $f'(y^{k};d^k) < 0$}
		\STATE compute $x^{k+1}$ using the Armijo line search from $y^{k}$;
		\ELSE
		\STATE $x^{k+1}\leftarrow y^k$;
		\ENDIF
		\STATE $f^*\leftarrow f(x^{k+1})$; $x^*\leftarrow x^{k+1}$;
		\STATE $\Delta f \leftarrow |f^* - f(x^{k})|/(1+|f^*|)$; $\Delta x \leftarrow \|x^*-x^{k}\|/(1+\|x^*\|)$;
		\STATE $k \leftarrow k+1$;
		\ENDWHILE
		\RETURN $x^*$;
	\end{algorithmic}
\end{algorithm}

Theorem \ref{thm:convofBDCA} is the main result in this section for the convergence analysis of the Boosted-DCA whose proof is established in Appendix \ref{appendix:J_thm:convofBDCA}.
\begin{theorem}\label{thm:convofBDCA}
	Let $\{(x^k,y^k,z^k)\}$ be the sequence generated by Boosted-DCA Algorithm \ref{alg:BDCAforP} for problem \eqref{prob:CCDC} from an initial point $x^0\in \dom \partial h$. Let $g$ be convex over $\C$ with modulus $\rho_g\geq 0$ and $h$ be convex over $\C$ with modulus $\rho_h\geq 0$. Suppose that either $g$ and $h$ is strongly convex over $\C$ (i.e., $\rho_g + \rho_h>0$), the sequence $\{(x^k,y^k,z^k)\}$ is bounded and $f$ is bounded from below over $\C$. Then 
	\begin{enumerate}
		\item[(i)] (sufficiently descent property) for every $k=1,2,\ldots$, we have \begin{equation}
			f(x^k)-f(y^{k}) \geq \frac{\rho_g + \rho_h}{2}\|x^{k}-y^{k}\|^2.
		\end{equation} 
		\item[(ii)] (convergence of $\{f(x^k)\}$) the sequence $\{f(x^k)\}_{k\geq 1}$ is non-increasing and convergent.
		\item[(iii)] (convergence of $\{\|x^k-y^{k}\|\}$ and  $\{\|x^k-x^{k+1}\|\}$) $\|x^k-y^{k}\|\xrightarrow{k\to \infty} 0$ and $\|x^k-x^{k+1}\|\xrightarrow{k\to \infty} 0$.
		\item[(iv)] (square summable property) $\sum_{k\geq 0} \|y^{k}-x^{k}\|^2 < \infty$ and $\sum_{k\geq 0} \|x^{k+1}-x^{k}\|^2 < \infty$.
		\item[(v)] (subsequential convergence of $\{x^k\}$) any limit point of the sequence $\{x^k\}$ is a DC critical point of 		\eqref{prob:CCDC}. Moreover, if $h$ is continuously differentiable, then any limit point of the sequence $\{x^k\}$ is a strongly DC critical point of \eqref{prob:CCDC}.\\
		Let $\Phi(x):=f(x)+\chi_{\mathcal{C}}(x)$ and suppose that $\Phi$ verifies the {\L}ojasiewicz subgradient inequality, and $h$ has locally Lipschitz continuous gradient over $\C$. Then  \item[(vi)] (summable property) $\sum_{k\geq 0} \|y^{k}-x^{k}\| < \infty$ and $\sum_{k\geq 0} \|x^{k+1}-x^{k}\| < \infty$.
		\item[(vii)] (convergence of $\{x^k\}$) the sequence $\{x^k\}$ converges to a strongly DC critical point of \eqref{prob:CCDC}.
	\end{enumerate}
\end{theorem}

\begin{figure}[ht!]
	\centering
	\includegraphics[width=0.7\linewidth]{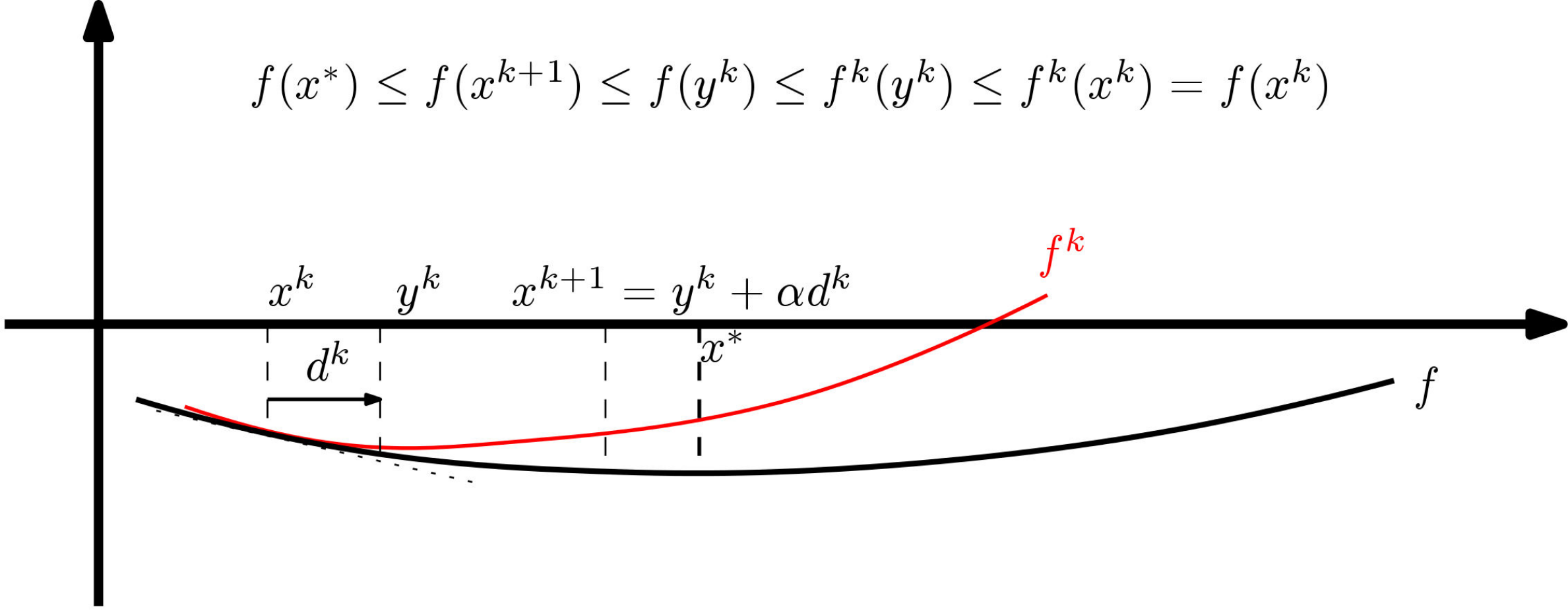}
	\caption{How BDCA accelerates the convergence of DCA.}
	\label{fig:BDCA}
\end{figure}
Figure \ref{fig:BDCA} illustrates how Boosted-DCA accelerates the convergence of DCA. The classical DCA starting from $x^k$ performs as minimizing a convex overestimation of $f$ at $x^k$, i.e., the convex function $f^k$, to get its minimum at $y^{k}$; while the Boosted-DCA starting from $x^k$ will proceed an Armijo-type line search at $y^{k}$ along the DC descent direction $d^k=y^{k} - x^{k}$. This combination could lead to a better candidate $x^{k+1}$ verifying that $f(x^{k+1}) \leq f(y^{k})$, and thus accelerating the iteration points approaching a local minimum. Note that, this acceleration could be particularly efficient when $x^k$ is located in a flat region and $f^k$ is not a good local approximation of $f$ around $x^k$.  

On the other hand, it is worth noting that DCA and Boosted-DCA with the same initial point may converge to different solution as illustrated in Figure \ref{fig:BDCAfindbettersolution}. DCA starting from $x^k$ tends to the nearest local minimum $x^*$. While, Boosted-DCA starting from $x^k$, if $\alpha$ is well chosen as illustrated, leads to a better solution $\bar{x}^*$. So that the line search could increase the potential of avoiding DCA trapped into the nearest undesirable local minima.
\begin{figure}[ht!]
	\centering
	\includegraphics[width=0.7\linewidth]{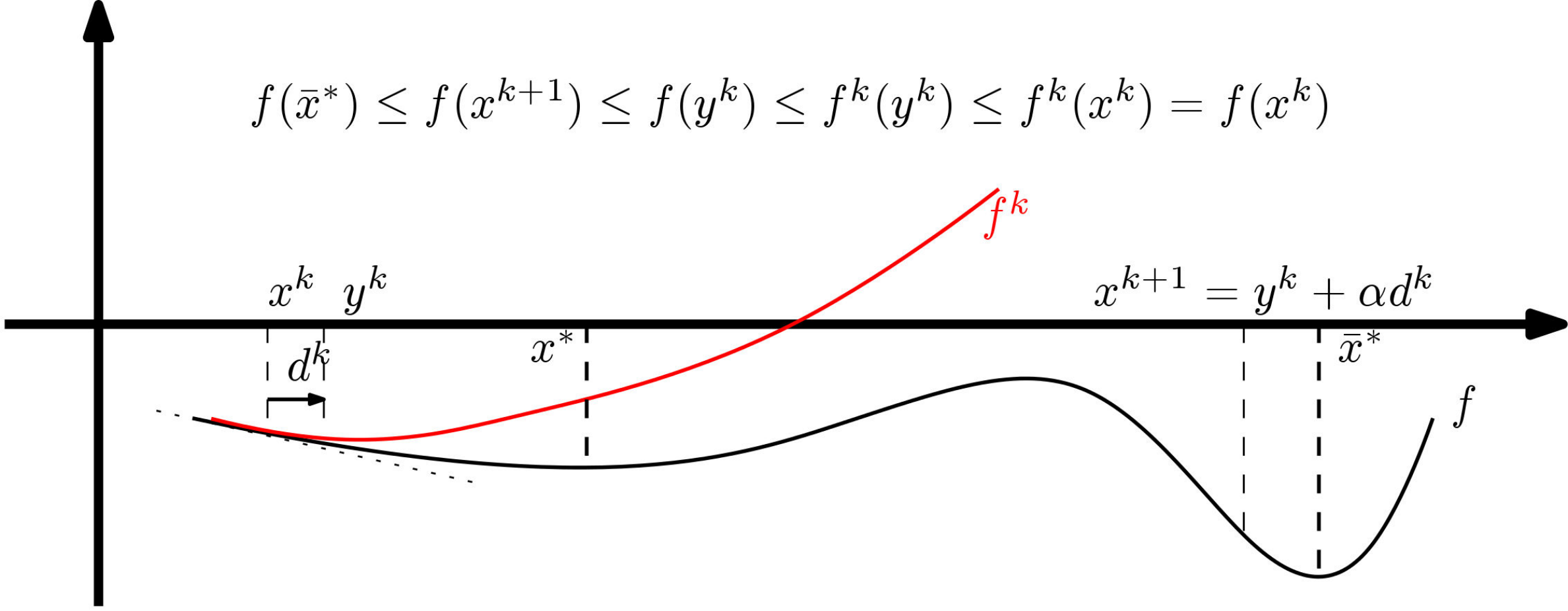}
	\caption{Boosted-DCA and DCA may converge to different local minima.}
	\label{fig:BDCAfindbettersolution}
\end{figure}



\section{Boosted-DCA for \eqref{P_mvsk} model}\label{sec:UDCA&UBDCA}
\subsection{Boosted-DCA with projective DC decomposition}\label{subsec:UBDCA}
For \eqref{P_mvsk} model, the distance between any two points in $\Omega$ is smaller than $\sqrt{2}$, so that we choose the initial $\alpha =\frac{\sqrt{2}}{\|d^k\|}$, which is large enough since $\alpha =\|x^{k+1}-y^{k}\| / \|d^k\| \leq  \frac{\sqrt{2}}{\|d^k\|}$.   Moreover, all functions $f$, $G$ and $H$ are polynomials (i.e., continuously differentiable and analytic) and $G$ is $\eta$-strongly convex with $\eta>0$ computed by \eqref{eq:estimateeta}. So that all convergence results in Theorem \ref{thm:convofBDCA} are valid. The detailed Boosted-DCA applied to problem \eqref{P_mvsk} with projective DC decomposition is described in Algorithm \ref{alg:BDCA_univ}, namely UBDCA.

\begin{algorithm}[ht!]
	\caption{UBDCA for problem \eqref{P_mvsk}}
	\label{alg:BDCA_univ}
	\begin{algorithmic}[1]
		\REQUIRE initial point $x^0\in \C$; tolerance for optimal value $\epsilon_1 >0$; tolerance for optimal solution $\epsilon_2>0$;
		\ENSURE computed solution $x^*$;
		
		\STATE $k\leftarrow 0$; $\Delta f\leftarrow \infty; \Delta x\leftarrow \infty$;
		\WHILE{$\Delta f > \epsilon_{1}$ or $\Delta x > \epsilon_{2}$}
		
		\STATE \textbf{solve the convex quadratic program \eqref{prob:qcp} to obtain an optimal solution $y^k$;}
		\STATE $d^k\leftarrow y^{k}-x^{k}$;
		\IF{$A(y^{k})\subset A(x^{k})$ and $\langle \nabla f(y^{k}),d^k \rangle < 0$}
		\STATE update $x^{k+1}$ by the Armijo line search with initial $\alpha=\frac{\sqrt{2}}{\|d^k\|}$ from $y^{k}$;
		\ELSE
		\STATE $x^{k+1}\leftarrow y^{k}$;
		\ENDIF
		\STATE $f^*\leftarrow f(x^{k+1})$; $x^*\leftarrow x^{k+1}$;
		\STATE $\Delta f \leftarrow |f^* - f(x^{k})|/(1+|f^*|)$; $\Delta x \leftarrow \|x^*-x^{k}\|/(1+\|x^*\|)$;
		\STATE $k \leftarrow k+1$;
		\ENDWHILE
	\end{algorithmic}
\end{algorithm}

\subsection{Boosted-DCA with DC-SOS decomposition}\label{sec:BDCA}
Similarly, we denote BDCA for the Boosted-DCA with DC-SOS decomposition. The only difference between BDCA and UBDCA is in the line 3, where the convex quadratic optimization \eqref{prob:qcp} in UBDCA is replaced by the convex quartic polynomial optimization \eqref{prob:Pk}. Again, by choosing any $\rho>0$, the DC-SOS decomposition given in  \eqref{eq:DCSOS-stronglyconvex} has strongly convex polynomial DC components. Hence, all convergence results in Theorem \ref{thm:convofBDCA} hold.

\section{Numerical simulation}\label{sec:Simu}
\paragraph{Experimental setup}
The numerical experiments are performed on a Dell Workstation equipped with $4$ Intel i$7$-$6820$HQ (8 cores), $2.70$GHz CPU and $32$ GB RAM. Our DC algorithms are implemented on MATLAB R2019a, namely \verb|MVSKOPT|, based on a \emph{DC optimization toolbox} (namely \verb|DCAM|) and a \emph{multivariate polynomial matrix modeling toolbox} (namely \verb|POLYLAB|). The \verb|DCAM| toolbox provides three main classes: DC function class (\verb|dcfunc|), DC programming problem class (\verb|dcp|), and DCA class (\verb|dca|), which can be used to model and solve a DC programming problem within few lines of codes. This toolkit is released as an open-source code under the license of MIT on Github \cite{DCAM}. The \verb|POLYLAB| toolbox is also developed by us to build efficiently multivariate polynomials, whose code is released as well on Github at \cite{Polylab}. We kindly welcome researchers for extensive tests of our codes in your applications, and we appreciate a lot of your feedbacks and contributions.

\paragraph{Data description}
Two datasets are tested in our experiments. 
\begin{enumerate}
	\item[$\bullet$]\textbf{Synthetic datasets}: the dataset is randomly generated by taking the number of assets $n$ in $\{4,6,\ldots, 20\}$ and the number of periods $T$ as $30$. For each $n$, we generate $3$ models in which the investor's preference $c$ is randomly chosen with $c\geq 0$. The returns rates $R_{i,t}$ are taken in $[-0.1, 0.4]$ for all $i\in \{1,\ldots,n\}, t\in \{1,\ldots,T\}$ using MATLAB function \verb|rand|. This dataset is used to test the performance of DCA and BDCA. 
	\item[$\bullet$]\textbf{Real datasets}: we take the weekly real return rates of $1151$ assets in Shanghai A shares ranged from January 2018 to December 2018 ($51$ weeks) downloaded from CSMAR \url{http://www.gtarsc.com/} database. These data are used to analyze the optimal portfolios and plot efficient frontier on real stock market. 
\end{enumerate} 

\paragraph{High-order moment computation}
The input tensors of four moments (mean, covariance, co-skewness and co-kurtosis) are computed using the formulations \eqref{eq:mu}, \eqref{eq:Sigma}, \eqref{eq:S} and \eqref{eq:K}. The ``curse of dimensionality" is a crucial problem to construct MVSK models. Three important issues and the proposed improvements are needed to be noted:

\begin{enumerate}
	\item[$\bullet$] \textit{Tensor sparsity issue}: We have explained in our previous work \cite{niu2011} that the moments and co-moments are often non-zeros which yields a dense nonconvex quartic polynomial optimization problem for MVSK model. Therefore, the number of monomials increases as fast as the order $O(\frac{n^4}{4!})$, since $\dim \R_4[x] = \binom{n+4}{4}$. This inherent difficulty makes it very time-consuming to generate dense high-order multivariate polynomials in MATLAB. Figure \ref{fig:probsizevsconstructtime} shows the performances of different modeling tools (including \verb|POLYLAB|\cite{Polylab}, \verb|YALMIP|\cite{yalmip}, \verb|SOSTOOLS|\cite{SOSTOOLS}, MATLAB Symbolic Math Toolbox using \verb|sym|, and MATLAB Optimization Toolbox using \verb|optimvar|) for constructing MVSK models. Clearly, \verb|POLYLAB| is much more fastest than the others, which is the reason why we use \verb|POLYLAB| to model polynomial optimization problems. Anyway, regarding to Figures \ref{subfig:3a} and \ref{subfig:3b}, the modeling time of \verb|POLYLAB| is still growing very quickly. Based on this observation, we can predict with a fourth order polynomial interpolation that the generating time of an MVSK model with $50$ variables could take about $1.59$ hours for \verb|POLYLAB|, $2.68$ hours for \verb|SOSTOOLS|, $4.72$ hours for \verb|Sym|, $7.48$ hours for \verb|YALMIP|, and $12.19$ hours for \verb|optimvar|. So the sparsity issue is one of the most important problems to limit the size of the MVSK model in practice. That is also the reason to develop \verb|POLYLAB| as a by-product in this project.
	\begin{figure}[ht!]
		\centering
		\subfigure[Number of assets v.s. cpu time]{
			\begin{minipage}{0.46\textwidth}\label{subfig:3a}
				\includegraphics[width=\textwidth]{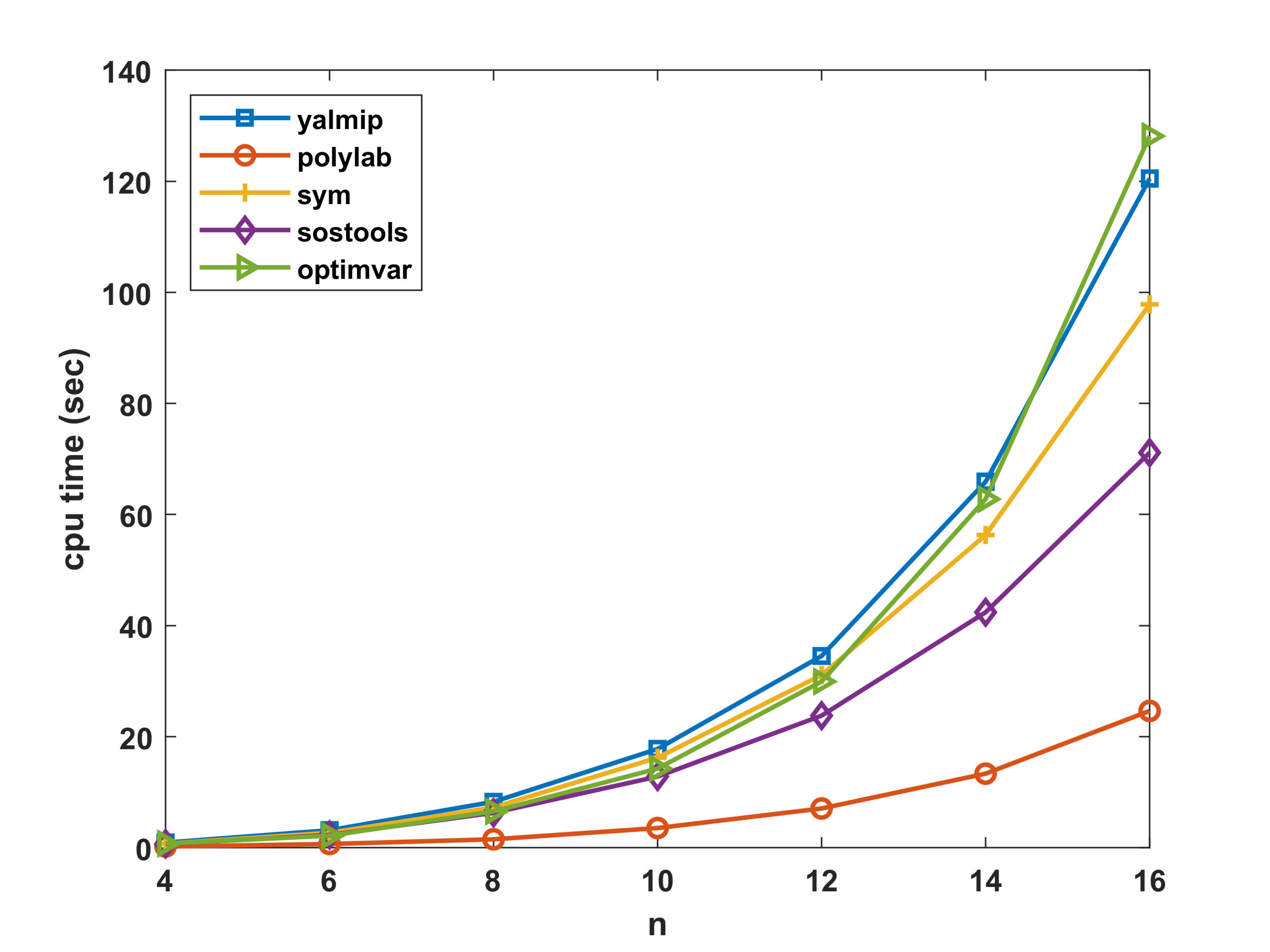}
			\end{minipage}
		}
		\subfigure[Number of assets v.s. log cpu time]{
			\begin{minipage}{0.46\textwidth}\label{subfig:3b}
				\includegraphics[width=\textwidth]{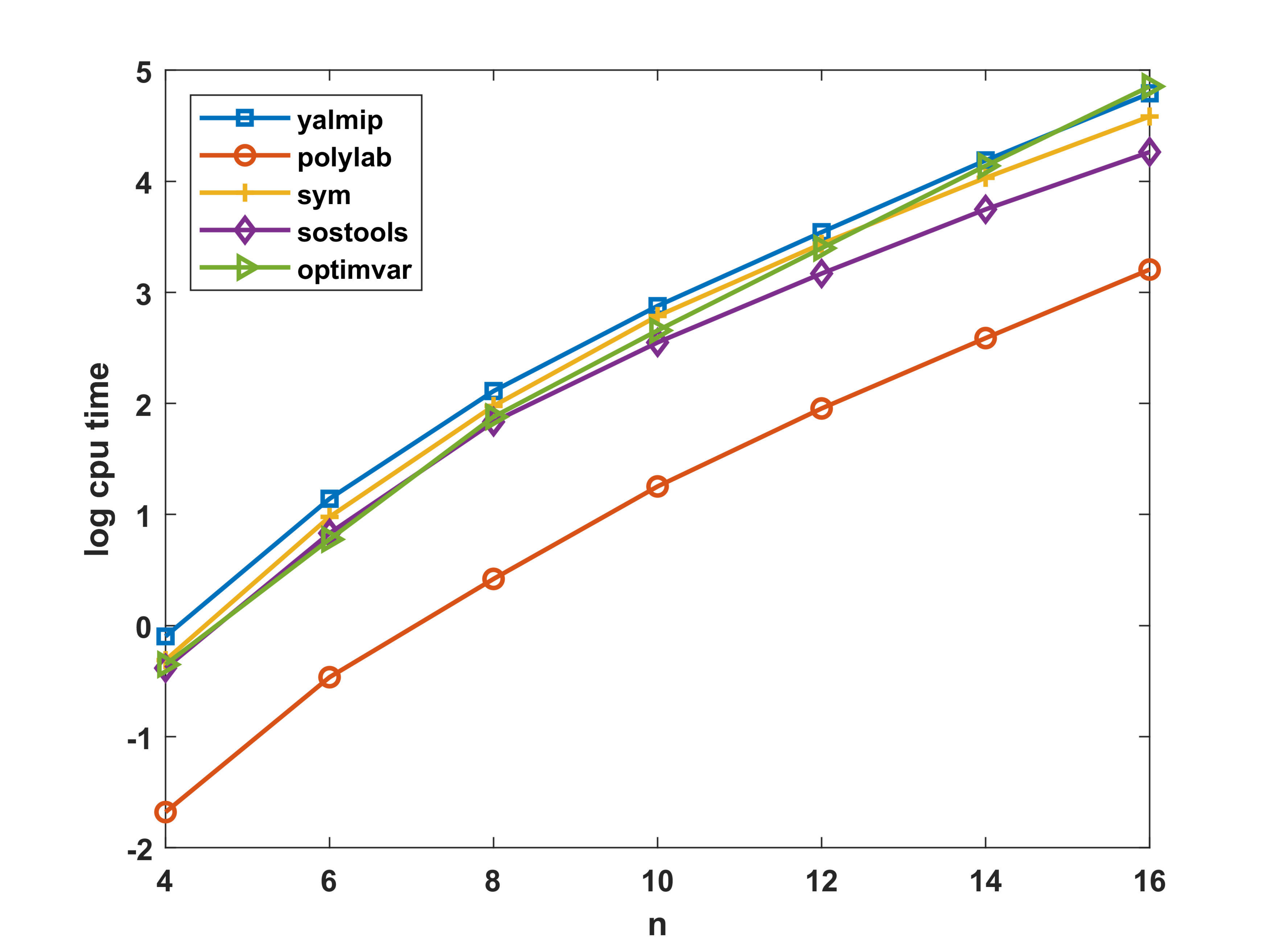}
			\end{minipage}
		}
		\caption{Performance of MVSK modelings using different modeling tools on MATLAB.}
		\label{fig:probsizevsconstructtime}
	\end{figure} 
	\item[$\bullet$] \emph{Computer memory issue}: Based on the symmetry of the moment tensors, it is unnecessary to allocate computer memories for saving all high-order moments and co-moments. E.g., due to the size limitation of the allowed MATLAB array, the construction of an $n^4$ co-kurtosis tensor with $n=300$ yields approximately $60.3$GB memories, that is intractable in our $32$GB RAM testing device. In our previous work \cite{niu2011}, we have tried using the Kronecker product and MATLAB \verb|mex| programming technique to compute a co-skewness (resp. co-kurtosis) tensor as an $n\times n^2$ (resp. $n\times n^3$) sparse matrix by retaining only the independent elements. Even though, saving huge amounts of moments data in memory is still very space-consuming. To overcome this difficulty, we propose computing the moments and co-moments entries Just-In-Time (namely, JIT technique) when they are needed, and without saving them in memory at all. Moreover, these moments and co-moments need only to be computed once when constructing the polynomial objective function, and the resulted polynomial has at most $\binom{n+4}{4}$ monomials which does not require large mount of memories. JIT technology is particularly useful to overcome the bottleneck of computer memory issue and improves a lot the numerical performance in very large-scale simulations.
	\item[$\bullet$] \emph{Gradient computing issue}: Concerning the gradient computation required in DCA and any first- and second-order method, it is computational expensive for exact gradient when $n$ is large, although $\nabla f(x)=-c_1 \mu + 2 c_2 \Sigma x - 3c_3 \hat{S} (x\otimes x) + 4c_3 \hat{K} (x\otimes x \otimes x)$ can be derived explicitly, see e.g. \cite{de2003incorporating}. Surprisingly, for a given point $x\in \R^n$, the computational complexity for $\nabla f(x)$ is even lower than $f(x)=-c_1 x^{\top}\mu + c_2 x^{\top}\Sigma x - c_3 x^{\top}\hat{S} (x\otimes x) + c_3 x^{\top}\hat{K} (x\otimes x \otimes x)$. Therefore, it is not hopeful to benefit from the numerical gradient such as $\partial_i f(x) \approx (f(x + \delta e_i) - f(x - \delta e_i))/2\delta$ with small $\delta>0$ for improving gradient evaluations.
\end{enumerate}

\subsection{Numerical tests with syntactic datastes}
\subsubsection{Numerical performance of different DC algorithms}\label{subsec:perfofdcalgos}
In this subsection, we will present the performance of our proposed four DC algorithms for MVSK model, namely UDCA (with commonly used DC decomposition); DCA (with DC-SOS decomposition), UBDCA and BDCA (Boosted UDCA and DCA). The polynomial convex optimization sub-problems required in DCA and BDCA are solved by \verb|KNITRO| \cite{knitro} (an implementation of an interior-point-method for convex optimization), which seems to be the fastest solver comparing with MATLAB \cite{MATLAB} \verb|fmincon|, IPOPT \cite{ipopt} and CVX \cite{cvx}; while the quadratic convex optimization sub-problems required in UDCA and UBDCA are solved by a strongly polynomial-time algorithm BPPPA, which can be easily implemented on MATLAB (see e.g., \cite{niu2011} for BPPPA). 

We test on the synthetic datasets with $n\in \{4,6,\ldots,20\}$ and $T=30$. For each $n$, we generate three problems with different investor's preferences ($c=(10,1,10,1)$ for risk-seeking, $c=(1,10,1,10)$ for risk-aversing, and $c=(10,10,10,10)$ for risk neutral). The initial point $x^0$ for DCA and BDCA is randomly generated in $\{0,1\}^n$. Note that based on our previous work \cite{niu2011}, good initial points for DCA can be estimated by solving the mean-variance model $\min\{-c_1m_1(x) + c_2 m_2(x): x\in \Omega\}$ since the third and fourth moments are extremely small compared with first and second moments, so that the solution of the mean-variance model is already very close to the final solution. Another good initialization (particularly useful for real market data) is to find a mean-variance portfolio closed to the naive $1/n$ portfolio by solving the convex quadratic optimization problem $\min\{-c_1m_1(x) + c_2 m_2(x) + \tau\|x - 1/n\|_2^2: x\in \Omega\}$ where $\tau>0$ is a weighting parameter. In this paper, we are not interested in the influence of initialization for DCA, but more focusing on the performance of difference DCA with same random initial point. The tolerances $\varepsilon_{1}= 10^{-6}$ and $\varepsilon_{2} = 10^{-4}$. In Armijo line search, the initial stepsize $\alpha = \frac{\sqrt{2}}{\|d\|}$, the reduction factor $\beta=0.5$, the parameter $\sigma = 10^{-3}$, and the stopping tolerance for line search $\varepsilon=10^{-8}$.  

Table \ref{tab:performcomp} summarizes some details of our tested synthetic models and their numerical results obtained by different DC algorithms (DCA, BDCA, UDCA and UBDCA). Some labels are explained as follows: Labels for MVSK models include the number of assets (n), the number of period (T), and the number of monomials (monos); Labels for numerical results of DC algorithms include the number of iterations (iter), the solution time (time), and the objective value (obj). We plot their numerical results in Figure \ref{fig:performcomp_time} in which the horizontal axis is the number of assets $n$, the vertical axis in Figure \ref{subfig:4a} is the average solution time of the three tested models (with different type of investor's preference) for the same number of assets $n$, and the vertical axis in Figure \ref{subfig:4b} is the logarithm of the average solution time.
\begin{table}[h]
	\caption{Numerical results of DCA, BDCA, UDCA and UBDCA with parameters $\varepsilon_{1}= 10^{-6}$, $\varepsilon_{2} = 10^{-4}$, $\alpha = \frac{\sqrt{2}}{\|d\|}$, $\beta=0.5$, $\sigma = 10^{-3}$, and $\varepsilon=10^{-8}$}
	\label{tab:performcomp}
	\centering
	\begin{adjustbox}{angle=90} 
		\resizebox{\columnwidth}{!}{
			\begin{tabular}{c|c|c|ccc|ccc|ccc|ccc} \hline
				\multirow{2}{*}{n} & \multirow{2}{*}{T} & \multirow{2}{*}{monos} & \multicolumn{3}{c|}{DCA} & \multicolumn{3}{c|}{BDCA} & \multicolumn{3}{c|}{UDCA} & \multicolumn{3}{c}{UBDCA} \\
				\cline{4-15}
				& & & iter & time(sec.) & obj & iter & time(sec.) & obj& iter & time(sec.) & obj& iter & time(sec.) & obj\\
				\hline\hline
				$4$ & $30$ & $69$ & $2$ & $0.01$ & $-1.980e+00$ & $2$ & $0.02$ & $-1.980e+00$ & $4$ & $0.00$ & $-1.980e+00$ & $3$ & $0.02$ & $-1.980e+00$\\
				$4$ & $30$ & $69$ & $6$ & $0.03$ & $-9.947e-02$ & $5$ & $0.03$ & $-9.947e-02$ & $28$ & $0.01$ & $-9.946e-02$ & $7$ & $0.00$ & $-9.947e-02$\\
				$4$ & $30$ & $69$ & $5$ & $0.03$ & $-1.678e+00$ & $5$ & $0.03$ & $-1.678e+00$ & $17$ & $0.01$ & $-1.678e+00$ & $8$ & $0.00$ & $-1.678e+00$\\
				$6$ & $30$ & $209$ & $8$ & $0.07$ & $-1.813e+00$ & $7$ & $0.05$ & $-1.813e+00$ & $16$ & $0.01$ & $-1.813e+00$ & $8$ & $0.01$ & $-1.813e+00$\\
				$6$ & $30$ & $209$ & $7$ & $0.05$ & $-1.267e-01$ & $5$ & $0.04$ & $-1.267e-01$ & $51$ & $0.02$ & $-1.267e-01$ & $6$ & $0.00$ & $-1.267e-01$\\
				$6$ & $30$ & $209$ & $13$ & $0.09$ & $-1.529e+00$ & $8$ & $0.06$ & $-1.529e+00$ & $44$ & $0.02$ & $-1.529e+00$ & $14$ & $0.01$ & $-1.529e+00$\\
				$8$ & $30$ & $494$ & $8$ & $0.07$ & $-1.848e+00$ & $7$ & $0.06$ & $-1.848e+00$ & $18$ & $0.01$ & $-1.848e+00$ & $11$ & $0.01$ & $-1.848e+00$\\
				$8$ & $30$ & $494$ & $9$ & $0.10$ & $-1.528e-01$ & $7$ & $0.09$ & $-1.528e-01$ & $83$ & $0.06$ & $-1.528e-01$ & $10$ & $0.01$ & $-1.528e-01$\\
				$8$ & $30$ & $494$ & $15$ & $0.14$ & $-1.700e+00$ & $10$ & $0.11$ & $-1.700e+00$ & $72$ & $0.05$ & $-1.700e+00$ & $16$ & $0.02$ & $-1.700e+00$\\
				$10$ & $30$ & $1000$ & $57$ & $0.68$ & $-1.800e+00$ & $10$ & $0.13$ & $-1.800e+00$ & $124$ & $0.15$ & $-1.800e+00$ & $15$ & $0.03$ & $-1.800e+00$\\
				$10$ & $30$ & $1000$ & $11$ & $0.16$ & $-1.473e-01$ & $7$ & $0.12$ & $-1.473e-01$ & $186$ & $0.27$ & $-1.472e-01$ & $15$ & $0.04$ & $-1.472e-01$\\
				$10$ & $30$ & $1000$ & $24$ & $0.33$ & $-1.823e+00$ & $9$ & $0.13$ & $-1.823e+00$ & $137$ & $0.16$ & $-1.823e+00$ & $15$ & $0.03$ & $-1.823e+00$\\
				$12$ & $30$ & $1819$ & $36$ & $0.67$ & $-1.848e+00$ & $12$ & $0.23$ & $-1.848e+00$ & $135$ & $0.27$ & $-1.848e+00$ & $15$ & $0.05$ & $-1.848e+00$\\
				$12$ & $30$ & $1819$ & $16$ & $0.42$ & $-1.510e-01$ & $9$ & $0.24$ & $-1.510e-01$ & $246$ & $0.61$ & $-1.509e-01$ & $17$ & $0.07$ & $-1.510e-01$\\
				$12$ & $30$ & $1819$ & $8$ & $0.19$ & $-1.663e+00$ & $7$ & $0.16$ & $-1.663e+00$ & $131$ & $0.27$ & $-1.663e+00$ & $19$ & $0.06$ & $-1.663e+00$\\
				$14$ & $30$ & $3059$ & $78$ & $2.18$ & $-2.277e+00$ & $10$ & $0.31$ & $-2.277e+00$ & $203$ & $0.65$ & $-2.277e+00$ & $14$ & $0.07$ & $-2.277e+00$\\
				$14$ & $30$ & $3059$ & $23$ & $0.88$ & $-1.551e-01$ & $11$ & $0.47$ & $-1.551e-01$ & $361$ & $1.53$ & $-1.549e-01$ & $14$ & $0.10$ & $-1.550e-01$\\
				$14$ & $30$ & $3059$ & $31$ & $1.43$ & $-1.844e+00$ & $10$ & $0.45$ & $-1.844e+00$ & $192$ & $0.69$ & $-1.844e+00$ & $19$ & $0.11$ & $-1.844e+00$\\
				$16$ & $30$ & $4843$ & $32$ & $1.59$ & $-1.869e+00$ & $11$ & $0.60$ & $-1.869e+00$ & $115$ & $0.60$ & $-1.869e+00$ & $20$ & $0.17$ & $-1.869e+00$\\
				$16$ & $30$ & $4844$ & $24$ & $1.48$ & $-1.593e-01$ & $12$ & $0.77$ & $-1.593e-01$ & $508$ & $3.48$ & $-1.592e-01$ & $31$ & $0.35$ & $-1.593e-01$\\
				$16$ & $30$ & $4844$ & $79$ & $6.40$ & $-1.757e+00$ & $16$ & $1.29$ & $-1.757e+00$ & $487$ & $3.03$ & $-1.757e+00$ & $20$ & $0.19$ & $-1.757e+00$\\
				$18$ & $30$ & $7314$ & $12$ & $1.03$ & $-2.236e+00$ & $8$ & $0.69$ & $-2.236e+00$ & $46$ & $0.39$ & $-2.236e+00$ & $23$ & $0.34$ & $-2.236e+00$\\
				$18$ & $30$ & $7313$ & $25$ & $2.38$ & $-1.698e-01$ & $15$ & $1.76$ & $-1.698e-01$ & $453$ & $4.66$ & $-1.697e-01$ & $31$ & $0.59$ & $-1.698e-01$\\
				$18$ & $30$ & $7312$ & $101$ & $10.63$ & $-2.057e+00$ & $23$ & $2.70$ & $-2.057e+00$ & $516$ & $4.48$ & $-2.057e+00$ & $35$ & $0.51$ & $-2.057e+00$\\
				$20$ & $30$ & $10625$ & $120$ & $15.93$ & $-1.986e+00$ & $13$ & $1.99$ & $-1.986e+00$ & $345$ & $4.39$ & $-1.986e+00$ & $31$ & $0.61$ & $-1.986e+00$\\
				$20$ & $30$ & $10622$ & $63$ & $10.83$ & $-1.672e-01$ & $23$ & $4.20$ & $-1.672e-01$ & $793$ & $13.38$ & $-1.669e-01$ & $42$ & $1.12$ & $-1.672e-01$\\
				$20$ & $30$ & $10623$ & $59$ & $10.57$ & $-1.861e+00$ & $14$ & $2.76$ & $-1.861e+00$ & $523$ & $7.12$ & $-1.861e+00$ & $38$ & $0.87$ & $-1.861e+00$\\
				\hline
				\multicolumn{3}{c|}{average} & $32$ & $2.53$ & & $10$ & $0.72$ & & $216$ & $1.72$ & & $18$ & $0.20$ & \\
				\hline
		\end{tabular}}	
	\end{adjustbox} 
\end{table}

\begin{figure}[ht!]
	\centering
	\subfigure[$n$ v.s. average solution time]{
		\begin{minipage}{0.46\textwidth}\label{subfig:4a}
			\includegraphics[width=\textwidth]{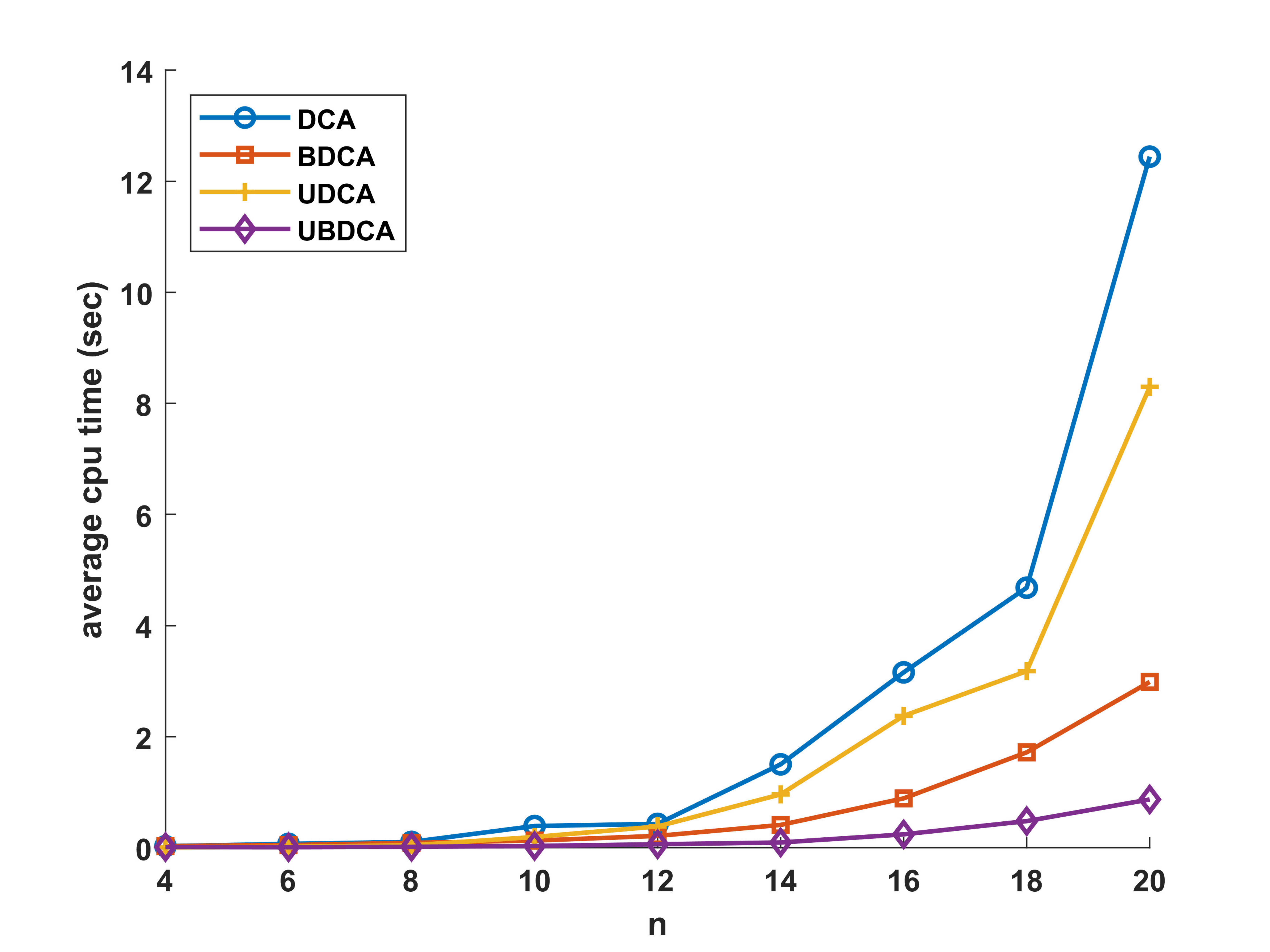}
		\end{minipage}
	}
	\subfigure[$n$ v.s. log average solution time]{
		\begin{minipage}{0.46\textwidth}\label{subfig:4b}
			\includegraphics[width=\textwidth]{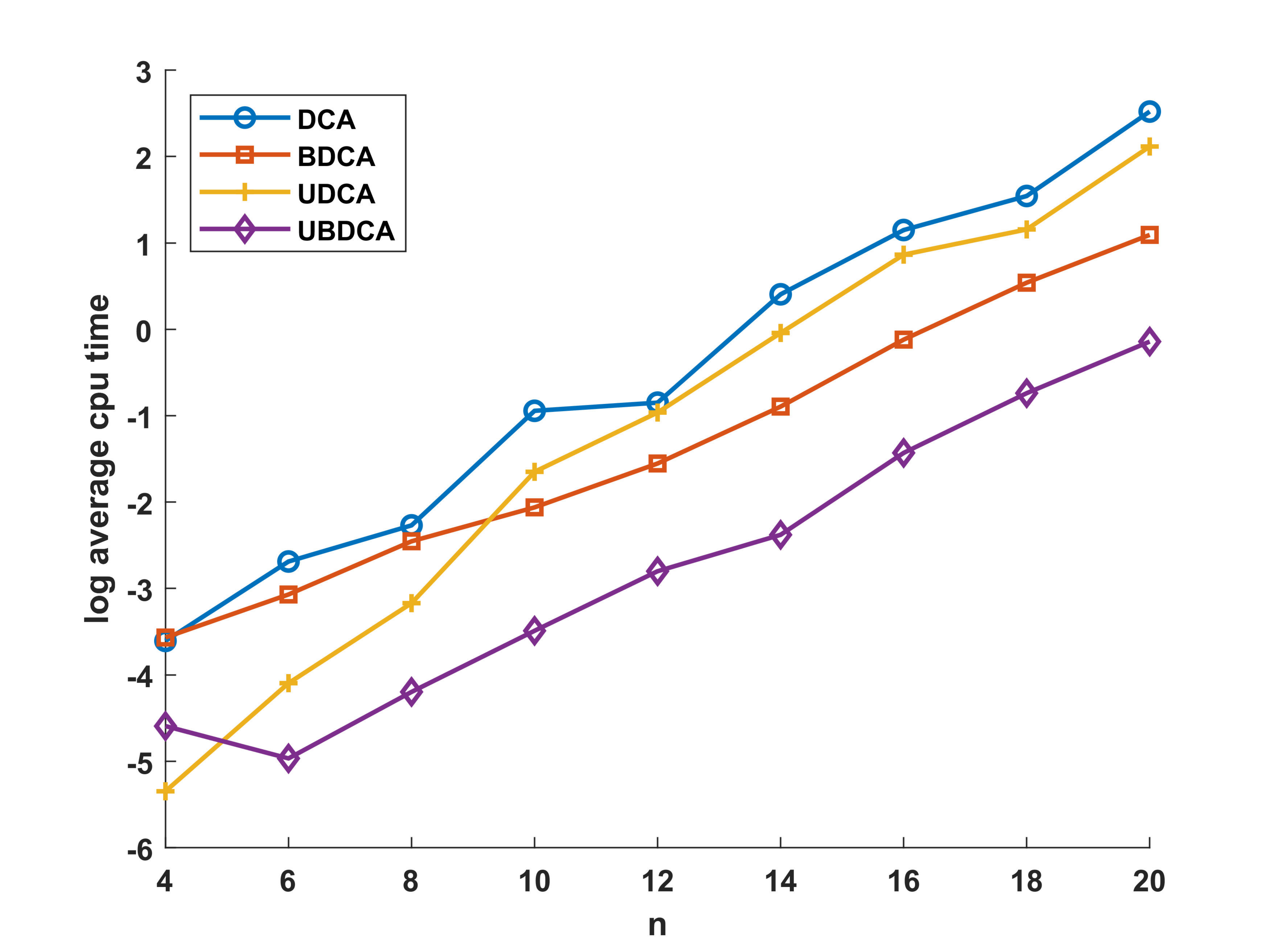}
		\end{minipage}
	}
	\caption{Performance of DCA, BDCA, UDCA and UBDCA for different number of assets $n$ with parameters $\varepsilon_{1}= 10^{-6}$, $\varepsilon_{2} = 10^{-4}$, $\alpha = \frac{\sqrt{2}}{\|d\|}$, $\beta=0.5$, $\sigma = 10^{-3}$, and $\varepsilon=10^{-8}$.}
	\label{fig:performcomp_time}
\end{figure} 
\paragraph{Comments on numerical results of four DCA-based algorithms:}
Concerning the computation time of different DC algorithms, we observe in Figures \ref{subfig:3a} and \ref{subfig:3b} that the boosted DCAs (UBDCA and BDCA) require less average number of iterations and result faster convergence than the classical DCAs (UDCA and DCA), thus the boosted DCAs indeed accelerate the convergence. Among the proposed four DCA-based algorithms, the fastest one is UBDCA, then BDCA, UDCA, and DCA. Based on the average cpu time in Table \ref{tab:performcomp}, we can estimate that UBDCA is about $3.6$ times faster than BDCA, $8$ times faster than UDCA and $12$ times faster than DCA. Figure \ref{subfig:3b} shows that the solution time for all DC algorithms seems increase exponentially with respect to the number of assets $n$, since the logarithm of the average solution time increases almost linearly with respect to the number of assets.  

Concerning the average number of iterations. It is observed in Table \ref{tab:performcomp} that the average number of iterations for the DC-SOS decomposition based DC algorithms (DCA and BDCA) is always smaller than the projective DC decomposition based DC algorithms (UDCA and UBDCA). UDCA requires $6.75$ times more iterations than DCA, and UBDCA requires $1.8$ times more iterations than BDCA. This observation demonstrates that the proposed DC-SOS decomposition indeed provides better convex over-approximation than the classical projective DC decomposition. 

Concerning the quality of the computed solutions obtained by different DC algorithms, we plot in Figure \ref{fig:performcomp_obj} that the objective values obtained by different DC algorithms substitute the objective values obtained by BDCA. Apparently, BDCA often obtains smaller objective values than the other three DC algorithms, but their differences are very small of order $O(10^{-4})$. Moreover, DC algorithms based on the DC-SOS decomposition (DCA and BDCA) often provide smaller objective values than DC algorithms based on the projective DC decomposition (UDCA and UBDCA). We believe that this is the benefit of DC-SOS decomposition technique, since the number of iterations using DC-SOS decomposition is always smaller than using the projective DC decomposition, which demonstrates that the DC-SOS decomposition can indeed provide better convex overestimator, thus DC-SOS is a promising approach, as expected, to provide high quality DC decomposition for polynomials.  
\begin{figure}[ht!]
	\centering
	\includegraphics[width=\linewidth]{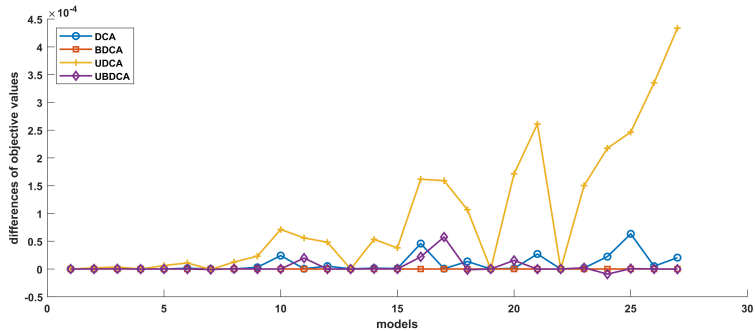}
	\caption{Differences between the objective values of four DC algorithms and the objective values of BDCA with respect to different MVSK models.}
	\label{fig:performcomp_obj}
\end{figure}

A possible way to understand why BDCA works better than DCA is that: DCA attempts to decrease $\|y^{k}-x^k\|$ during the iterations, while line search tries to increase $\|y^{k}-x^k\|$ ( since a new point $x^{k+1}$ is computed from $y^{k}$ such that $\|x^{k+1}-x^k\| > \|y^{k}-x^k\|$). Consequently, we sacrifice the decrease in $\|y^{k}-x^k\|$ to improve the decrease in the objective value with $f(x^{k+1})< f(y^{k})$, and to increase the potential to escape saddle points. This trade-off leads to improvements in the numerical results of BDCA. 

Note that the DC-SOS decomposition will produce high-order convex polynomial optimization subproblems, which usually takes more time to solve than the convex quadratic optimization problem. Therefore, a fast convex polynomial optimization solver dealing with CSOS polynomial objective function and linear constraints is extremely important to further improve the performance of DCA and BDCA. Classical nonlinear optimization approaches combing with Nesterov's acceleration technique may lead to superlinear convergence, and therefore deserves more attention. 

\subsubsection{Comparison with other methods}
In our previous work \cite{niu2011}, we have reported the numerical performances of several existing solvers Gloptipoly, LINGO and fmincon (SQP and Trust Region algorithms) for MVSK models. The interested reader can refer to \cite{niu2011} for more details. In this paper, we are going to compare our boosted DC algorithms (BDCA and UBDCA) with other existing solvers (KNITRO\cite{knitro}, FILTERSD\cite{filtersd}, IPOPT\cite{ipopt} and MATLAB FMINCON\cite{MATLAB}) which may represent the state-of-the-art level of local optimization solvers for large-scale nonlinear programming problems. Note that all of these methods are using POLYLAB as polynomial modeling tool in our tests. 

\paragraph{Comments on numerical results of other methods:} Numerical results of KNITRO, FILTERSD, IPOPT and FMINCON with default parameter settings are reported in Table \ref{tab:performothersolvers}. The comparisons of these solvers with the boosted DCAs (BDCA and UBDCA) are demonstrated in Figure \ref{fig:performcomp_time_others} in which the number of assets v.s. the logarithm of the average solution times of different methods is given in Figure \ref{subfig:5a}; while the differences between the objective values of all solvers and the objective values of KNITRO for all tested models are given in Figure \ref{subfig:5b}. Apparently, KNITRO often provides best numerical solutions within shortest solution time. 
\begin{table}[h]
	\caption{Numerical results of KNITRO, FILTERSD, IPOPT and FMINCON with default parameters.}
	\label{tab:performothersolvers}
	\centering
	\begin{adjustbox}{angle=90} 
		\resizebox{\columnwidth}{!}{
			\begin{tabular}{c|c|c|cc|cc|cc|cc} \hline
				\multirow{2}{*}{n} & \multirow{2}{*}{T} & \multirow{2}{*}{monos} & \multicolumn{2}{c|}{KNITRO} & \multicolumn{2}{c|}{FILTERSD} & \multicolumn{2}{c|}{IPOPT} & \multicolumn{2}{c}{FMINCON} \\
				\cline{4-11}
				& & & time(sec.) & obj & time(sec.) & obj & time(sec.) & obj & time(sec.) & obj\\
				\hline\hline
				$4$ & $30$ & $69$ & $0.01$ & $-1.980e+00$ & $0.01$ & $-1.980e+00$ & $0.01$ & $-1.980e+00$ & $0.01$ & $-1.980e+00$\\
				$4$ & $30$ & $69$ & $0.01$ & $-9.947e-02$ & $0.01$ & $-9.947e-02$ & $0.01$ & $-9.947e-02$ & $0.01$ & $-9.947e-02$\\
				$4$ & $30$ & $69$ & $0.01$ & $-1.678e+00$ & $0.01$ & $-1.678e+00$ & $0.01$ & $-1.678e+00$ & $0.01$ & $-1.678e+00$\\
				$6$ & $30$ & $209$ & $0.01$ & $-1.813e+00$ & $0.01$ & $-1.813e+00$ & $0.02$ & $-1.813e+00$ & $0.02$ & $-1.813e+00$\\
				$6$ & $30$ & $209$ & $0.01$ & $-1.267e-01$ & $0.01$ & $-1.267e-01$ & $0.01$ & $-1.267e-01$ & $0.02$ & $-1.267e-01$\\
				$6$ & $30$ & $209$ & $0.01$ & $-1.529e+00$ & $0.01$ & $-1.529e+00$ & $0.01$ & $-1.529e+00$ & $0.02$ & $-1.529e+00$\\
				$8$ & $30$ & $494$ & $0.01$ & $-1.848e+00$ & $0.01$ & $-1.848e+00$ & $0.02$ & $-1.848e+00$ & $0.02$ & $-1.848e+00$\\
				$8$ & $30$ & $494$ & $0.01$ & $-1.528e-01$ & $0.02$ & $-1.528e-01$ & $0.02$ & $-1.528e-01$ & $0.03$ & $-1.528e-01$\\
				$8$ & $30$ & $494$ & $0.01$ & $-1.700e+00$ & $0.01$ & $-1.700e+00$ & $0.02$ & $-1.700e+00$ & $0.02$ & $-1.700e+00$\\
				$10$ & $30$ & $1000$ & $0.02$ & $-1.800e+00$ & $0.01$ & $-1.800e+00$ & $0.02$ & $-1.800e+00$ & $0.03$ & $-1.800e+00$\\
				$10$ & $30$ & $1000$ & $0.02$ & $-1.473e-01$ & $0.05$ & $-1.473e-01$ & $0.04$ & $-1.473e-01$ & $0.04$ & $-1.473e-01$\\
				$10$ & $30$ & $1000$ & $0.01$ & $-1.823e+00$ & $0.01$ & $-1.823e+00$ & $0.04$ & $-1.823e+00$ & $0.03$ & $-1.823e+00$\\
				$12$ & $30$ & $1819$ & $0.02$ & $-1.848e+00$ & $0.01$ & $-1.848e+00$ & $0.03$ & $-1.848e+00$ & $0.04$ & $-1.848e+00$\\
				$12$ & $30$ & $1819$ & $0.03$ & $-1.510e-01$ & $0.08$ & $-1.510e-01$ & $0.07$ & $-1.510e-01$ & $0.08$ & $-1.510e-01$\\
				$12$ & $30$ & $1819$ & $0.02$ & $-1.663e+00$ & $0.02$ & $-1.663e+00$ & $0.03$ & $-1.663e+00$ & $0.04$ & $-1.663e+00$\\
				$14$ & $30$ & $3059$ & $0.03$ & $-2.277e+00$ & $0.03$ & $-2.277e+00$ & $0.05$ & $-2.277e+00$ & $0.05$ & $-2.277e+00$\\
				$14$ & $30$ & $3059$ & $0.05$ & $-1.551e-01$ & $0.18$ & $-1.551e-01$ & $0.09$ & $-1.551e-01$ & $0.11$ & $-1.551e-01$\\
				$14$ & $30$ & $3059$ & $0.04$ & $-1.844e+00$ & $0.05$ & $-1.844e+00$ & $0.06$ & $-1.844e+00$ & $0.06$ & $-1.844e+00$\\
				$16$ & $30$ & $4843$ & $0.04$ & $-1.869e+00$ & $0.02$ & $-1.869e+00$ & $0.08$ & $-1.869e+00$ & $0.06$ & $-1.869e+00$\\
				$16$ & $30$ & $4844$ & $0.08$ & $-1.593e-01$ & $0.19$ & $-1.593e-01$ & $0.16$ & $-1.593e-01$ & $0.15$ & $-1.593e-01$\\
				$16$ & $30$ & $4844$ & $0.07$ & $-1.757e+00$ & $0.13$ & $-1.757e+00$ & $0.13$ & $-1.757e+00$ & $0.11$ & $-1.757e+00$\\
				$18$ & $30$ & $7314$ & $0.07$ & $-2.236e+00$ & $0.03$ & $-2.236e+00$ & $0.16$ & $-2.236e+00$ & $0.11$ & $-2.236e+00$\\
				$18$ & $30$ & $7313$ & $0.15$ & $-1.698e-01$ & $0.39$ & $-1.698e-01$ & $0.26$ & $-1.698e-01$ & $0.23$ & $-1.698e-01$\\
				$18$ & $30$ & $7312$ & $0.08$ & $-2.057e+00$ & $0.05$ & $-2.057e+00$ & $0.18$ & $-2.057e+00$ & $0.15$ & $-2.057e+00$\\
				$20$ & $30$ & $10625$ & $0.11$ & $-1.986e+00$ & $0.08$ & $-1.986e+00$ & $0.24$ & $-1.986e+00$ & $0.17$ & $-1.986e+00$\\
				$20$ & $30$ & $10622$ & $0.19$ & $-1.672e-01$ & $0.49$ & $-1.672e-01$ & $0.41$ & $-1.672e-01$ & $0.37$ & $-1.672e-01$\\
				$20$ & $30$ & $10623$ & $0.14$ & $-1.861e+00$ & $0.23$ & $-1.861e+00$ & $0.25$ & $-1.861e+00$ & $0.20$ & $-1.861e+00$\\
				\hline
				\multicolumn{3}{c|}{average} & $0.05$ & & $0.08$ & & $0.09$ & & $0.08$ & \\
				\hline
		\end{tabular}}	
	\end{adjustbox} 
\end{table}

It worth noting that, except \verb|fmincon|, all other solvers are developed in either C/C++ or Fortran, and are invoked through MATLAB interfaces, thus they all perform pretty fast. We can observe that although BDCA and UBDCA are developed in MATLAB, they are still comparable with these state-of-the-art solvers. 
\begin{figure}[ht!]
	\centering
	\subfigure[$n$ v.s. log average solution time.]{
		\begin{minipage}{0.46\textwidth}\label{subfig:5a}
			\includegraphics[width=\textwidth]{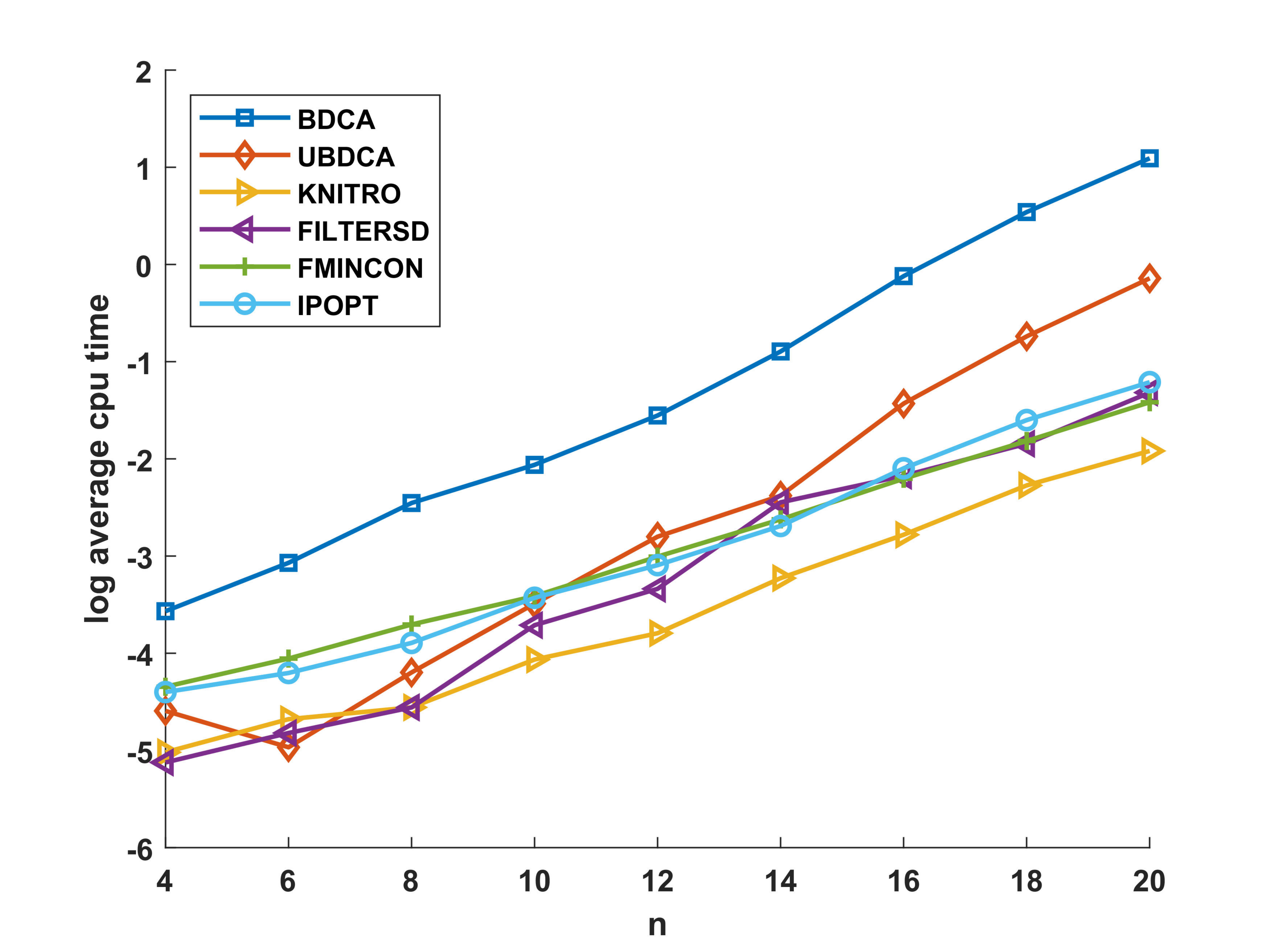}
		\end{minipage}
	}
	\subfigure[models v.s. objective values.]{
		\begin{minipage}{0.46\textwidth}\label{subfig:5b}
			\includegraphics[width=\textwidth]{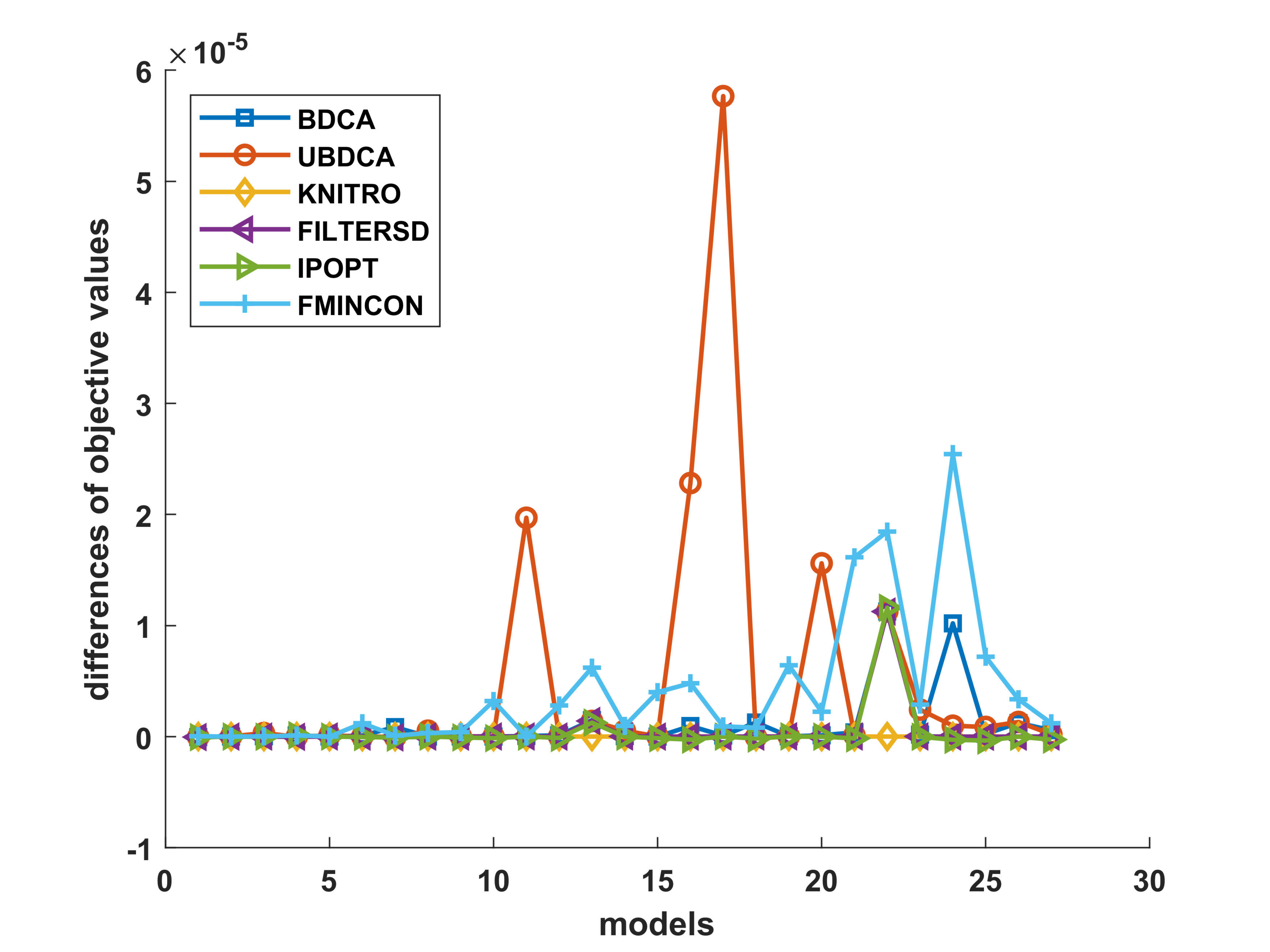}
		\end{minipage}
	}
	\caption{Comparisons among KNITRO, FILTERSD, IPOPT, FMINCON, BDCA and UBDCA.}
	\label{fig:performcomp_time_others}
\end{figure} 

\subsection{Numerical tests with real datasets}
In this subsection, we are interested in the shape of the portfolio efficient frontier. We are going to plot efficient frontiers for optimal portfolios provided by MVSK models for different types of investors (risk-neutral, risk-seeking and risk-aversing). We use real datasets of Shanghai A shares, and randomly select $10$ potentially `good' candidates among $1151$ assets, which are selected based on their positive average returns within $51$ weeks. Among these candidates, we establish MVSK models with desired expected return $m_1$ varying from $0$ to $0.4$ by step $0.001$. These models are in form of
$$\min \{ c_2 m_2(x) - c_3 m_3(x) + c_4 m_4(x): x\in \Omega, m_1(x) = r_k \}.$$
where $r_k$ is the desired expected return in $\{0,0.001,0.002,\ldots,0.4\}$, and the investor's preference $c$ is randomly chosen as follows: 
\begin{enumerate}
	\item[$\bullet$] for risk-neutral investor, we take $c_i\in [20,22], \forall i\in \{1,\ldots,4\}$; 
	\item[$\bullet$] for risk-aversing investor, we take $(c_2,c_4)\in [20,22]^2$ and $(c_1,c_3)\in [1,3]^2$;
	\item[$\bullet$] for risk-seeking investor, we take $(c_1,c_3)\in [20,22]^2$ and $(c_2,c_4)\in [1,3]^2$.
\end{enumerate} 
Then we use BDCA to solve these models for each type of investor, and obtain optimal portfolios to generate efficient frontiers. Figure \ref{fig:ef} illustrates efficient frontiers for three types of investors. Figure \ref{subfig:6a} presents the classical Mean-Variance efficient frontiers. As we expected, these frontiers are likely as portions of hyperbola. Figures \ref{subfig:6b}, \ref{subfig:6c} and \ref{subfig:6d} represent high-dimensional efficient frontiers (Mean-Variance-Skewness frontiers, Mean-Variance-Kurtosis frontiers, and Mean-Skewness-Kurtosis frontiers). We observe that the shape of efficient frontiers for three types of investors are quite different from each other. Risk-aversing efficient frontier has lower risk and lower expected returns; Risk-seeking efficient frontier has higher risk and higher expected returns; while risk-neutral efficient frontier is just between them. Moreover, the skewness and kurtosis of the optimal portfolios are both increasing when the mean (return) and variance (risk) are large enough, which indicates that a higher mean-variance optimal portfolio will have a higher probability of gains, but also have more uncertainty of returns. For more insights on the shapes of high-dimensional efficient frontiers, the reader can refer to \cite{de2003incorporating,de2004}.

\begin{figure}[ht!]
	\centering
	\subfigure[]{
		\begin{minipage}{0.46\textwidth}\label{subfig:6a}
			\includegraphics[width=\textwidth]{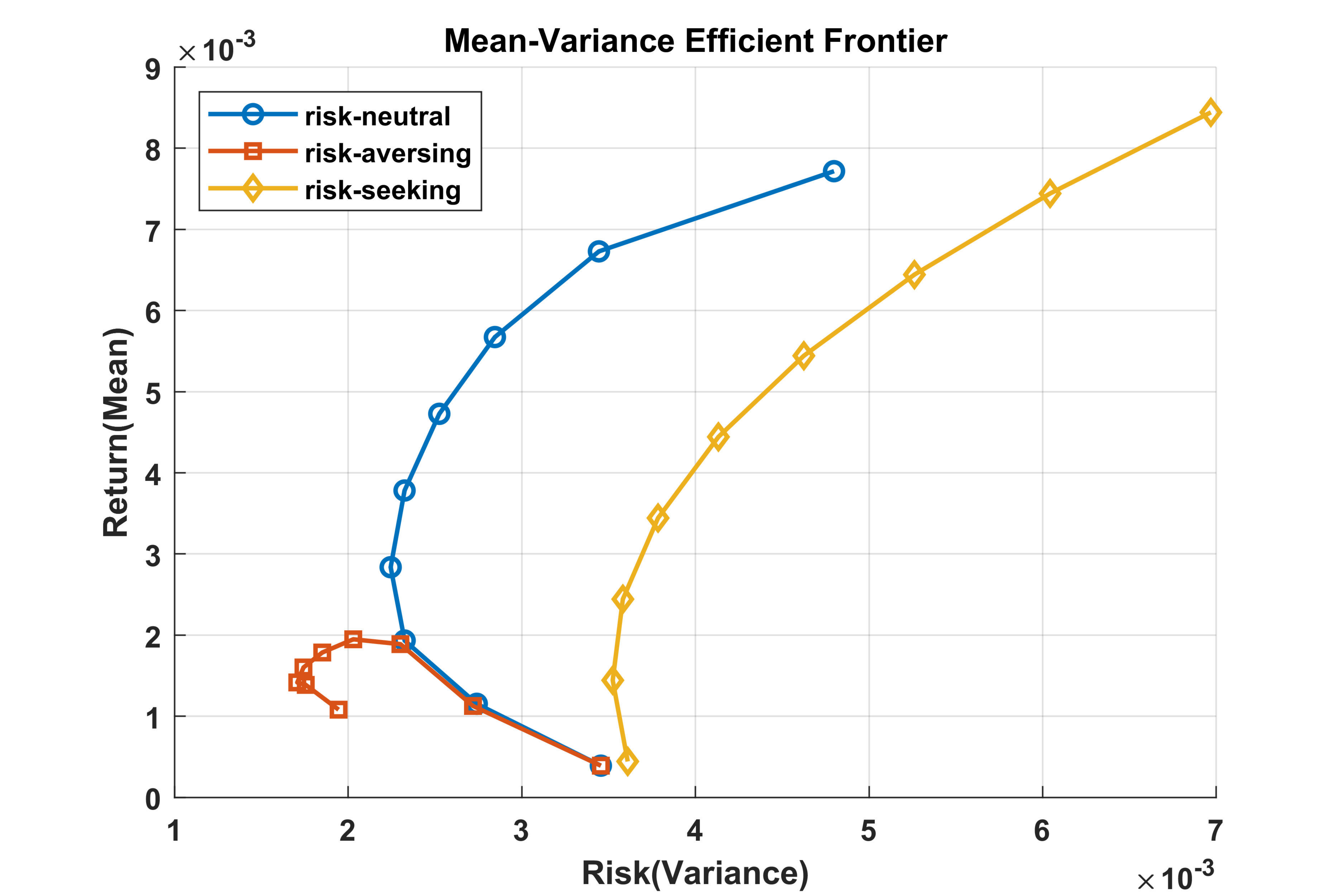}
		\end{minipage}
	}
	\subfigure[]{
		\begin{minipage}{0.46\textwidth}\label{subfig:6b}
			\includegraphics[width=\textwidth]{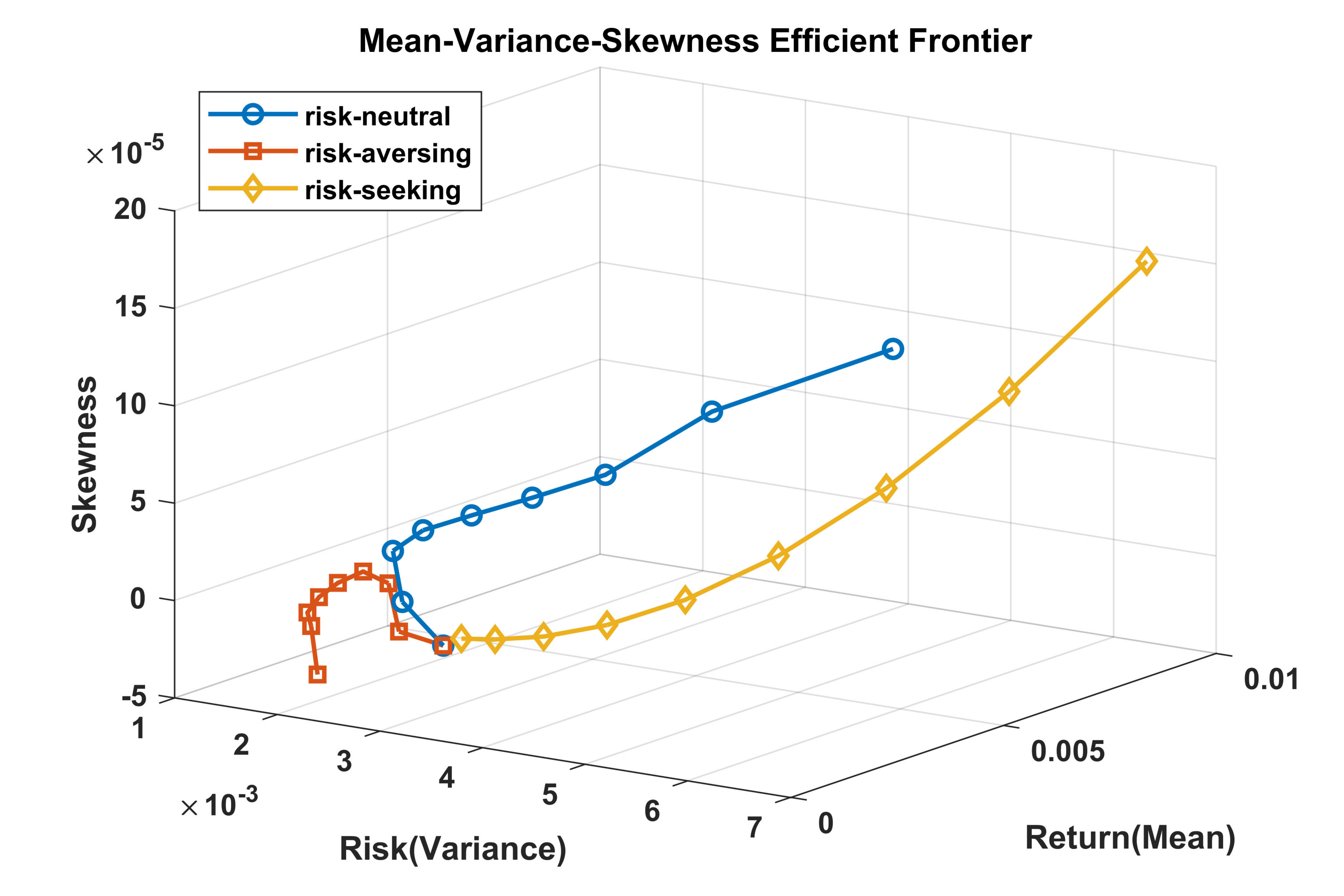}
		\end{minipage}
	}
	\subfigure[]{
		\begin{minipage}{0.46\textwidth}\label{subfig:6c}
			\includegraphics[width=\textwidth]{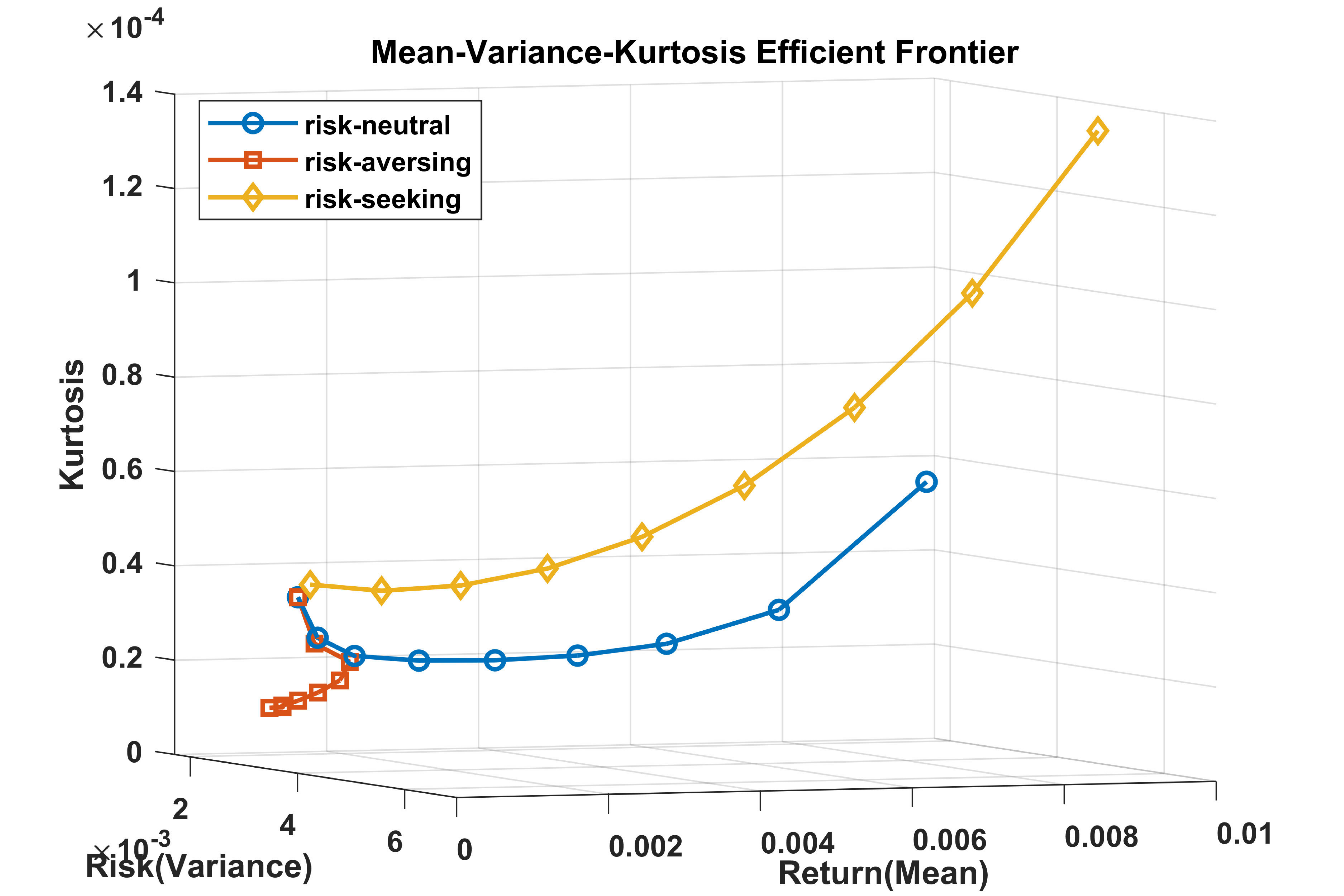}
		\end{minipage}
	}
	\subfigure[]{
		\begin{minipage}{0.46\textwidth}\label{subfig:6d}
			\includegraphics[width=\textwidth]{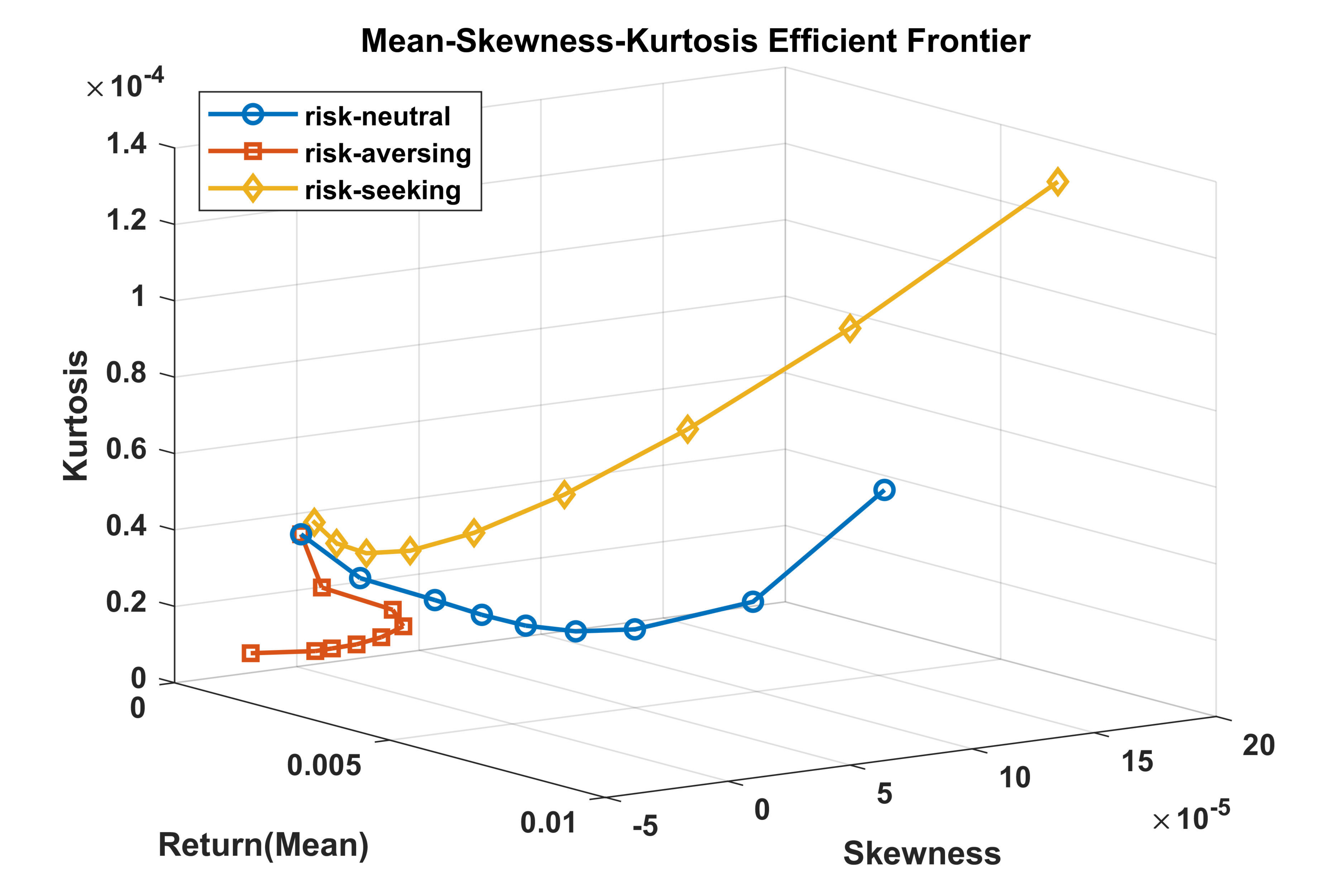}
		\end{minipage}
	}
	\caption{High dimensional portfolio efficient frontiers for risk-neutral, risk-aversing and risk-seeking cases.}
	\label{fig:ef}
\end{figure}

\section{Conclusion and perspectives}\label{sec:Conclu}
In this paper, we proposed several DC programming approaches for solving the high-order moment portfolio optimization problem (mean-variance-skewness-kurtosis model). We first reformulated the NVSK model as a DC programming problem based on the projective DC decomposition and DC-SOS decomposition, then developed corresponding DC algorithms, namely UDCA and DCA, to find local optimal solutions. Acceleration techniques for DCA based on DC descent direction and Armijo-type line search, namely Boosted-DCA, are investigated. Both smooth and nonsmooth DC programs with convex constraints are considered. Then, we use Boosted-DCA on MVSK model and established respectively to BDCA (Boosted-DCA with DC-SOS decomposition) and UBDCA (Boosted-DCA with the projective DC decomposition). Numerical simulations comparing DCA, BDCA, UBDCA and UDCA demonstrated that the boosted-DCA (BDCA and UBDCA) indeed accelerated the convergence of DCA (DCA and UDCA). Moreover, numerical results also showed that DC-SOS is a promising approach to provide high quality DC decomposition for polynomials.      

Our future works are varies. (1) a faster convex optimization solver for convex polynomial optimization subproblems based on DC-SOS over linear constraints is extremely important. It deserves more attention on developing more efficient convex optimization algorithm based on the very specific structure of the subproblem. There are two possible ideas. The first idea is the quadratic reformulation technique for DC-SOS decomposition. It is possible to formulate the subproblem with the particular DC-SOS decomposition given in this paper, as a convex quadratic program within convex quadratic constrains, by introducing additional variables to SOS terms. Then, convex quadratic optimization solvers (e.g., CPLEX, GUROBI and Mosek) can be applied to solve the subproblem, which is expected to be empirically more efficient than using a general nonlinear optimization solver (e.g., KNITRO and IPOPT). The second idea is to use accelerated convex optimization algorithms such as FISTA or accelerated interior point methods and Newton methods (based on Nesterov's acceleration or heavy ball acceleration) to solve the convex subproblem related to DC-SOS decomposition. (2) There is no need to completely solve the convex optimization subproblem in DCA. It suffices to find a better feasible solution $y^{k}\in \Omega$ verifying $f(y^{k})<f(x^k)$. E.g., we can use a descent algorithm to search for a better feasible point $y^{k}$ from $x^k$ without finding the global minimizer of the convex optimization subproblem. This idea, namely \emph{partial solution strategy}, has demonstrated to be useful for achieving empirically faster convergence in large-scale DC optimization problems (see e.g., \cite{niu2014}), whose convergence theorem has to be established. (3) Convergence theorem for Boosted-DCA applied to general DC program with DC set deserves more investigation, and the convergence rates for Boosted-DCA under different assumptions need more efforts. Researches in these directions will be reported subsequently.  

\section*{Acknowledgments.}
The first author is supported by the National Natural Science Foundation of China (Grant 11601327).

%


%
%
%

\begin{APPENDICES}
\section{\textbf{Proof of Theorem \ref{thm:convtod-stationaryforsmoothdcp}}}\label{appendix:A_thm:convtod-stationaryforsmoothdcp}
	The convex constrained DC program is equivalent to the standard form:
	$$\min_x\{ (g+\chi_{\mathcal{C}}) (x) - h(x)\}.$$ Applying DCA to the standard DC form with an initial point $x^0\in \R^n$ will generate a sequence $\{x^k\}$ such that $x^k\in \mathcal{C}, \forall k=1,2,\ldots$. Let $x^*$ be a limit point of the sequence $\{x^k\}$, then $x^*\in \mathcal{C}$ since $\mathcal{C}$ is closed. Based on the general convergence theory of DCA (see e.g., \cite{pham1997}), $x^*$ is a DC critical point of the standard DC program, i.e.,  
	\begin{equation}
	\label{eq:thm1-01}
	\emptyset \neq \partial (g+\chi_{\mathcal{C}}) (x^*) \cap \partial h(x^*) = \partial (g+\chi_{\mathcal{C}}) (x^*) \cap \{\nabla h(x^*)\},
	\end{equation}
	where the last equality is based on the differentiability of $h$.\\
	It follows that
	$$\emptyset \neq \{\nabla h(x^*)\} = \partial h(x^*) \subset   \partial (g+\chi_{\mathcal{C}}) (x^*),$$
	which proves the strongly DC criticality of $x^*$. 
	
\section{\textbf{The computation of DC decomposition for $m_3$ based on DC-SOS}}\label{appendix:B_m3}
Firstly, DC decompositions based on parity DC-SOS for monomials $x_i^2x_k$, $x_i x_j x_k$ and $x_i^3$ in $m_3$ and the gradients are computed as follows:

\noindent $\rhd$ \underline{For $x_i^3, \forall i\in \NN$}: a DC decomposition for $x_i^3$ on $\R^n_+$ is
$x_i^3 = g_i(x) - h_i(x)$
where
\begin{equation}\label{eq:gi&hi}
g_i(x) = x_i^3, h_i(x) = 0
\end{equation}
are both convex functions on $\R^n_+$. Their gradients are   
\begin{equation}\label{eq:dgi&dhi}
\nabla g_i(x) = 3x_i^2 e_i,\nabla h_i(x) = 0_{\R^n}.
\end{equation}
Note that we do not introduce a DC-SOS decomposition for $x_i^3$ since it is already convex, although it is obviously possible to introduce a DC-SOS decomposition for $x_i^3$ by replacing $x_k$ by $x_i$ in the first case $x_i^2 x_k$.

\noindent $\rhd$  \underline{For $x_i^2x_k, \forall (i,k)\in \PP$}: a parity DC-SOS decomposition is
\begin{align*}
x_i^2x_k = &\frac{1}{4}(x_i^2 - 0) \left((x_k+1)^2 - (x_k-1)^2\right) =  g_{i,k}(x) - h_{i,k}(x)
\end{align*}
where 
\begin{equation}
\label{eq:gik}
g_{i,k}(x) = \frac{1}{8} \left[\left(x_i^2 + (x_k+1 )^2 \right)^2+ (x_k-1)^4 \right],
\end{equation}
\begin{equation}
\label{eq:hik}
h_{i,k}(x)=\frac{1}{8}\left[(x_k+1)^4 + \left(x_i^2+(x_k-1)^2\right)^2\right]
\end{equation}
are CSOS on $\R^n$. Let $e_i$ denote the $i$-th unit vector of $\R^n$, then
\begin{equation}
\label{eq:dgik}
\nabla g_{i,k}(x)= \frac{1}{2}{{x}_{i}}\,\left( {{{x}_{k}}^{2}}+2{{x}_{k}}+{{{x}_{i}}^{2}}+1\right) e_i + \frac{1}{2}(2{{{x}_{k}}^{3}}+{{{x}_{i}}^{2}}\,{{x}_{k}}+6{{x}_{k}}+{{{x}_{i}}^{2}}) e_k,
\end{equation}
\begin{equation}
\label{eq:dhik}
\nabla h_{i,k}(x)= \frac{1}{2}{{x}_{i}}\,\left( {{{x}_{k}}^{2}}-2{{x}_{k}}+{{{x}_{i}}^{2}}+1\right) e_i + \frac{1}{2}(2{{{x}_{k}}^{3}}+{{{x}_{i}}^{2}}\,{{x}_{k}}+6{{x}_{k}}-{{{x}_{i}}^{2}}) e_k.
\end{equation}

\noindent $\rhd$ \underline{For $x_i x_j x_k, \forall (i,j,k)\in \QQ$}: a DC-SOS decomposition is given similarly as
\begin{align*}
x_i x_j x_k = & (x_i x_j) (x_k) =  g_{i,j,k}(x) - h_{i,j,k}(x),
\end{align*}
where 
\begin{equation}\label{eq:gijk}
g_{i,j,k}(x) = \frac{1}{32} \left[\left((x_i+x_j)^2 + (x_k+1 )^2 \right)^2+ \left((x_i-x_j)^2+(x_k-1)^2\right)^2 \right],
\end{equation}
\begin{equation}\label{eq:hijk}
h_{i,j,k}(x)=\frac{1}{32} \left[\left((x_i+x_j)^2 + (x_k-1 )^2 \right)^2+ \left((x_i-x_j)^2+(x_k+1)^2\right)^2 \right]
\end{equation}
are both CSOS on $\R^n$. Their gradients are
\begin{align}
\label{eq:dgijk}
\nabla g_{i,j,k}(x)=&\frac{1}{4}({{x}_{i}}\,{{{x}_{k}}^{2}}+2{{x}_{j}}\,{{x}_{k}}+3{{x}_{i}}\,{{{x}_{j}}^{2}}+{{{x}_{i}}^{3}}+{{x}_{i}})e_i + \frac{1}{4}({{x}_{j}}\,{{{x}_{k}}^{2}}+2{{x}_{i}}\,{{x}_{k}}+{{{x}_{j}}^{3}}+3{{{x}_{i}}^{2}}\,{{x}_{j}}+{{x}_{j}}) e_j \nonumber \\
& + \frac{1}{4}({{{x}_{k}}^{3}}+{{{x}_{j}}^{2}}\,{{x}_{k}}+{{{x}_{i}}^{2}}\,{{x}_{k}}+3{{x}_{k}}+2{{x}_{i}}\,{{x}_{j}}) e_k,
\end{align}
\begin{align}
\label{eq:dhijk}
\nabla h_{i,j,k}(x)= &\frac{1}{4}({{x}_{i}}\,{{{x}_{k}}^{2}}-2{{x}_{j}}\,{{x}_{k}}+3{{x}_{i}}\,{{{x}_{j}}^{2}}+{{{x}_{i}}^{3}}+{{x}_{i}})e_i + \frac{1}{4}({{x}_{j}}\,{{{x}_{k}}^{2}}-2{{x}_{i}}\,{{x}_{k}}+{{{x}_{j}}^{3}}+3{{{x}_{i}}^{2}}\,{{x}_{j}}+{{x}_{j}}) e_j \nonumber\\
& + \frac{1}{4}({{{x}_{k}}^{3}}+{{{x}_{j}}^{2}}\,{{x}_{k}}+{{{x}_{i}}^{2}}\,{{x}_{k}}+3{{x}_{k}}-2{{x}_{i}}\,{{x}_{j}}) e_k.
\end{align}	
$\rhd$ \underline{For $m_3$:} based on above three cases, a DC decomposition for $m_3$ is given by $$m_3(x) =  g_{m_3}(x) - h_{m_3}(x),$$
where \begin{align}\label{eq:gm3}
g_{m_3}(x) = & \sum_{i\in \IPS}S_{i,i,i}~g_i(x) + 3 \sum_{(i,j)\in \JPS}S_{i,i,k}~ g_{i,k}(x)  -  3 \sum_{(i,j)\in \JMS}S_{i,i,k}~ h_{i,k}(x) \nonumber\\ 
& + 6 \sum_{(i,j,k)\in \KPS}S_{i,j,k}~ g_{i,j,k}(x) - 6 \sum_{(i,j,k)\in \KMS}S_{i,j,k} ~h_{i,j,k}(x),
\end{align}
\begin{align}\label{eq:hm3}
h_{m_3}(x) = & - \sum_{i\in \IMS}S_{i,i,i}~g_i(x) + 3 \sum_{(i,j)\in \JPS}S_{i,i,k}~ h_{i,k}(x)  -  3 \sum_{(i,j)\in \JMS}S_{i,i,k}~ g_{i,k}(x) \nonumber\\ 
& + 6 \sum_{(i,j,k)\in \KPS}S_{i,j,k}~ h_{i,j,k}(x) - 6 \sum_{(i,j,k)\in \KMS}S_{i,j,k}~ g_{i,j,k}(x).
\end{align}
are both convex functions on $\Omega$, in which the index sets
$\IPS: = \{ i\in \NN: S_{i,i,i} >0 \}; \IMS:=\{ i\in \NN: S_{i,i,i} < 0 \};
\JPS: = \{ (i,k)\in \PP: S_{i,i,k}>0\}; \JMS: = \{ (i,k)\in \PP: S_{i,i,k}<0\}; \KPS: = \{ (i,j,k)\in \QQ: S_{i,j,k}>0\};$ and $\KMS :=  \{ (i,j,k)\in \QQ: S_{i,j,k}<0\};$ the functions $g_i, g_{i,k}, h_{i,k}, g_{i,j,k}$ and $h_{i,j,k}$ are given respectively in \eqref{eq:gi&hi}, \eqref{eq:gik}, \eqref{eq:hik}, \eqref{eq:gijk} and \eqref{eq:hijk}; and their gradients are computed accordingly.

\section{\textbf{The computation of DC decomposition for $m_4$ based on DC-SOS}}\label{appendix:C_m4}
DC decompositions of $5$ types of monomials $x_i^4$, $x_i^3x_k$, $x_i^2 x_k^2$, $x_i^2 x_j x_k$ and $x_i x_j x_k x_l$ based on DC-SOS and their gradients are computed as follows:

\noindent $\rhd$  \underline{For $x_i^4, \forall i\in \NN$}:  $x_i^4 =\widetilde{g}_i(x) - \widetilde{h}_i(x)$
where 
\begin{equation}
\label{eq:gitilde&hitilde}
\widetilde{g}_i(x) = x_i^4, \widetilde{h}_i(x) = 0,
\end{equation}
are convex functions on $\R^n$, and their gradients are
\begin{equation}
\label{eq:dgti&dhti}
\nabla \widetilde{g}_i(x) = 4x_i^3e_i, \nabla \widetilde{h}_i(x) = 0_{\R^n}.	
\end{equation}
\noindent $\rhd$  \underline{For $x_i^2 x_k^2, \forall (i,k)\in \PPH$:} $x_i^2 x_k^2 =  \widehat{g}_{i,k}(x) - \widehat{h}_{i,k}(x)$
where 
\begin{equation}
\label{eq:ghik&hhik}
\widehat{g}_{i,k}(x) = \frac{1}{2}(x_i^2 + x_k^2)^2;~\widehat{h}_{i,k}(x)=\frac{1}{2}(x_i^4 + x_k^4),
\end{equation}
are convex functions on $\R^n$, and their gradients are 
\begin{equation}
\label{eq:dghik}
\nabla \widehat{g}_{i,k}(x) = 2\left( {{{x}_{k}}^{2}}+{{{x}_{i}}^{2}}\right) ({{x}_{i}} e_i + {{x}_{k}} e_k), 
\end{equation}
\begin{equation}
\label{eq:dhhik}
\nabla \widehat{h}_{i,k}(x) = 2\left(x_i^3 e_i + x_k^3 e_k \right). 
\end{equation}
\noindent $\rhd$ \underline{ For $x_i^3x_k, \forall (i,k)\in \PP$:} 
$x_i^3x_k = \widetilde{g}_{i,k}(x) - \widetilde{h}_{i,k}(x)$
where 
\begin{equation}
\label{eq:giktilde}
\widetilde{g}_{i,k}(x) = \frac{1}{8} \left[\left(x_i^2 + (x_i+x_k )^2 \right)^2+ (x_i-x_k)^4 \right],
\end{equation}
\begin{equation}
\label{eq:hiktilde}
\widetilde{h}_{i,k}(x)=\frac{1}{8}\left[(x_i+x_k)^4 + \left(x_i^2+(x_i-x_k)^2\right)^2\right],
\end{equation}
are convex functions on $\R^n$, and their gradients are 
\begin{equation}
\label{eq:dgtik}
\nabla \widetilde{g}_{i,k}(x) = \frac{1}{2} {{x}_{i}}\,\left( 7{{{x}_{k}}^{2}}+3{{x}_{i}}\,{{x}_{k}}+5{{{x}_{i}}^{2}}\right)  e_i + \frac{1}{2}(2{{{x}_{k}}^{3}}+7{{{x}_{i}}^{2}}\,{{x}_{k}}+{{{x}_{i}}^{3}}) e_k,
\end{equation}
\begin{equation}
\label{eq:dhtik}
\nabla \widetilde{h}_{i,k}(x) =  \frac{1}{2}{{x}_{i}}\,\left( 7{{{x}_{k}}^{2}}-3{{x}_{i}}\,{{x}_{k}}+5{{{x}_{i}}^{2}}\right)  e_i + \frac{1}{2}(2{{{x}_{k}}^{3}}+7{{{x}_{i}}^{2}}\,{{x}_{k}}-{{{x}_{i}}^{3}}) e_k.
\end{equation}
\noindent $\rhd$   \underline{For $x_i^2 x_j x_k, \forall (i,j,k)\in \QQH$:} $x_i^2 x_j x_k = \widetilde{g}_{i,j,k}(x) - \widetilde{h}_{i,j,k}(x)$
where 
\begin{equation}
\label{eq:gtijk}
\widetilde{g}_{i,j,k}(x) = \frac{1}{8} \left[\left(x_i^2 + (x_j+x_k )^2 \right)^2+ (x_j-x_k)^4 \right];
\end{equation}
\begin{equation}
\label{eq:htijk}
\widetilde{h}_{i,j,k}(x)=\frac{1}{8}\left[(x_j+x_k)^4 + \left(x_i^2+(x_j-x_k)^2\right)^2\right]
\end{equation}
are convex functions on $\R^n$, and their gradients are 
\begin{align}
\label{eq:dgtijk}
\nabla \widetilde{g}_{i,j,k}(x) = &
\frac{1}{2}{{x}_{i}}\,\left( {{{x}_{k}}^{2}}+2{{x}_{j}}\,{{x}_{k}}+{{{x}_{j}}^{2}}+{{{x}_{i}}^{2}}\right) e_i \nonumber +\frac{1}{2}(6{{x}_{j}}\,{{{x}_{k}}^{2}}+{{{x}_{i}}^{2}}\,{{x}_{k}}+2{{{x}_{j}}^{3}}+{{{x}_{i}}^{2}}\,{{x}_{j}})e_j \nonumber\\
&+ \frac{1}{2} (2{{{x}_{k}}^{3}}+6{{{x}_{j}}^{2}}\,{{x}_{k}}+{{{x}_{i}}^{2}}\,{{x}_{k}}+{{{x}_{i}}^{2}}\,{{x}_{j}}) e_k,
\end{align}
\begin{align}
\label{eq:dhtijk}
\nabla \widetilde{h}_{i,j,k}(x) = &
\frac{1}{2}{{x}_{i}}\,\left( {{{x}_{k}}^{2}}-2{{x}_{j}}\,{{x}_{k}}+{{{x}_{j}}^{2}}+{{{x}_{i}}^{2}}\right) e_i \nonumber + \frac{1}{2}(6{{x}_{j}}\,{{{x}_{k}}^{2}}-{{{x}_{i}}^{2}}\,{{x}_{k}}+2{{{x}_{j}}^{3}}+{{{x}_{i}}^{2}}\,{{x}_{j}})e_j \nonumber\\
&+ \frac{1}{2} (2{{{x}_{k}}^{3}}+6{{{x}_{j}}^{2}}\,{{x}_{k}}+{{{x}_{i}}^{2}}\,{{x}_{k}}-{{{x}_{i}}^{2}}\,{{x}_{j}}) e_k.
\end{align}
\noindent $\rhd$  \underline{For $x_i x_j x_k x_l, \forall (i,j,k,l)\in \RR$:} $
x_i x_j x_k x_l= g_{i,j,k,l}(x) - h_{i,j,k,l}(x)$
where 
\begin{equation}
\label{eq:gijkl}
g_{i,j,k,l}(x) = \frac{1}{32} [\left((x_i+x_j)^2 + (x_k+x_l)^2 \right)^2  + \left((x_i-x_j)^2 + (x_k-x_l)^2\right)^2],
\end{equation}
\begin{equation}
\label{eq:hijkl}
h_{i,j,k,l}(x)=\frac{1}{32} [\left((x_i+x_j)^2 + (x_k-x_l)^2 \right)^2 + \left((x_i-x_j)^2 + (x_k+x_l)^2\right)^2]
\end{equation}
are convex functions on $\R^n$, and their gradients are
{\small\begin{align}
\label{eq:dgijkl}
\nabla g_{i,j,k,l}(x) = &
\frac{1}{4}\big[({{x}_{i}}\,{{{x}_{l}}^{2}}+2{{x}_{j}}\,{{x}_{k}}\,{{x}_{l}}+{{x}_{i}}\,{{{x}_{k}}^{2}}+3{{x}_{i}}\,{{{x}_{j}}^{2}}+{{{x}_{i}}^{3}}) e_i + ({{x}_{j}}\,{{{x}_{l}}^{2}}+2{{x}_{i}}\,{{x}_{k}}\,{{x}_{l}}+{{x}_{j}}\,{{{x}_{k}}^{2}}+{{{x}_{j}}^{3}}+3{{{x}_{i}}^{2}}\,{{x}_{j}})e_j \nonumber\\
&+  (3{{x}_{k}}\,{{{x}_{l}}^{2}}+2{{x}_{i}}\,{{x}_{j}}\,{{x}_{l}}+{{{x}_{k}}^{3}}+{{{x}_{j}}^{2}}\,{{x}_{k}}+{{{x}_{i}}^{2}}\,{{x}_{k}}) e_k +  ({{{x}_{l}}^{3}}+3{{{x}_{k}}^{2}}\,{{x}_{l}}+{{{x}_{j}}^{2}}\,{{x}_{l}}+{{{x}_{i}}^{2}}\,{{x}_{l}}+2{{x}_{i}}\,{{x}_{j}}\,{{x}_{k}}) e_l\big],
\end{align}}	
{\small\begin{align}
\label{eq:dhijkl}
\nabla h_{i,j,k,l}(x) = &
\frac{1}{4}\big[({{x}_{i}}\,{{{x}_{l}}^{2}}-2{{x}_{j}}\,{{x}_{k}}\,{{x}_{l}}+{{x}_{i}}\,{{{x}_{k}}^{2}}+3{{x}_{i}}\,{{{x}_{j}}^{2}}+{{{x}_{i}}^{3}}) e_i+ \frac{1}{4}({{x}_{j}}\,{{{x}_{l}}^{2}}-2{{x}_{i}}\,{{x}_{k}}\,{{x}_{l}}+{{x}_{j}}\,{{{x}_{k}}^{2}}+{{{x}_{j}}^{3}}+3{{{x}_{i}}^{2}}\,{{x}_{j}})e_j \nonumber\\
&+ \frac{1}{4} (3{{x}_{k}}\,{{{x}_{l}}^{2}}-2{{x}_{i}}\,{{x}_{j}}\,{{x}_{l}}+{{{x}_{k}}^{3}}+{{{x}_{j}}^{2}}\,{{x}_{k}}+{{{x}_{i}}^{2}}\,{{x}_{k}}) e_k+ \frac{1}{4} ({{{x}_{l}}^{3}}+3{{{x}_{k}}^{2}}\,{{x}_{l}}+{{{x}_{j}}^{2}}\,{{x}_{l}}+{{{x}_{i}}^{2}}\,{{x}_{l}}-2{{x}_{i}}\,{{x}_{j}}\,{{x}_{k}}) e_l\big].
\end{align}}
$\rhd$ \underline{For $m_4$:} based on above $5$ cases, a DC decomposition of $m_4$ is given by $$m_4(x) =  g_{m_4}(x) - h_{m_4}(x),$$
where $g_{m_4}$ and $h_{m_4}$ are convex functions on $\R^n$ defined by:
{\footnotesize \begin{align}\label{eq:gm4}
g_{m_4}(x) = & \sum_{i\in \IPK}K_{i,i,i,i}~\widetilde{g}_i(x) + 4 \sum_{(i,k)\in \JPK}K_{i,i,i,k}~\widetilde{g}_{i,k}(x) - 4 \sum_{(i,k)\in \JMK}K_{i,i,i,k}~\widetilde{h}_{i,k}(x)  + 6 \sum_{(i,k)\in\JHPK}K_{i,i,k,k}~\widehat{g}_{i,k}(x) \nonumber\\
&- 6 \sum_{(i,k)\in \JHMK}K_{i,i,k,k}~\widehat{h}_{i,k}(x) +  12 \sum_{(i,j,k)\in \KHPK}K_{i,i,j,k}~\widetilde{g}_{i,j,k}(x) -  12 \sum_{(i,j,k)\in \KHMK}K_{i,i,j,k}~\widetilde{h}_{i,j,k}(x) \nonumber\\
& + 24 \sum_{(i,j,k,l)\in \LPK}K_{i,j,k,l}~g_{i,j,k,l}(x) - 24 \sum_{(i,j,k,l)\in \LMK}K_{i,j,k,l}~h_{i,j,k,l}(x),
\end{align}}
{\footnotesize\begin{align}\label{eq:hm4}
h_{m_4}(x) = & - \sum_{i\in \IMK}K_{i,i,i,i}~\widetilde{g}_i(x) + 4 \sum_{(i,k)\in \JPK}K_{i,i,i,k}~\widetilde{h}_{i,k}(x) - 4 \sum_{(i,k)\in \JMK}K_{i,i,i,k}~\widetilde{g}_{i,k}(x) + 6 \sum_{(i,k)\in\JHPK}K_{i,i,k,k}~\widehat{h}_{i,k}(x) \nonumber \\
&- 6 \sum_{(i,k)\in \JHMK}K_{i,i,k,k}~\widehat{g}_{i,k}(x) +  12 \sum_{(i,j,k)\in \KHPK}K_{i,i,j,k}~\widetilde{h}_{i,j,k}(x) -  12 \sum_{(i,j,k)\in \KHMK}K_{i,i,j,k}~\widetilde{g}_{i,j,k}(x)\nonumber\\
& + 24 \sum_{(i,j,k,l)\in \LPK}K_{i,j,k,l}~h_{i,j,k,l}(x) - 24 \sum_{(i,j,k,l)\in \LMK}K_{i,j,k,l}~g_{i,j,k,l}(x)
\end{align}}
in which the index sets 
$\IPK: = \{ i\in \NN: K_{i,i,i,i} >0 \}; \IMK:=\{ i\in \NN: K_{i,i,i,i} < 0 \}; \JPK: = \{ (i,k)\in \PP: K_{i,i,i,k}>0\}; \JMK: = \{ (i,k)\in \PP: K_{i,i,i,k}<0\}; \JHPK: = \{ (i,k)\in \PPH: K_{i,i,k,k}>0\}; \JHMK: = \{ (i,k)\in \PPH: K_{i,i,k,k}<0\}; \KHPK: = \{ (i,j,k)\in \QQH: K_{i,i,j,k}>0\}; \KHMK :=  \{ (i,j,k)\in \QQH: K_{i,i,j,k}<0\}; \LPK: = \{ (i,j,k,l)\in \RR: K_{i,j,k,l}>0\};$ and $\LMK :=  \{ (i,j,k,l)\in \RR: K_{i,j,k,l}<0\};$ the functions $\widetilde{g}_i, \widetilde{g}_{i,k}, \widetilde{h}_{i,k}, \widehat{g}_{i,k}, \widehat{h}_{i,k}, \widetilde{g}_{i,j,k}, \widetilde{h}_{i,j,k}, g_{i,j,k,l}$ and $h_{i,j,k,l}$ are defined in \eqref{eq:gitilde&hitilde}, \eqref{eq:giktilde}, \eqref{eq:hiktilde}, \eqref{eq:ghik&hhik}, \eqref{eq:gtijk}, \eqref{eq:htijk}, \eqref{eq:gijkl} and \eqref{eq:hijkl} respectively; and their gradients are computed accordingly.

\section{\textbf{Proof of Theorem \ref{thm:convUDCA}}}\label{appendix:D_thm:convUDCA}
(i) By the strong convexity of $G$ (for large enough $\eta$) and the convexity of $H$, we have
\begin{equation}
	\label{eq:strongconvexityofG}
	G(x^{k}) \geq G(x^{k+1}) + \langle \nabla G(x^{k+1}), x^k - x^{k+1} \rangle + \frac{\eta}{2}\|x^{k}-x^{k+1}\|^2,
\end{equation} 
\begin{equation}
	\label{eq:convexityofH}
	H(x^{k+1}) \geq H(x^{k}) + \langle \nabla H(x^{k}), x^{k+1} - x^{k} \rangle.
\end{equation}
The first-order optimality condition for $x^{k+1}= \argmin \big\{ G(x)  - \langle x, \nabla H(x^k)\rangle : x\in \Omega \big\}$ gives
$$\nabla H(x^k) \in \nabla G(x^{k+1}) + N_{\Omega}(x^{k+1}),$$
where $N_{\Omega}(x^{k+1})$ is the normal cone of $\Omega$ at $x^{k+1}$.
Thus $\exists y\in N_{\Omega}(x^{k+1})$ such that 
\begin{equation}
	\label{eq:rel01}
	\nabla H(x^k) = \nabla G(x^{k+1}) + y
\end{equation}
and
\begin{equation}
	\label{eq:rel02}
	\langle y, x^{k+1}-x^k\rangle \geq 0.
\end{equation}
Then 
\begin{eqnarray}\label{eq:rel03}
	\langle \nabla H(x^{k}), x^{k+1} - x^{k} \rangle &\overset{\eqref{eq:rel01}}{=}& \langle \nabla G(x^{k+1}) + y, x^{k+1} - x^{k} \rangle \nonumber\\
	&=& \langle \nabla G(x^{k+1}) , x^{k+1} - x^{k} \rangle + \langle y, x^{k+1} - x^{k} \rangle \nonumber\\
	&\overset{\eqref{eq:rel02}}{\geq}& \langle \nabla G(x^{k+1}) , x^{k+1} - x^{k} \rangle.
\end{eqnarray}
It follows that
\begin{eqnarray}
	f(x^k)-f(x^{k+1}) &=& G(x^k) - G(x^{k+1}) - H(x^k) + H(x^{k+1})\nonumber\\
	&\overset{\eqref{eq:convexityofH},\eqref{eq:strongconvexityofG}}{\geq}& \langle \nabla G(x^{k+1}), x^k - x^{k+1} \rangle + \frac{\eta}{2}\|x^{k}-x^{k+1}\|^2 + \langle \nabla H(x^{k}), x^{k+1} - x^{k} \rangle\nonumber\\
	&\overset{\eqref{eq:rel03}}{\geq}& \frac{\eta}{2}\|x^{k}-x^{k+1}\|^2,\nonumber
\end{eqnarray} 
that is
\begin{equation}
	\label{eq:lowbdofdifff}
	\boxed{f(x^k)-f(x^{k+1}) \geq \frac{\eta}{2}\|x^{k}-x^{k+1}\|^2.}
\end{equation}
(ii) The sequence $\{x^k\}$ generated by UDCA is bounded since all $x^k, \forall k\in \N$ are included in the compact set $\Omega$. The sequence $\{f(x^k)\}$ is lower bounded since any polynomial $f$ over a compact set $\Omega$ is bounded. The inequality \eqref{eq:lowbdofdifff} implies that the sequence $\{f(x^k)\}$ is non-increasing. The convergence of the sequence $\{f(x^k)\}$ is follows immediately by the boundedness and the non-increasing of the sequence $\{f(x^k)\}$.\\
(iii) The convergence of $\{\|x^k-x^{k+1}\|\}$ to $0$ follows immediately form the inequality \eqref{eq:lowbdofdifff} and the convergence of the sequence $\{f(x^k)\}$.\\
(iv) Summing the inequality \eqref{eq:lowbdofdifff} for $k$ from $0$ to $N$, we obtain
$$\frac{\eta}{2}\sum_{k=0}^{N} \|x^k - x^{k+1}\|^2 \leq f(x^0) - f(x^{N+1}) \leq f(x^0) - \min_{x\in \Omega} f(x).$$
Hence, taking the limit when $N\to \infty$ and by the fact that $\min_{x\in \Omega} f(x)$ is finite, we obtain  
$$\frac{\eta}{2}\sum_{k=0}^{\infty} \|x^k  - x^{k+1}\|^2 \leq f(x^0) - \min_{x\in \Omega} f(x) < \infty.$$
Then $$\boxed{\sum_{k=0}^{\infty} \|x^k  - x^{k+1}\|^2 < \infty.}$$
The proof of $\sum_{k=0}^{\infty} \|x^k-x^{k+1}\|<\infty$ is similar to \cite[Theorem 3.4]{Lethi2018Convergence} based on the \emph{\L{}ojasiewicz subgradient inequality}. \\
Let $\mathcal{X}$ be the set of limit points of the sequence $\{x^k\}\subset \Omega$. The non-increasing and the convergence of the sequence $\{f(x^k)\}$ and the continuity of $f$ imply that $f(\mathcal{X})$ is a constant denoted by $f^*$. Let $\Psi(x) := f(x) +\chi_{\Omega}(x) - f^*$. Then, 
\begin{equation}
	\label{eq:nonegofpsi}
	\Psi(x^k) = f(x^k) - f^*\geq 0, \forall k=1,2,\ldots,
\end{equation} and \begin{equation}
	\label{eq:limito0}
	\lim_{k\to \infty} \Psi(x^k)  = \Psi(\mathcal{X}) =0.
\end{equation}
Clearly, $\Psi$ is subanalytic and then satisfies the {\L}ojasiewicz subgradient inequality (Theorem \ref{thm:Loja-ineq}) as: there exist a {\L}ojasiewicz component $\theta\in [0,1)$, a constant $M>0$ and $\epsilon>0$ such that 
$$|\Psi(x)-\Psi(\mathcal{X})|^{\theta} \leq M \|y\|, \forall x\in \cup_{z\in \mathcal{X}} B(z,\epsilon)\footnote{$B(z,\epsilon):=\{x: \|x-z\|<\epsilon\}$ stands for the open ball centered at $z$ with radius $\epsilon$.}, \forall y\in \partial^L \Psi(x),$$
It follows by $\Psi(\mathcal{X}) = 0$ that 
\begin{equation}
	\label{eq:L-ineq}
	|\Psi(x)|^{\theta} \leq M \|y\|, \forall x\in \cup_{z\in \mathcal{X}} B(z,\epsilon), \forall y\in \partial^L \Psi(x).
\end{equation}
Since $\mathcal{X}$ is the set of limit points of the sequence $\{x^k\}$, then, 
\begin{equation}
	\label{eq:boundednessofxk}
	\exists N>0, \forall k\geq N, x^{k}\in \cup_{z\in \mathcal{X}} B(z,\epsilon).
\end{equation}
Combining \eqref{eq:nonegofpsi}, \eqref{eq:L-ineq} and \eqref{eq:boundednessofxk}, we get  
\begin{equation}
	\label{eq:L-ineqxk}
	\boxed{\exists N>0, M>0, \theta\in [0,1), \forall k\geq N, \forall y\in \partial^L \Psi(x^k), \Psi(x^k)^{\theta} \leq M \|y\|.}
\end{equation}
By the concavity of the function $t\mapsto t^{1-\theta}$ ($\theta\in [0,1)$) on $(0,\infty)$, we have 
\begin{equation}
	\label{eq:concavity}
	(\Psi(x^{k}))^{1-\theta} - (\Psi(x^{k+1}))^{1-\theta} \geq (1-\theta) \Psi(x^{k})^{-\theta} (\Psi(x^{k}) - \Psi(x^{k+1})).
\end{equation}
Since the convex polynomial $H(x)= \frac{\eta}{2} \|x\|^2_2 - f(x)$ has locally Lipschicz continuous gradient over the compact set $\Omega$, then  
\begin{equation}\label{eq:L-smoothofh}
	\boxed{\exists L>0, \forall k\geq 1, \|\nabla H(x^k) - \nabla H(x^{k+1})\|\leq L \|x^k-x^{k+1}\|.}
\end{equation}
The first order optimality condition for the convex subproblem \eqref{prob:qcp} gives 
$$ \nabla H(x^k) \in \nabla G(x^{k+1}) +  N_{\Omega}(x^{k+1}).$$
Thus
$$\nabla H(x^k) - \nabla H(x^{k+1})  \in \nabla G(x^{k+1})  - \nabla H(x^{k+1}) + N_{\Omega}(x^{k+1})  = \partial^L \Psi (x^{k+1}),$$
that is 
\begin{equation}
	\label{eq:elementinpartialpsiwithh}
	\boxed{\forall k\geq 1, \nabla H(x^k) - \nabla H(x^{k+1}) \in \partial^L \Psi (x^{k+1}).}
\end{equation}
It follows from \eqref{eq:L-ineqxk} and \eqref{eq:elementinpartialpsiwithh} that 
\begin{equation}
	\label{eq:L-ineqatxkwithh}
	\boxed{\exists N>0, M>0, \theta\in [0,1),\forall k\geq N, \Psi(x^{k+1})^{\theta} \leq M \|\nabla H(x^k) - \nabla H(x^{k+1})\|.}
\end{equation}
By the Young's inequality $a\leq a^2/b + b/4$ with $a,b>0$, we have
\begin{equation}\label{eq:youngineq}
	\|x^{k} - x^{k+1}\| \leq \frac{\|x^{k} - x^{k+1}\|^2}{\|x^{k} - x^{k-1}\|} + \frac{\|x^{k} - x^{k-1}\|}{4}.
\end{equation}
Then, for each $k\geq N+1$:
\begin{eqnarray*}
	(\Psi(x^{k}))^{1-\theta} - (\Psi(x^{k+1}))^{1-\theta}
	&\overset{\eqref{eq:concavity}}{\geq}& (1-\theta) \Psi(x^{k})^{-\theta} (\Psi(x^{k}) - \Psi(x^{k+1})) \\
	&\overset{\eqref{eq:L-ineqatxkwithh}}{\geq}&(1-\theta) \frac{\Psi(x^{k}) - \Psi(x^{k+1})}{M\|\nabla H(x^{k}) - \nabla H(x^{k-1})\|} \\
	&\overset{\eqref{eq:lowbdofdifff}}{\geq}& \frac{\eta(1-\theta)}{2M} \frac{\|x^{k} - x^{k+1}\|^2}{\|\nabla H(x^{k}) - \nabla H(x^{k-1})\|}\\
	&\overset{\eqref{eq:L-smoothofh}}{\geq}& \frac{\eta(1-\theta)}{2ML} \frac{\|x^{k} - x^{k+1}\|^2}{\|x^{k} - x^{k-1}\|}\\
	&\overset{\eqref{eq:youngineq}}{\geq}& \frac{\eta(1-\theta)}{2ML} \left( \|x^{k} - x^{k+1}\| - \frac{1}{4}\|x^{k} - x^{k-1}\|\right),
\end{eqnarray*}
that is 
\begin{equation}\label{eq:ineq01}
	\boxed{\forall k\geq N+1, \|x^{k} - x^{k+1}\| \leq  \frac{1}{4}\|x^{k} - x^{k-1}\| + \frac{2ML}{\eta(1-\theta)} (\Psi(x^{k}))^{1-\theta} - (\Psi(x^{k+1}))^{1-\theta}).}
\end{equation}
Summing for $k$ from $N+1$ to $\infty$, we get
$$\frac{3}{4}\sum_{k=N+1}^{\infty} \|x^{k} - x^{k+1}\| \leq \frac{1}{4}\|x^{N+1} - x^{N}\| +  \frac{2ML}{\eta(1-\theta)}  (\Psi(x^{N+1}))^{1-\theta} < \infty,$$
that is 
$$\boxed{\sum_{k=0}^{\infty} \|x^{k} - x^{k+1}\| < \infty.}$$
(v)  The convergence of $\sum_{k=0}^{\infty} \|x^k-x^{k+1}\|$ indicates that $\{x^k\}$ is a Cauchy sequence which is convergent. Let $x^*$ denote the limit of $\{x^k\}$. Problem \eqref{prob:DC_univ} is a convex constrained DC program with continuously differentiable $H$, it follows from Theorem \ref{thm:convtod-stationaryforsmoothdcp} that $x^*$ is a strongly critical point and  
\begin{equation}\label{prob:subprob}
x^* = \argmin_{x\in \Omega} \{G(x) - \langle \nabla H(x^*), x\rangle\}.
\end{equation}
Then $x^*$ verifies the KKT conditions for problem \eqref{prob:subprob} as
$$\begin{cases}
\nabla G(x^*) - \nabla H(x^*) - \lambda - \mu e = \nabla f(x^*) - \lambda - \mu e = 0,\\
x^*\in \Omega,\\
\lambda\tran x^* = 0,\\
\lambda\geq 0, \mu\in \R,
\end{cases}$$
where $(\lambda,\mu)$ is the Lagrangian multiplier. The KKT conditions for problem \eqref{prob:subprob} are exactly the KKT conditions for the DC problem \eqref{P_mvsk}, i.e., $x^*$ is a KKT point of \eqref{P_mvsk}.

\section{\textbf{Proof of Theorem \ref{thm:decdirect}}} \label{appendix:E_thm:decdirect}
Under the assumption of constraint qualifications, the KKT conditions for the convex optimization problem \begin{equation}\label{prob:cvxsubprob}
y^k\in \min\{ g(x) - \langle x, \nabla h(x^k) \rangle : u(x)\leq 0, v(x)=0 \}
\end{equation}
reads
\begin{equation}\label{eq:kkt}
\left\{
\begin{array}{l}
\nabla g(y^{k}) - \nabla h(x^{k}) + \sum_{i=1}^{p}\lambda_i \nabla u_i(y^{k}) +  \sum_{i=1}^{q}\mu_j \nabla v_j(y^{k}) = 0,\\
u(y^{k})\leq 0, v(y^{k})=0,\\
\lambda_i u_i(y^{k}) = 0, i=1,\ldots,p,\\
\lambda \geq 0, \mu\in \R^{q},
\end{array}\right.
\end{equation}
where $(\lambda,\mu)\in \R_+^p\times \R^q$ is the Lagrangian multiplier. It follows from \eqref{eq:kkt} and $d^k:=y^{k}-x^{k}$ that
\begin{equation}
\label{eq:ineq}
\begin{aligned}
\langle \nabla f(y^{k}), d^k \rangle &=  \langle \nabla g(y^{k}) - \nabla h(y^{k}), d^k\rangle\\	
&=  \langle \nabla h(x^{k}) - \sum_{i=1}^{p}\lambda_i \nabla u_i(y^{k}) -  \sum_{i=1}^{q}\mu_j \nabla v_j(y^{k}) - \nabla h(y^{k}) ,d^k \rangle \\
&=  \underbrace{\langle \nabla h(x^{k}) - \nabla h(y^{k}) ,d^k  \rangle}_{(I)} - \underbrace{\langle \sum_{i=1}^{p}\lambda_i \nabla u_i(y^{k}) + \sum_{i=1}^{q}\mu_j \nabla v_j(y^{k}) ,d^k \rangle}_{(II)}.
\end{aligned}
\end{equation}
The sign of the part $(I)$ is determined by the monotonicity of $\nabla H$ since $h$ is convex, i.e., 
\begin{equation}
\label{eq:(I)}
(I) =  - \langle \nabla h(x^k) - \nabla h(y^{k}), d^k \rangle \leq 0.
\end{equation}
The sign of the part $(II)$ is determined by the convexity of $u$ and the affinity of $v$, i.e.,
\begin{equation}\label{eq:convexu&affinev}
\left\{
\begin{array}{ll}
u_i(y^{k}) - u_i(x^{k})&\leq \langle \nabla u_i(y^{k}), d^k \rangle, i=1,\ldots,p,\\
v_j(y^{k}) - v_j(x^{k})&= \langle \nabla v_j(y^{k}), d^k \rangle, j=1,\ldots,q.
\end{array}\right.
\end{equation}
Then it follows from \eqref{eq:kkt}, \eqref{eq:convexu&affinev}, $u(x^k)\leq 0$ and $v(x^k)=0$ that
\begin{equation}
\label{eq:(II)}
\begin{aligned}
(II) &= \sum_{i=1}^{p}\lambda_i \langle \nabla u_i(y^{k}), d^k \rangle + \sum_{i=1}^{q}\mu_j \langle \nabla v_j(y^{k}), d^k \rangle \\
&\geq \sum_{i=1}^{p}\lambda_i (u_i(y^{k}) - u_i(x^{k})) + \sum_{j=1}^{q}\mu_j (v_j(y^{k}) - v_j(x^{k}))\\
&= \sum_{i=1}^{p}\underbrace{\lambda_i u_i(y^{k})}_{=0} - \sum_{i=1}^{p}\underbrace{\lambda_i u_i(x^{k})}_{\leq 0} + \sum_{j=1}^{q}\underbrace{\mu_j v_j(y^{k})}_{=0} - \sum_{j=1}^{q}\underbrace{\mu_j v_j(x^{k})}_{=0} \geq 0.
\end{aligned}
\end{equation}
Combining \eqref{eq:ineq}, \eqref{eq:(I)} and \eqref{eq:(II)}, we get the required inequality
$$\boxed{\langle \nabla f(y^{k}),d^k \rangle  \leq 0.}$$

\section{\textbf{Proof of Theorem \ref{thm:decdirectbis}}} \label{appendix:F_thm:decdirectbis}
The $\rho$-strong convexity ($\rho>0$) of $h$ implies the strongly monotone of $\nabla h$ as   
\begin{equation}
\label{eq:strongcvx}	
\langle \nabla h(x^k) - \nabla h(y^{k}), x^k - y^{k} \rangle \geq \rho \|x^{k}-y^k \|^2.
\end{equation}
By analogue in Theorem \ref{thm:decdirect}, it follows from \eqref{eq:ineq}, \eqref{eq:(II)} and \eqref{eq:strongcvx} that
$$\langle \nabla f(y^{k}),d^{k} \rangle =  (I) - (II)\leq - \rho \|d^k\|^2,$$
which yields the required inequality $$\boxed{\langle \nabla f(y^{k}),d^k \rangle \leq - \rho \|d^k\|^2.}$$

\section{\textbf{Proof of Proposition \ref{prop:SCforDCdescentdirection}}}\label{appendix:prop:SCforDCdescentdirection} (i) Theorem \ref{thm:decdirect} indicates that $\langle \nabla f(y^k),d^k \rangle\leq 0$. If the inequality is strict, then $d^k\neq 0$ and by the Taylor expand of $f$ at $y^k+td^k$ with $t>0$, we have 
$$f(y^k+td^k) = f(y^k) + t \langle \nabla f(y^k), d^k \rangle + o(\|td^k\|).$$
Then
$$\frac{f(y^k+td^k) - f(y^k)}{t} = \langle \nabla f(y^k), d^k \rangle + \frac{\|d^k\|o(t)}{t}.$$
Taking $t\to 0^+$, we get 
$$\lim_{t\to 0^+} \frac{f(y^k+td^k) - f(y^k)}{t} = \langle \nabla f(y^k), d^k \rangle < 0,$$
implying that $\exists \eta_0>0$, $\forall t\in (0,\eta_0)$, 
$$f(y^k+td^k) < f(y^k).$$
Moreover, $d^k$ is a feasible direction implies that $\exists \eta_1>0, \forall t\in (0,\eta_1)$, 
$$ y^k+td^k \in \C.$$
Thus, taking $\eta = \min\{\eta_0,\eta_1\}$, we have $\forall t\in (0,\eta)$, 
$$f(y^k+td^k) < f(y^k) \text{ and } y^k+td^k \in \C.$$
That is, $d^k$ is a DC descent direction of $f$ at $y^k$ over $\C$. \\
(ii) Theorem \ref{thm:decdirectbis} with $\rho>0$ and $d^k\neq 0$ implies that $$\langle \nabla f(y^k),d^k \rangle\leq -\rho \|d^k\|^2<0.$$
Then we have exactly the case (i).

\section{\textbf{Proof of Proposition \ref{prop:feasdirect}}} \label{appendix:G_prop:feasdirect}
For any feasible direction $d^k$ and for any index $i\in A(y^{k})$, we have the inequality $\langle \nabla u_i (y^{k}), d^k\rangle \leq 0$. Replacing $d^k$ by $y^{k}-x^k$ and using the convexity of $u_i$, we get 
$$u_i(x^k) - u_i(y^{k})\geq  - \langle \nabla u_i (y^{k}), d^k\rangle \geq 0.$$
We obtain from the above inequality and $u_i(y^{k}) =0, \forall i\in A(y^{k})$ that 
\begin{equation}\label{eq:ineq1}
u_i(x^k)  \geq 0.
\end{equation}
Moreover, since $x^k$ is a feasible point, then $u_i(x^k)\leq 0$, combining with \eqref{eq:ineq1}, we obtain $u_i(x^k) = 0,$ which implies $i\in A(x^k)$. Thus $$\boxed{A(y^{k})\subset A(x^k).}$$ 

\section{\textbf{Proof of Theorem \ref{thm:necsuffcondfeasdirect}}}\label{appendix:H_thm:necsuffcondfeasdirect}
The necessary part is proved in Theorem \ref{prop:feasdirect}. For the sufficient part: the affinity of $v$ implies that  
\begin{equation}\label{eq:thm9-01}
\langle \nabla v(y^{k}), d^k \rangle = v(y^{k}) - v(x^k);
\end{equation}
Two points $x^k$ and $y^{k}$ in $\mathcal{C}$ imply that 
\begin{equation}\label{eq:thm9-03}
v(x^k)=v(y^{k}) = 0; ~ u(x^k)\leq 0; ~ u(y^{k})\leq 0.
\end{equation} 	
$\rhd$ By the affinity of $v$, for any $\lambda>0$, we have $$v(y^{k} + \lambda d^k) = v(y^{k}) + \lambda \langle \nabla v(y^{k}), d^k \rangle \overset{\eqref{eq:thm9-01}}{=} v(y^{k}) + \lambda (v(y^{k}) - v(x^{k})) \overset{\eqref{eq:thm9-03}}{=}  0.$$
Thus,
\begin{equation}
\label{eq:thm9-result1}
\boxed{\forall \lambda>0, v(y^{k} + \lambda d^k) = 0.}
\end{equation}
$\rhd$ By the affinity of $u$, for any $\lambda>0$, we have
\begin{equation}
\label{eq:thm9-04}
u(y^{k}+\lambda d^k) = u(y^{k}) + \lambda (u(y^{k}) - u(x^{k})).
\end{equation}
\underline{Case 1}: If $i\in A(y^{k})\subset A(x^{k})$, then $u_i(x^k)=u_i(y^{k})=0$, and we get from \eqref{eq:thm9-04} that 
$$\forall \lambda>0, u_i(y^{k}+\lambda d^k)  = 0.$$
\underline{Case 2}: If $i\notin A(y^{k})$ and $i\in A(x^k)$, then $u_i(y^{k})<0$ and $u_i(x^k)=0$, we get from \eqref{eq:thm9-04} that 
$$\forall \lambda>0, u_i(y^{k}+\lambda d^k)  = (1+\lambda) u_i(y^{k}) < 0.$$
\underline{Case 3}: If $i\notin A(y^{k})$ and $i\notin A(x^k)$, then $u_i(y^{k})<0$ and $u_i(x^k)<0$, we get from \eqref{eq:thm9-04} that for any sign of $u_i(y^{k})-u_i(x^k)$,  
$$\exists \bar{\lambda}>0, \forall \lambda \in (0,\bar{\lambda}),~ u_i(y^{k}+\lambda d^k) = u_i(y^{k}) + \lambda (u_i(y^{k}) - u_i(x^{k})) < 0.$$
We conclude from Cases 1-3 that 
\begin{equation}
\label{eq:thm9-result2}
\boxed{\exists \bar{\lambda}>0, \forall \lambda \in (0,\bar{\lambda}), ~ u(y^{k}+\lambda d^k) \leq 0.}
\end{equation}
It follows from \eqref{eq:thm9-result1} and \eqref{eq:thm9-result2} that $d$ is a feasible direction of $\mathcal{C}$ at $y^{k}$.

\section{\textbf{Proof of Theorem \ref{thm:decdirect-nondiff}}}\label{appendix:I_thm:decdirect-nondiff}
Let $g$ be differentiable convex, $h$ be non-differentiable convex, then
\begin{align}\label{eq:thm5-1}
f'(y^{k};d^k) =\langle \nabla g(y^{k}), d^k \rangle - h'(y^{k};d^k).
\end{align}
Since $h$ is non-differentiable and convex, then $$\partial h(y^{k}) = \{ z: h(y^{k} + s) \geq  h(y^{k}) + \langle z,s \rangle, \forall s \}.$$
Taking $s = td^k$ with $t> 0$, we have $\forall z^k\in \partial h(y^{k})$,
\begin{equation}
\label{eq:thm5-2}
\langle z^k , d^k \rangle \leq \frac{h(y^{k} + td^k) - h(y^{k})}{t}  \xrightarrow{t\to 0^+} h'(y^{k};d^k). 
\end{equation}
The equations \eqref{eq:thm5-1} and \eqref{eq:thm5-2} implies 
\begin{equation}\label{eq:thm5-3}
f'(y^{k};d^k) \leq \langle \nabla g(y^{k}) - z^k, d^k \rangle, \forall z^k\in \partial h(y^{k}).
\end{equation}
Now, DCA for problem \eqref{prob:CCDC} yields the next convex optimization subproblem
$$y^{k} \in \argmin \{ g(x)- \langle \xi^k ,x \rangle : x\in \C \},$$
where $\xi^k\in \partial h (x^k)$. The first-order optimality implies
$$0 \in \nabla g(y^{k}) - \xi^k +  N_{\mathcal{C}} (y^{k}).$$
Then,
$$ \xi^k - \nabla g(y^{k}) \in N_{\mathcal{C}} (y^{k}).$$
By the definition of the normal cone $N_{\mathcal{C}} (y^{k})$, we have 
$$\langle \xi^k - \nabla g(y^{k}) , y^{k} - x \rangle \geq 0, \forall x\in \C.$$
This inequality is obviously true for $x=x^k\in \C$, so that 
\begin{equation}
\label{eq:thm5-4}
\langle \xi^k - \nabla g(y^{k}) , y^{k} - x^k \rangle \geq 0.
\end{equation}
Combining equations \eqref{eq:thm5-3}, \eqref{eq:thm5-4} and $d^k = y^{k} - x^k$, we get

\begin{align}
\label{eq:thm5-5}
f'(y^{k};d^k) \overset{\eqref{eq:thm5-3}}{\leq} \langle \nabla g(y^{k}) - z^k, d^k \rangle  =\langle \nabla g(y^{k}) - \xi^k, d^k \rangle + \langle \xi^k - z^k, d^k \rangle \overset{\eqref{eq:thm5-4}}{\leq} \langle \xi^k - z^k, d^k \rangle.
\end{align}
That is,
$$\boxed{f'(y^{k};d^k)\leq \langle \xi^k - z^k, d^k \rangle, \xi^k\in \partial h(x^k), z^k\in \partial h(y^k).}$$
Now, we can derive the results for non-differentiable $h$ as follows:\\
\noindent $\rhd$ \underline{If $h$ is non-differentiable convex:} by the monotonicity of $\partial h$, we have for $\xi^k\in \partial h(x^k)$ and $z^k\in \partial h(y^k)$ that
\begin{equation}
\label{eq:thm5-6}
\langle \xi^k - z^k , d^k \rangle  = \langle \xi^k - z^k , y^{k} - x^k \rangle \leq 0.
\end{equation}
It follows from \eqref{eq:thm5-5} and \eqref{eq:thm5-6} that 
$$\boxed{f'(y^{k};d^k) \leq 0.}$$
\noindent $\rhd$ \underline{If $h$ is non-differentiable $\rho$-strongly convex:} by the strong monotonicity of $\partial h$, we have for $\xi^k\in \partial h(x^k)$ and $z^k\in \partial h(y^k)$ that 
\begin{equation}
\label{eq:thm5-7}
\langle \xi^k - z^k , d^k \rangle  = \langle \xi^k - z^k , y^{k} - x^k \rangle \leq -\rho \|d^k\|^2.
\end{equation}
It follows from \eqref{eq:thm5-5} and \eqref{eq:thm5-7} that 
$$\boxed{f'(y^{k};d^k) \leq -\rho \|d^k\|^2.}$$
\noindent $\rhd$ \underline{If $u_i$ is non-differentiable convex:} we can prove in a similar way as in Theorem \ref{prop:feasdirect} that $d^k$ is a feasible direction of $\C$ at $y^{k}$ with the replacement of $\langle \nabla u_i (y^{k}), d^k\rangle $ by $u_i'(y^{k};d^k)$. 

\section{\textbf{Proof of Theorem \ref{thm:convofBDCA}}}\label{appendix:J_thm:convofBDCA}
(i) For every $k=0,1,\ldots$, the first order optimality condition for the convex problem
$$y^{k}\in \argmin\{g(x) - \langle z^k , x \rangle ~|~ x\in \C\},$$
with $z^k\in \partial h(x^k)$ reads
$$	0\in \partial g(y^{k}) + N_{\C}(y^{k}) - z^k.$$
Thus 
\begin{equation}
	\label{eq:1stoptcondPk}
z^k \in\partial h(x^k) \cap (\partial g(y^{k}) + N_{\C}(y^{k})).
\end{equation}
By the $\rho_h$-convexity of $h$ and $z^k\in \partial h(x^k)$, then
\begin{equation}
	\label{eq:strongconvexityofh}
	h(y^{k})\geq h(x^{k}) + \langle y^{k}-x^{k}, z^k \rangle + \frac{\rho_h}{2}\|y^{k}-x^{k}\|^2.
\end{equation}
By the $\rho_g$-convexity of $g$ and $w^k\in \partial g(y^{k})$, we have 
\begin{equation}
	\label{eq:strongconvexityofg}
	g(x^{k})\geq g(y^{k}) + \langle x^{k}-y^{k}, w^k \rangle + \frac{\rho_g}{2}\|y^{k}-x^{k}\|^2. 
\end{equation}
Taking $v^k\in N_{\C}(y^{k})$ and setting $w^k=z^k-v^k$, then we get from \eqref{eq:strongconvexityofg} that 
\begin{equation}\label{eq:strongconvexityofg-bis}
	g(x^{k}) + \langle y^{k}-x^{k}, z^k-v^k \rangle - \frac{\rho_g}{2}\|y^{k}-x^{k}\|^2 \geq g(y^{k}).
\end{equation}
By the definition of the normal cone,
$$v^k\in N_{\C}(y^{k}) \implies \langle x-y^{k}, v^k \rangle \leq 0, \forall x\in \C.$$
Then  
\begin{equation}
	\label{eq:normconeineq}
	\langle x^k-y^{k}, v^k \rangle \leq 0, \forall k=1,2,\ldots.
\end{equation}
It follows that for all $k=1,2,\ldots$
\begin{eqnarray*}
	f(y^{k})&=&g(y^{k})-h(y^{k})\\
	&\overset{\eqref{eq:strongconvexityofh}}{\leq}&g(y^{k})-\left( h(x^k) + \langle y^{k} - x^k,z^k \rangle + \frac{\rho_h}{2}\|y^{k}-x^{k}\|^2 \right)\\
	&\overset{\eqref{eq:strongconvexityofg-bis}}{\leq}& g(x^k) - h(x^k) + \langle x^{k}-y^{k}, v^k \rangle - \frac{\rho_g+\rho_h}{2}\|y^{k}-x^{k}\|^2\\
	&\overset{\eqref{eq:normconeineq}}{\leq}& f(x^k) - \frac{\rho_g+\rho_h}{2}\|y^{k}-x^{k}\|^2,
\end{eqnarray*}
which leads to the required inequality
\begin{equation}\label{eq:descentlemma}
	\boxed{f(x^k) - f(y^k) \geq \frac{\rho_g+\rho_h}{2}\|y^{k}-x^{k}\|^2, \forall k=1,2,\ldots.}
\end{equation}
(ii) The point $x^{k+1}$ is updated either by $y^k$ or by the Armijo's rule. In the later case, we have  
\begin{equation}
	\label{eq:Armijo'sineq}
	f(y^k) - f(x^{k+1}) \geq -\sigma \alpha f'(y^k;d^k) \geq 0.
\end{equation}
It follows from \eqref{eq:descentlemma}, \eqref{eq:Armijo'sineq} and the lower boundedness of $f$ over $\C$ that for every $k=1,2,\ldots$,  
$$-\infty < \min \{f(x): x\in \C\} \leq f(x^{k+1})\leq f(y^{k})\leq f(x^{k}).$$
Therefore, the sequence $\{f(x^k)\}_{k\geq 1}$ is non-increasing and bounded from below, thus convergent. \\
(iii) By the convergence of $\{f(x^k)\}_{k\geq 1}$ and the relation $f(x^{k+1})\leq f(y^k)\leq f(x^{k})$ for every $k=1,2,\ldots,$ we obtain that the sequence $\{f(y^k)\}_{k\geq 1}$ converges to the same limit of the sequence $\{f(x^k)\}_{k\geq 1}$. Then, taking the limit of \eqref{eq:descentlemma} for $k\to \infty$, we obtain that 
\begin{equation}
	\label{eq:convofyk-xk}
	\boxed{\|y^k-x^k\|\xrightarrow{k\to \infty}0.}
\end{equation}
The iteration point $x^{k+1}$ is computed by the line search formulation 
\begin{equation}
	\label{eq:Armijoupdate}
	x^{k+1} = y^k + \alpha_k (y^k-x^k), \forall k=1,2,\ldots,
\end{equation}
where $\alpha_k\in(0,\bar{\alpha}]$ ($\bar{\alpha}$ is an upper bound for $\alpha_k$, e.g., $\bar{\alpha}=1$) if Armijo's rule is applied and $\alpha_k=0$ otherwise. Then we get from \eqref{eq:Armijoupdate} and \eqref{eq:convofyk-xk} that 
$$\|x^{k+1}-x^k\| \overset{\eqref{eq:Armijoupdate}}{=} \|y^k + \alpha_k (y^k-x^k) - x^k\| = (1+\alpha_k)\| y^k-x^k\|\leq (1+\bar{\alpha})\| y^k-x^k\|\xrightarrow[\eqref{eq:convofyk-xk}]{k\to \infty} 0.$$
That is $$\boxed{\|x^{k+1}-x^k\| \xrightarrow{k\to \infty}0.}$$ 
(iv) The equation \eqref{eq:descentlemma} and $f(y^k)\geq f(x^{k+1})$ for every $k=1,2,\ldots$ imply that 
\begin{equation}
	\label{eq:descentlemma-bis}
	\|y^{k}-x^{k}\|^2 \leq \frac{2}{\rho_g+\rho_h}(f(x^k) - f(x^{k+1})), \forall k=1,2,\ldots.
\end{equation}
Summing \eqref{eq:descentlemma-bis} for $k$ from $1$ to $\infty$, we get 
$$\sum_{k=1}^{\infty}\|y^{k}-x^{k}\|^2 \leq \frac{2}{\rho_g+\rho_h} (f(x^1) - \lim_{k\to\infty} f(x^{k+1})) < \infty.$$
That is
\begin{equation}
	\label{eq:summableyk-xksquare}
	\boxed{\sum_{k=0}^{\infty}\|y^{k}-x^{k}\|^2<\infty.}
\end{equation}
Similarly, replacing $x^{k+1}$ by \eqref{eq:Armijoupdate}, we get 
$$\sum_{k=1}^{\infty}\|x^{k+1}-x^{k}\|^2 \overset{\eqref{eq:Armijoupdate}}{=} \sum_{k=1}^{\infty}\|(1+\alpha_k)(y^k-x^{k})\|^2\leq (1+\bar{\alpha})^2\sum_{k=1}^{\infty}\|y^k-x^{k}\|^2\overset{\eqref{eq:summableyk-xksquare}}{<}\infty.$$
Hence,
$$\boxed{\sum_{k=0}^{\infty}\|x^{k+1}-x^{k}\|^2 < \infty.}$$
(v) For every $k=1,2,\ldots$, the first order optimality condition given in \eqref{eq:1stoptcondPk} reads:
$$z^k \in (\partial g(y^k) + \partial \chi_{\mathcal{C}}(y^k)) \cap \partial h(x^k).$$
Then, by the boundedness of the sequence $\{x^k\}$, there exist a convergent subsequence denoted by $\{x^{k_j}\}_{j\in \N}\subset \C$ and its limit denoted by $x^*\in \C$. We get from $\|y^k-x^k\|\xrightarrow{k\to \infty} 0$ that the subsequence $\{y^{k_j}\}$ converges to $x^*$ as well. Then it follows by the closedness of the graphs of $\partial g$, $\partial h$ and $\partial \chi_{\mathcal{C}}$ that any limit point of the subsequence $\{z^{k_j}\}$ is included in $(\partial g(x^*) + \partial \chi_{\mathcal{C}}(x^*)) \cap \partial h(x^*)$. The boundedness of the sequence $\{z^k\}$ implies that the set of limit points of the subsequence $\{z^{k_j}\}$ is non-empty. Hence
$$\boxed{(\partial g(x^*) + \partial \chi_{\mathcal{C}}(x^*)) \cap \partial h(x^*)\neq \emptyset.}$$
That is, $x^*$ is a DC critical point of \eqref{prob:CCDC}. It follows from Theorem \ref{thm:convtod-stationaryforsmoothdcp} that if $h$ is continuously differentiable, then $x^*$ is a strongly DC critical point of \eqref{prob:CCDC}.\\
(vi) The convergence of $\sum_{k\geq 0}\|y^k-x^k\|$ for the Boosted-DCA is proved in a similar way as the convergence of $\sum_{k\geq 0}\|x^{k+1}-x^k\|$ in Theorem \ref{thm:convUDCA} using the {\L}ojasiewicz subgradient inequality and the assumption that $h$ has locally Lipschicz continuous gradient over $\C$ (i.e., $h$ is continuously differentiable over $\C$ as well). \\
Now, based on the convergence of $\sum_{k\geq 0}\|y^k-x^k\|$, we can also establish the convergence of $\sum_{k\geq 0}\|x^{k+1}-x^k\|$ for the Boosted-DCA as:
$$\sum_{k\geq 1}\|x^{k+1}-x^k\| = \sum_{k\geq 1}\|(1+\alpha_k)(y^k-x^k)\|\leq (1+\bar{\alpha})\sum_{k\geq 1}\|y^k-x^k\|< \infty.$$
Therefore, $\{x^k\}$ is a Cauchy sequence, and thus convergent.\\
(vii) The convergence of the sequence $\{x^k\}$ to a strongly DC critical point of \eqref{prob:CCDC} is an immediate consequence of (v) and (vi) in the case where $h$ is continuously differentiable over $\C$.
\end{APPENDICES}


\bibliographystyle{informs2014} 
\bibliography{bibfile} 


\end{document}